\documentclass[12pt,a4paper]{amsart}

\usepackage{amsmath, amsthm, amssymb, paralist, xspace, graphicx, url, amscd, euscript, mathrsfs, epic, eepic,color, verbatim}

\usepackage[shortlabels]{enumitem}

\usepackage{leftidx}

\usepackage[margin=1.17in]{geometry}

\usepackage[normalem]{ulem}

\usepackage[colorlinks=true]{hyperref}

\usepackage[all]{xy}

\newcommand{\oW}{\overline{W}} 
\newcommand{\tW}{\widetilde{W}} 

\newcommand{\tcW}{\widetilde{\mathcal{W}}} 

\newcommand{\tomega}{\widetilde{\omega}} 

\newcommand{\crit}{\operatorname{Crit}} 

\DeclareMathOperator{\tred}{tred} 

\newcommand{\cI}{\mathcal{I}_{\mu}} 
\newcommand{\ocI}{\overline{\mathcal{I}}_{\mu}} 
\newcommand{\fev}{\mathfrak{ev}} 
\DeclareMathOperator{\ev}{ev} 

\newcommand{\BG}{\mathbf{BG}}
\newcommand{\BC}{\mathbf{BC}^*_\omega}
\newcommand{\LR}{\mathcal L_\omega}

\newcommand{\bL}{\mathbf{L}}
\newcommand{\bE}{\mathbf{E}}
\newcommand{\ul}{\underline}
\newcommand{\ol}{\overline}

\newcommand{\etwist}{a}
\newcommand{\ttwist}{\tilde{r}}

\newcommand{\ddata}{\vec{\varsigma}}

\newcommand{\cH}{\mathcal{H}} 

\newcommand{\oM}{\overline M} 

\newcommand{\ocM}{\overline{\mathcal{M}}} 

\newcommand{\NN}{\mathbb{N}} 
\newcommand{\ZZ}{\mathbb{Z}} 
\newcommand{\CC}{\mathbb C}
\newcommand{\bk}{\mathbf{k}}

\newcommand{\fM}{\mathfrak{M}} 
\newcommand{\fH}{\mathfrak{H}} 

\newcommand{\fS}{\mathfrak{S}} 
\newcommand{\fU}{\mathfrak{U}} 

\newcommand{\fMtw}{\mathfrak{M}^{\mathrm{tw}}} 

\newcommand{\IU}{\mathring{\fU}} 
\newcommand{\IR}{\mathring{\scrR}} 

\newcommand{\tUH}{\widetilde{\UH}}
\newcommand{\tU}{\widetilde{\fU}}

\newcommand{\fY}{\mathfrak{Y}} 
\newcommand{\scrY}{\mathscr{Y}} 

\newcommand{\fC}{\mathfrak{C}} 
\newcommand{\fE}{\mathfrak{E}} 
\newcommand{\fX}{\mathfrak{X}}

\newcommand{\fP}{\mathfrak{P}}

\newcommand{\zero}{\mathbf{0}}

\newcommand{\uC}{\underline{\mathcal{C}}} 

\newcommand{\ucC}{\underline{C}} 

\newcommand{\uS}{\underline{S}} 

\newcommand{\iiota}{\mathring{\iota}} 

\newcommand{\cL}{\mathcal L}
\newcommand{\cM}{\mathcal M} 
\newcommand{\cU}{\mathcal{U}} %
\newcommand{\cP}{\mathcal P} 
\newcommand{\Sp}{\mathcal{L}_{\fX}} 

\newcommand{\cA}{\mathcal{A}} 

\newcommand{\cX}{\mathcal{X}} 

\newcommand{\cC}{\mathcal{C}} 

\newcommand{\cZ}{\mathcal{Z}}

\newcommand\scrM{\mathscr{M}} 

\newcommand{\scrR}{\mathscr{R}}

\newcommand{\SH}{\mathscr{R}}
\newcommand{\SR}{\mathscr{R}} 

\newcommand{\UH}{\mathscr{U}} 

\newcommand{\TT}{\mathbb T} 
\newcommand{\EE}{\mathbb{E}} 
\newcommand{\LL}{\mathbb L} 

\newcommand{\FF}{\mathbb{F}} 

\newcommand{\obs}{\operatorname{Obs}} 

\newcommand{\cok}{\operatorname{cok}} 

\newcommand{\bs}{\mathbf{s}} 

\newcommand{\fm}{\mathfrak{m}} 

\newcommand{\fc}{\mathfrak{c}} 

\newcommand{\fp}{\mathfrak{p}} 

\newcommand{\ft}{\mathfrak{t}} 

\newcommand{\fr}{\mathfrak{r}} 

\newcommand{\cO}{\mathcal O} 

\newcommand{\cT}{\mathcal T} 

\newcommand{\PP}{\mathbb{P}}

\newcommand{\A}{\mathbb{A}}  

\newcommand{\GG}{\mathbb{G}}
\newcommand{\Gm}{\mathbb{G}_{m}}

\newcommand{\vir}{\mathrm{vir}}
\newcommand{\red}{\mathrm{red}}
\newcommand{\reg}{\mathrm{reg}}
\newcommand{\cpt}{\mathrm{cpt}}

\newcommand{\bw}{\mathbf w}

\DeclareMathOperator{\Aut}{Aut}

\DeclareMathOperator{\spec}{Spec}

\DeclareMathOperator{\diff}{d}

\DeclareMathOperator{\rk}{rk}
\DeclareMathOperator{\age}{age}
\DeclareMathOperator{\vb}{Vb}

\newcommand{\poleq}{\preccurlyeq} 

\newtheorem{proposition}{Proposition}[section]
\newtheorem{corollary}[proposition]{Corollary}
\newtheorem{lemma}[proposition]{Lemma}

\newtheorem{theorem}[proposition]{Theorem}
\theoremstyle{definition}
\newtheorem{assumption}[proposition]{Assumption}
\newtheorem{definition}[proposition]{Definition}

\newtheorem{example}[proposition]{Example}

\theoremstyle{remark}
\newtheorem{remark}[proposition]{Remark}

\numberwithin{equation}{section}

\begin{document}

\title{The logarithmic gauged linear sigma model}
\author{Qile Chen \and Felix Janda \and Yongbin Ruan}

\address[Q. Chen]{Department of Mathematics\\
Boston College\\
Chestnut Hill, MA 02467\\
U.S.A.}
\email{qile.chen@bc.edu}

\address[F. Janda]{Department of Mathematics\\
University of Notre Dame\\
Notre Dame, IN 46656\\
U.S.A.}
\email{fjanda@nd.edu}

\address[Y. Ruan]{Institute for Advanced Study in Mathematics\\
Zhejiang University\\
Hangzhou, China}
\email{ruanyb@zju.edu.cn}

\date{\today}

\begin{abstract}
  We introduce the notion of log $R$-maps, and develop a proper moduli
  stack of stable log $R$-maps in the case of a hybrid gauged linear
  sigma model.
  Two virtual cycles (canonical and reduced) are constructed for these
  moduli stacks.
  The main results are two comparison theorems relating the reduced
  virtual cycle to the cosection localized virtual cycle, as well as
  the reduced virtual cycle to the canonical virtual cycle.
  This sets the foundation of a new technique for computing higher
  genus Gromov--Witten invariants of complete intersections.
\end{abstract}



\maketitle

\setcounter{tocdepth}{1}
\tableofcontents

\section{Introduction}

\subsection{The gauged linear sigma model}
In 1993, Witten gave a physical derivation of the Landau--Ginzburg
(LG)/Calabi--Yau (CY) correspondence by constructing a family of
theories, known as the \emph{gauged linear sigma model} or GLSM
\cite{Wi93}.
A mathematical realization of the LG model, called the
Fan--Jarvis--Ruan--Witten (FJRW) theory has been established in
\cite{FJR13} via topological and analytical methods.
On the algebraic side, an approach using the cosection localization
\cite{KiLi13} to construct the GLSM virtual cycle was discovered in
\cite{ChLi12, CLL15} in the narrow case, and in general in
\cite{CKL18, KL18}.
Along the cosection approach, some hybrid models were studied in
\cite{Cl17}, and a general algebraic theory of GLSM for
(\emph{compact-type sectors} of) GIT targets was put on a firm
mathematical footing by Fan, Jarvis and the third author \cite{FJR18}.
A further algebraic approach for broad sectors using matrix
factorizations has been developed in \cite{PoVa16, CFGKS18P}, while an
analytic approach has been developed in \cite{FJR16}, \cite{TX18}.

As discovered in \cite{ChLi12} and further developed in \cite{KiOh18P,
  ChLi18P, CJW19P}, GLSM can be viewed as a deep generalization of the
hyperplane property of Gromov--Witten (GW) theory for arbitrary
genus.
However, comparing to GW theory, a major difference as well as a main
difficulty of GLSM is the appearance of an extra torus action on the
target, called the \emph{R-charge}, which makes the moduli stacks in
consideration for defining the GLSM virtual cycles non-proper in
general.
This makes the powerful tool of virtual localization \cite{GrPa99}
difficult to apply.

This is the second paper of our project aiming at a logarithmic GLSM
theory that solves the non-properness issue and provides a
localization formula by combining the cosection approach with the
logarithmic maps of Abramovich--Chen--Gross--Siebert \cite{AbCh14,
  Ch14, GrSi13}.
This leads to a very powerful technique for computing higher genus
GW/FJRW-invariants of complete intersections in GIT quotients.
Applications include computing higher genus invariants of quintic
$3$-folds \cite{GJR17P, GJR18P}\footnote{A
  different approach to higher genus Gromov--Witten invariants of
  quintic threefolds has been developed by Chang--Guo--Li--Li--Liu
  \cite{CGL18Pb, CGL18Pa, CGLL18P, CLLL15P, CLLL16P}.}, and the cycle
of holomorphic differentials \cite[Conjecture~A.1]{PPZ16P} by
establishing a localization formula of $r$-spin cycles conjectured by
the second author \cite{CJRSZ19P}.
This conjectural localization formula was the original motivation of
this project.

In our first paper \cite{CJRS18P}, we developed a \emph{principalization} of the boundary of the moduli of log maps, which provides a natural framework for extending cosections to the boundary of the logarithmic compactification. The simple but important $r$-spin case has been studied in \cite{CJRS18P} via the log compactification for maximal explicitness.

The goal of the current paper is to further establish a log GLSM
theory in the hybrid model case which allows possibly \emph{non-GIT
  quotient} targets.
As our theory naturally carries two different perfect obstruction
theories, we further prove explicit relations among these virtual
cycles.
This will provide a solid foundation for our forthcoming paper
\cite{CJR20P2} where various virtual cycles involved in log GLSM will
be further decomposed using torus localizations and the developments
in \cite{ACGS20P,ACGS17}.

In the case of GIT quotient targets, another aspect of GLSM moduli
spaces is that they depend on a stability parameter and exhibit a
rich wall-crossing phenomenum.
To include general targets, the current paper concerns only the $\infty$-stability that
is closely related to stable maps. We leave the study of other
stability conditions in the case of GIT quotient targets to a future research.

\subsection{$R$-maps}

The following fundamental notion of $R$-maps is the result of our
effort to generalize pre-stable maps to the setting of GLSM with
possibly non-GIT quotient targets.
While the definition makes essential use of stacks, it is what makes
various constructions in this paper transparent.

\begin{definition}\label{def:R-map}
  Let $\fP \to \BC$ be a proper, DM-type morphism of log stacks where
  $\BC := \BG_m$ is the stack parameterizing line bundles with the
  trivial log structure.
  A \emph{logarithmic $R$-map} (or, for short, log $R$-map) over a log
  scheme $S$ with target $\fP \to \BC$ is a commutative diagram:
\[
\xymatrix{
 && \fP \ar[d] \\
\cC \ar[rru]^{f} \ar[rr]_{\omega^{\log}_{\cC/S}}&& \BC
}
\]
where $\cC \to S$ is a log curve (see \S\ref{sss:log-curves}) and the
bottom arrow is induced by the log cotangent bundle
$\omega^{\log}_{\cC/S}$.
The notation $\BC$ is reserved for parameterizing the line bundle
$\omega^{\log}_{\cC/S}$ of the source curve.
\emph{Pull-backs} of log $R$-maps are defined as usual via pull-backs
of log curves.

For simplicity, we will call $f\colon \cC \to \fP$ a log $R$-map without specifying arrows to $\BC$. Such $f$ is called an \emph{$R$-map} if it factors through the open substack $\fP^{\circ}\subset \fP$ with the trivial log structure.

A pre-stable map $\ul{f}\colon \ul{\cC} \to \ul{\fP}$ over $\ul{S}$ with
compatible arrows to $\BC$ is called an \emph{underlying $R$-map}.
Here $\ul{\fP}$ is the underlying stack obtained by removing the log
structure of $\fP$.
\end{definition}

\begin{remark}
  Our notion of $R$-maps originates from the $R$-charge in physics.
  In the physical formulation of GLSM \cite{Wi93}, there is a target
  space $X$ which is a K\"ahler manifold (usually a GIT quotient of a
  vector space) and a superpotential $W \colon X \to \CC$.
  To define the A-topological twist, one needs to choose a
  $\CC^*$-action on $X$, called the R-charge, such that $W$ has $R$-charge
  (weight) $2$.
  The weights of the R-charge on the coordinates of $X$ are used to
  twist the map or the fields of the theory \cite{Wi93,FJR18,GS09}.
  As pointed out by one of the referees, the setting of this article
  is more general than GLSM in physics in the sense that $X$ does not
  necessarily have an explicit coordinate description.
  For this purpose, we formulate the more abstract notion of $R$-maps
  as above.
  Our notion of $R$-maps agrees with those of \cite{Wi93,FJR18,GS09}
  when $X$ is a GIT quotient of a vector space.
  The moduli space of $R$-maps should give a mathematical
  description of the spaces on which the general A-twist localizes in
  physics.
\end{remark}

\begin{example}[quintic threefolds]
  Consider the target $\fP^\circ = [\vb_{\PP^4}(\cO(-5))/\CC^*_{\omega}]$,
  where $\CC^*_{\omega} \cong \Gm$ acts on the line bundle $\vb_{\PP^4}(\cO(-5))$ by
  scaling the fibers with weight one.
  The map $\fP^{\circ} \to \BC$ is the canonical map from the quotient
  description of $\fP^{\circ}$.
  In this case, an $R$-map $f\colon \cC \to \fP^{\circ}$ is equivalent to the
  data of a map $g\colon \cC \to \PP^4$ together with a section (or
  ``$p$-field'') $p \in H^0(\omega_C^{\log} \otimes f^* \cO(-5))$.
  Therefore, if $\cC$ is unmarked, we recover the moduli space of
  stable maps with $p$-fields \cite{ChLi12}, which is the GLSM moduli
  space \cite{FJR18} for a quintic hypersurface in $\PP^4$. The construction of this paper will provide a compactification of $\fP^{\circ}$ relative to $\BC$, and a compactification of the moduli of $p$-fields.
  We refer the reader to Section~\ref{sec:examples} for more examples in a general situation.
\end{example}

Just like in Gromov--Witten theory, various assumptions on $\fP$ are
needed to build a proper moduli space as well as a virtual cycle.
While a theory of stable log $R$-maps for general $\fP$ seems to
require much further development using the full machinery of
logarithmic maps, we choose to restrict ourselves to the so called
\emph{hybrid targets} which already cover a large class of interesting
examples including both FJRW theory and complete intersections in
Gromov--Witten theory.
We leave the general case to a future research.

\subsection{$R$-maps to hybrid targets}
\label{ss:target-data}

\subsubsection{The input}\label{sss:input}
A hybrid target is determined by the following data:
\begin{enumerate}
 \item A proper Deligne--Mumford stack $\cX$ with a projective coarse moduli scheme $X$.
 \item A vector bundle $\bE$ over $\cX$ of the form
\begin{equation*}
  \bE = \bigoplus_{i \in \ZZ_{> 0}} \bE_i
\end{equation*}
where $\bE_i$ is a vector bundle with the positive grading $i$. Write $d := \gcd\big(i \ | \ \bE_i \neq 0 \big)$.
 \item A line bundle $\bL$ over $\cX$.
 \item A positive integer $r$.
\end{enumerate}

For later use, fix an ample line bundle $H$ over $X$, and denote by $\cH$ its pull-back over $\cX$.

\subsubsection{The $r$-spin structure} The R-charge leads to the universal  $r$-spin structure as follows.
Consider the  cartesian diagram
\begin{equation}\label{diag:spin-target}
\xymatrix{
\fX \ar[rrr]^\Sp \ar[d] &&& \BG_m \ar[d]^{\nu_r} \\
\BC\times\cX \ar[rrr]^{\LR\boxtimes\bL^{\vee}} &&& \BG_m
}
\end{equation}
where $\LR$ is the universal line bundle over $\BC$,
$\nu_r$ is the $r$th power map, the bottom arrow is defined by $\LR\boxtimes\bL^{\vee}$, and the top arrow is
defined by the universal $r$-th root of
$\LR\boxtimes\bL^{\vee}$, denoted by $\Sp$.

\subsubsection{The targets}
Fix a \emph{twisting choice} $\etwist \in \frac{1}{d}\cdot\ZZ_{>0}$, and
set $\ttwist = \etwist \cdot r$.
We form the weighted projective stack bundle over $\fX$:
\begin{equation}\label{equ:universal-proj}
\ul{\fP} := \ul{\PP}^{\bw}\left(\bigoplus_{i > 0}(\bE^{\vee}_{i,\fX}\otimes\cL_{\fX}^{\otimes i})\oplus \cO_{\fX} \right),
\end{equation}
where $\bw$ is the collection of the weights of the $\GG_m$-action such that the weight on the $i$-th factor is the positive integer
$\etwist\cdot i$, while the weight of the last factor $\cO$ is 1.
Here for any vector bundle $V = \oplus_{i=1}^{r} V_i$ with a $\GG_m$-action of weight $\bw$, we use the notation
\begin{equation}\label{equ:def-proj-bundle}
\PP^\bw(V) = \left[\Big(\vb(V)\setminus \zero_V \Big) \Big/ \GG_m \right],
\end{equation}
where $\vb(V)$ is the total space of $V$, and $\zero_V$ is the zero section of $V$. Intuitively, $\ul\fP$ compactifies the GLSM given by
\[
\fP^{\circ} := \vb\left(\bigoplus_{i > 0}(\bE^{\vee}_{i,\fX}\otimes\cL_{\fX}^{\otimes i})\right).
\]

The boundary $\infty_{\fP} = \ul{\fP} \setminus \fP^{\circ}$ is the Cartier divisor
defined by the vanishing of the coordinate corresponding to
$\cO_{\fX}$.
We make $\ul\fP$ into a log stack $\fP$ by equipping it with the log structure corresponding to the Cartier divisor $\infty_\fP$.  Denote by $\zero_\fP$ the zero section of the vector bundle $\fP^{\circ}$.
We arrive at the following commutative diagram
\begin{equation}\label{equ:hyb-target}
\xymatrix{
\fP \ar[r]^{\fp} \ar[rd]_{\ft}& \fX \ar[r]^{\zeta} \ar[d] & \BC \\
&\cX&
}
\end{equation}
where $\zeta$ is the composition $\fX \to \BC \times \cX \to \BC$ with the left arrow the projection to $\BC$. By construction, $\zeta\circ\fp$ is proper of DM-type.

The general notion of log R-maps formulated using $\fP$ can be
described more concretely in terms maps with log fields, see
\S~\ref{ss:log-fields}.

\subsubsection{The stability}\label{sss:stability}

A log $R$-map $f\colon \cC \to \fP$ over $S$ is \emph{stable} if $f$ is
representable, and if for a sufficiently small $\delta_0 \in (0,1)$ there exists $k_0 > 1$ such that for any pair $(k,\delta)$ satisfying $k > k_0$ and $\delta_0 > \delta > 0$, the following holds
\begin{equation}\label{equ:hyb-stability}
(\omega_{\cC/S}^{\log})^{1 + \delta} \otimes (\ft \circ f)^* \cH^{\otimes k} \otimes f^* \cO(\ttwist\infty_{\fP}) > 0.
\end{equation}
The notation $>$ in \eqref{equ:hyb-stability} means that the left hand side has strictly positive degree when restricting to each irreducible component of the source curve.

\begin{remark}
  It has been shown that the stack of pre-stable log maps are proper
  over the stack of usual pre-stable maps \cite{AbCh14, Ch14, GrSi13}.
  Even given this, establishing a proper moduli stack remains a rather
  difficult and technical step in developing our theory.
  An evidence is that the moduli of underlying $R$-maps fails to be
  universally closed \cite[Section~4.4.6]{CJRS18P} in even most basic
  cases.
  The log structure of $\fP$ plays an important role in the properness
  as evidenced by the subtle stability \eqref{equ:hyb-stability} which
  was found after many failed attempts.
\end{remark}

\begin{remark}
  In case of rank one $\bE$, the stability \eqref{equ:hyb-stability}
  is equivalent to a similar formulation as in
  \cite[Definition~4.9]{CJRS18P} using $\zero_{\fP}$, see
  Remark~\ref{rem:r-spin-stability}.
  However, the latter does not generalize to the higher rank case,
  especially when $\bE$ is non-splitting.
  Consequently, we have to look for a stability condition of very
  different form, and a very different strategy for the proof of
  properness compared to the intuitive proof in \cite{CJRS18P}, see
  Section~\ref{sec:properties} for details. 
\end{remark}

\subsubsection{The moduli stack}
Denote by $\SR_{g, \ddata}(\fP,\beta)$ the category of stable log $R$-maps fibered over the category of log schemes with fixed discrete data $(g, \ddata, \beta)$ such that
\begin{enumerate}
 \item $g$ is the genus of the source curve.
 \item The composition of the log $R$-map with $\ft$ has curve class $\beta \in H_2(\cX)$.
 \item $\ddata = \{(\gamma_i, c_i)\}_{i=1}^n$ is a collection of pairs
   such that $c_i$ is the contact order of the $i$-th marking with
   $\infty_{\fP}$, and $\gamma_i$ is a component of the inertia stack
   fixing the local monodromy at the $i$-th (orbifold) marking
   (Definition~\ref{def:hyb-sector}).
\end{enumerate}

The first main result of the current article is the compactification:
\begin{theorem}[Theorem~\ref{thm:representability}]
  \label{thm:intro-representability}
  The category $\SR_{g, \ddata}(\fP,\beta)$ is
  represented by a proper logarithmic Deligne--Mumford stack.
\end{theorem}

\begin{remark}
  Different choices of data in Section \ref{sss:input} may lead to the
  same $\fP$, hence the same $\SR_{g, \ddata}(\fP,\beta)$.
  The ambiguity in our set-up is analogous to the non-unique choice
  of R-charge of general GLSM \cite[Section 3.2.3]{FJR18}.
\end{remark}

\subsection{Virtual cycles}
Another goal of this paper is to construct various virtual cycles of (log) $R$-maps. For this purpose, we now impose the condition that $\cX$ is smooth.

\subsubsection{The canonical virtual cycles}
Olsson's logarithmic cotangent complex \cite{LogCot} provides a \emph{canonical perfect obstruction theory} for $\SR_{g, \ddata}(\fP,\beta)$, see Section \ref{ss:canonical-obs}. If $c_i = 0$ for all $i$, we refer it as a \emph{holomorphic theory}. Otherwise, we call it a \emph{meromorphic theory}.  For our purposes, we are in particularly interested in the holomorphic theory and the closed substack
\[
\SR^{\cpt}_{g, \ddata}(\fP,\beta) \subset \SR_{g, \ddata}(\fP,\beta)
\]
where log $R$-maps factor through $\zero_{\fP}$ along all marked
points.
We call $\SR^{\cpt}_{g, \ddata}(\fP,\beta)$ the \emph{stack of log
  R-maps with compact-type evaluations}.
In this case, $\ddata$ is simply a collection of connected components
of the inertia stack $\ocI\zero_{\fP_{\bk}} $ of $\zero_{\fP_{\bk}} := \zero_{\fP}\times_{\BC}\spec \bk$ as all contact orders are zero.

The canonical perfect obstruction theory of $\SR_{g, \ddata}(\fP,\beta)$ induces a canonical perfect obstruction theory of $\SR^{\cpt}_{g, \ddata}(\fP,\beta)$, see \eqref{equ:obs-compact-evaluation}, hence defines the \emph{canonical virtual cycle}  $[\SR^{\cpt}_{g, \ddata}(\fP,\beta)]^{\vir}$.

\subsubsection{Superpotentials and cosection localized virtual cycles}
A \emph{superpotential} is a morphism of stacks $W \colon \fP^{\circ} \to \LR$ over $\BC$. Its {\em critical locus} $\crit(W)$ is the closed substack of $\fP^{\circ}$ where $\diff W: T_{\fP/\BC} \to W^*T_{\LR/\BC}$ degenerates. We will consider the case that $\crit(W)$ is proper over $\BC$.

This $W$ induces a canonical Kiem--Li cosection of the canonical obstruction of the open sub-stack $\SR^{\cpt}_{g, \ddata}(\fP^{\circ},\beta) \subset \SR_{g, \ddata}(\fP,\beta)$.  This leads to a \emph{cosection localized virtual cycle} $[\SR^{\cpt}_{g, \ddata}(\fP^{\circ},\beta)]_{\sigma}$ which represents $[\SR^{\cpt}_{g, \ddata}(\fP^{\circ},\beta)]^{\vir}$, and is supported on the proper substack
\begin{equation}\label{equ:red-cycle-support}
\IR_W \subset \SR^{\cpt}_{g, \ddata}(\fP^{\circ},\beta)
\end{equation} parameterizing $R$-maps to the critical locus of $W$, see Section \ref{sss:cosection-localized-class}.

The virtual cycle $[\SR^{\cpt}_{g, \ddata}(\fP^{\circ},\beta)]_{\sigma}$ is the \emph{GLSM virtual cycle} that we next recover as a virtual cycle over a {\em proper} moduli stack.

\subsubsection{The reduced virtual cycles}

In general, the canonical cosection over
$\SR^{\cpt}_{g, \ddata}(\fP^{\circ},\beta)$ does not have a nice
extension to $\SR^{\cpt}_{g, \ddata}(\fP,\beta)$.
The key to this is a proper morphism constructed in \cite{CJRS18P},
called a \emph{modular principalization}:
\[
F\colon \UH^{\cpt}_{g, \ddata}(\fP,\beta) \to \SR^{\cpt}_{g, \ddata}(\fP,\beta)
\]
where $\UH^{\cpt}_{g, \ddata}(\fP,\beta)$ is the moduli of stable log $R$-maps with {\em uniform maximal degeneracy}. Note that $F$ restricts to the identity on the common open substack $\SR^{\cpt}_{g, \ddata}(\fP^{\circ},\beta)$ of both its source and target.

The canonical perfect obstruction theory of $\SR^{\cpt}_{g, \ddata}(\fP,\beta)$ pulls back to a canonical perfect obstruction theory of $\UH^{\cpt}_{g, \ddata}(\fP,\beta)$, hence the canonical virtual cycle $[\UH^{\cpt}_{g, \ddata}(\fP,\beta)]^{\vir}$. Though $F$ does not change the virtual cycles in that $F_*[\UH^{\cpt}_{g, \ddata}(\fP^{},\beta)]^{\vir} = [\SR^{\cpt}_{g, \ddata}(\fP^{},\beta)]^{\vir}$, the cosection over $\SR^{\cpt}_{g, \ddata}(\fP^{\circ},\beta)$ extends to the boundary
\begin{equation}\label{equ:boundary}
\Delta_{\UH} := \UH^{\cpt}_{g, \ddata}(\fP,\beta) \setminus \SR^{\cpt}_{g, \ddata}(\fP^{\circ},\beta)
\end{equation}
with explicit poles \eqref{equ:canonical-cosection}. Then a general
machinery developed in Section~\ref{sec:POT-reduction}, produces a
\emph{reduced perfect obstruction theory} of
$\UH^{\cpt}_{g, \ddata}(\fP,\beta)$, hence the \emph{reduced virtual
  cycle} $[\UH^{\cpt}_{g, \ddata}(\fP,\beta)]^{\red}$, see
Section~\ref{sss:reduced-theory}.

\begin{remark}
The two virtual cycles $[\UH^{\cpt}_{g, \ddata}(\fP,\beta)]^{\vir}$ and $[\UH^{\cpt}_{g, \ddata}(\fP,\beta)]^{\red}$ have the same virtual dimension.
\end{remark}

\subsection{Comparing virtual cycles}

\subsubsection{Reduced versus cosection localized cycle}

We first show that log GLSM recovers GLSM:
\begin{theorem}[First comparison theorem \ref{thm:reduced=local}]
  Let $\iota\colon \IR_W \to \UH_{g, \ddata}(\fP, \beta)$ be the inclusion \eqref{equ:red-cycle-support}. Then we have
$$\iota_*[\SR^{\cpt}_{g, \ddata}(\fP^{\circ},\beta)]_{\sigma}  = [\UH^{\cpt}_{g, \ddata}(\fP, \beta)]^{\red}.$$
\end{theorem}

In Section~\ref{sec:examples}, we study a few examples explicitly. By the
first comparison theorem, the reduced virtual cycle of the compact moduli space of stable log $R$-maps recovers FJRW-theory and Clader's hybrid model when they are constructed using cosection localized virtual cycles \cite{CLL15, Cl17}.

Our machinery also applies to the Gromov--Witten theory of a complete
intersection, or more generally the zero locus $\cZ$ of a non-degenerate section $s$ of
a vector bundle $\bE$, Section \ref{ss:examples-GW}.
Examples include the quintic threefolds in $\PP^4$, and Weierstrass
elliptic fibrations, which are hypersurfaces in a $\PP^2$-bundle over
a not necessarily toric base $B$.
In this case, we may chose $r=1$ and 
$\ul{\fP} = \PP(\bE^\vee \otimes \LR \oplus \cO)$.
Combining with the results in \cite{ChLi12, KiOh18P, ChLi18P}, and
more generally in \cite{CJW19P, Pi20P}, we have 

\begin{corollary}[Proposition \ref{prop:glsm-gw}]
  Notations as above, we have
  \begin{equation*}
    p_*[\UH^{\cpt}_{g, \ddata}(\fP, \beta)]^{\red}
    = (-1)^{\rk(\bE)(1 - g) + \int_\beta c_1(\bE) - \sum_{j = 1}^n \age_j(\bE)} \iota_*[\scrM_{g, \ddata}(\cZ, \beta)]^\vir
  \end{equation*}
  where
  $p\colon \UH^{\cpt}_{g, \ddata}(\fP, \beta) \to \scrM_{g,
    \ddata}(\cX, \beta)$ sends a log $R$-map to the underlying stable
  map to $\cX$, $\scrM_{g, \ddata}(\cZ, \beta)$ is the moduli of
  stable maps to $\cZ$, and
  $\iota\colon \scrM_{g, \ddata}(\cZ, \beta) \to \scrM_{g,
    \ddata}(\cX, \beta)$ is the inclusion.
\end{corollary}

  Therefore Gromov--Witten invariants of $\cZ$ (involving only cohomology classes
  from the ambient $\cX$) can be computed in terms of (log) GLSM
  invariants defined using $(\cX, W)$.

\subsubsection{Canonical versus reduced virtual cycles}
The canonical perfect obstruction and the canonical cosection of $\UH^{\cpt}_{g, \ddata}(\fP, \beta)$ together defines a reduced perfect obstruction theory of $\Delta_{\UH}$, hence the reduced virtual cycle $[\Delta_{\UH}]^{\red}$, see Section \ref{ss:comparison-2}. The following relates the reduced virtual cycle with the canonical virtual cycle by a third virtual cycle:

\begin{theorem}[Second comparison theorem \ref{thm:comparison-2}]
  \begin{equation*}
    [\UH^{\cpt}_{g,\ddata}(\fP, \beta)]^\vir
    = [\UH^{\cpt}_{g,\ddata}(\fP, \beta)]^\red + \ttwist [\Delta_{\UH}]^\red.
  \end{equation*}
\end{theorem}

By Lemma \ref{lem:pole-of-potential}, $\ttwist$ is the order of poles
of $W$ along $\infty_{\fP}$.
In particular, it is a positive integer.

The fact that the difference between the reduced and canonical virtual
cycles is again virtual allows us to further decompose
$[\Delta_{\UH}]^{\red}$ in \cite{CJR20P2} in terms of canonical and
reduced virtual cycles of punctured and meromorphic theories using
\cite{ACGS17, ACGS20P}.
This is an important ingredient in the proof of the structural
properties of Gromov--Witten invariants of quintics in \cite{GJR18P}.

\subsubsection{Change of twists}
Let $\etwist_1, \etwist_2$ be two twisting choices leading to two targets $\fP_1$ and $\fP_2$ respectively.  Assume that $\frac{\etwist_1}{\etwist_2} \in \ZZ$. Then there is a morphism $\fP_1 \to \fP_2$ by taking $\frac{\etwist_1}{\etwist_2}$-th root stack along $\infty_{\fP_2}$. 

\begin{theorem}[Change of twist theorem \ref{thm:red-ind-twists}]
There is a canonical morphism
$$
\nu_{\etwist_1/\etwist_2} \colon \UH^{\cpt}_{g,\ddata}(\fP_1,\beta) \to \UH^{\cpt}_{g,\ddata}(\fP_2,\beta)
$$
induced by $\fP_1 \to \fP_2$. Pushing forward virtual cycles along $\nu_{\etwist_1/\etwist_2}$, we have
\begin{enumerate}
\item $\nu_{{\etwist_1/\etwist_2},*}[\UH^{\cpt}_{g,\ddata}(\fP_1,\beta)]^{\vir} = [\UH^{\cpt}_{g,\ddata}(\fP_2,\beta)]^{\vir}$,
\item $\nu_{{\etwist_1/\etwist_2},*}[\UH^{\cpt}_{g,\ddata}(\fP_1,\beta)]^{\red} = [\UH^{\cpt}_{g,\ddata}(\fP_2,\beta)]^{\red}$,
\item $\nu_{{\etwist_1/\etwist_2},*}[\Delta_{\UH,1}]^{\red} = \frac{\etwist_2}{\etwist_1} \cdot [\Delta_{\UH,2}]^{\red}.$
\end{enumerate}
where $\Delta_{\UH,i} \subset \UH^{\cpt}_{g,\ddata}(\fP_i,\beta)$ is the boundary \eqref{equ:boundary} for $i=1,2$.
\end{theorem}

\begin{remark}
  The flexibility of twisting choices allows different targets with
  isomorphic infinity hyperplanes.
  The above push-forwards together with the decomposition formulas in
  \cite{CJR20P2} will provide relations among invariants of different
  targets.
  For example, they can be used to prove the
  Landau--Ginzburg/Calabi--Yau correspondence for quintic threefolds
  \cite{GJR19P}, as well as to prove a formula \cite[Conjecture~A.1]{PPZ16P}
  for the class of the locus of holomorphic differentials with
  specified zeros \cite{CJRSZ19P}.
\end{remark}

\subsection{Plan of Paper}

The paper is organized as follows.
In Section~\ref{sec:rmap}, we introduce stable $R$-maps and collect
the basic properties of their moduli spaces.
The canonical and reduced virtual cycles are constructed and the comparison theorems are
proven in Section~\ref{sec:tale}.
In Section~\ref{sec:examples}, we work out several examples
explicitly.
Theorem~\ref{thm:intro-representability} is proven in
Section~\ref{sec:properties}.
Section~\ref{sec:POT-reduction} discusses reducing virtual cycles
along the boundary in more generality, and is used extensively in
Section~\ref{sec:tale}.

\subsection{Acknowledgments}
The first author would like to thank Dan Abramovich, Mark Gross and
Bernd Siebert for the collaborations on foundations of stable log maps which
influenced the development of this project.
Last two authors wish to thank Shuai Guo for the collaborations which
inspired the current work.
The authors would like to thank Adrien Sauvaget, Rachel Webb and Dimitri Zvonkine
for discussions related to the current work.
The authors thank Huai-Liang Chang, Young-Hoon Kiem, Jun Li and
Wei-Ping Li for their inspiring works on cosection localization needed in our construction.
Part of this research was carried out during a visit of the Institute
for Advanced Studies in Mathematics at Zhejiang University. Three of us would like to thank the
Institute for the support.
The first author was partially supported by NSF grant DMS-1700682 and DMS-2001089.
The second author was partially supported by an AMS Simons Travel
Grant and NSF grants DMS-1901748 and DMS-1638352.
The last author was partially supported by Institute for Advanced Study in Mathematics of Zhejiang University,
NSF grant DMS-1405245 and
NSF FRG grant DMS-1159265 .

\subsection{Notations}
In this paper, we work over an algebraically closed field of characteristic zero, denoted by $\bk$. All log structures are assumed to be \emph{fine and saturated} \cite{Ka88} unless otherwise specified.  A list of notations is provided below:

\begin{description}[labelwidth=2cm, align=right]
\item[$\vb(V)$] the total space of a vector bundle $V$

\item[$\PP^\bw(V)$] the weighted projective bundle stack with weights $\bw$

\item[$\cX$] a proper Deligne-Mumford stack with a projective coarse moduli
\item[$\cX \to X$] the coarse moduli morphism
\item[$\BC$] the universal stack of $\CC^*_{\omega}$-torsors
\item[$r$] a positive integer

\item[$\Sp \to \fX$] universal $r$-spin bundle

\item[$\fP \to \BC$] the target of log $R$-maps
\item[$\uC \to \uS$] a family of underlying curves over $\uS$
\item[$\uC \to \ucC$] the coarse moduli morphism of underlying curves

\item[$\cC \to S$] a family of log curves over $S$
\item[$\cC \to C$] the coarse moduli morphism of log curves

\item[$f\colon \cC \to \fP$] a log $R$-map

\item[$\beta$] a curve class in $\cX$
\item[$n$] the number of markings
\item[$\ddata$] collection of discrete data at all markings

\item[$\SH_{g,\ddata}(\fP^{\circ},\beta)$] the moduli stack of stable R-maps

\item[$\SH_{g,\ddata}(\fP,\beta)$] the moduli stack of stable log $R$-maps

\item[$\UH_{g,\ddata}(\fP,\beta)$] the moduli stack of stable log $R$-maps with uniform maximal degeneracy

\item[$\SH^{\cpt}_{g,\ddata}(\fP^{\circ},\beta)$] the moduli stack of stable R-maps with compact type evaluations

\item[$\SH^{\cpt}_{g,\ddata}(\fP,\beta)$] the moduli stack of stable log $R$-maps with compact type evaluations

\item[$\UH^{\cpt}_{g,\ddata}(\fP,\beta)$] the moduli stack of stable log $R$-maps with compact type evaluations and uniform maximal degeneracy

\item[$W\colon \fP^\circ \to \LR$] the superpotential

\end{description}

\section{Logarithmic $R$-maps}
\label{sec:rmap}

\subsection{Twisted curves and pre-stable maps}
We first collect some basic notions needed in our construction.

\subsubsection{Twisted curves}
Recall from \cite{AbVi02} that a \emph{twisted $n$-pointed curve} over a scheme $\ul{S}$ consists of the following data
\[
(\ul{\cC} \to \ul{C} \to \ul{S}, \{p_i\}_{i=1}^n)
\]
where
\begin{enumerate}
 \item $\ul{\cC}$ is a Deligne--Mumford stack proper over $\ul{S}$, and \'etale locally is a nodal curve over $\ul{S}$.
 \item $p_i \subset \ul{\cC}$ are disjoint closed substacks in the smooth locus of $\ul{\cC} \to \ul{S}$.
 \item $p_i \to \ul{S}$ are \'etale gerbes banded by the multiplicative group $\mu_{r_i}$ for some positive integer $r_i$.
 \item the morphism $\ul{\cC} \to \ul{C}$ is the coarse moduli morphism.
 \item Each node of $\ul{\cC} \to \ul{S}$ is balanced.
 \item $\ul{\cC} \to \ul{C}$ is an isomorphism over $\ul{\cC}_{gen}$, where $\ul{\cC}_{gen}$ is the complement of the markings and the stacky critical locus of $\ul{\cC} \to \ul{S}$.
\end{enumerate}

The balancing condition means that formally locally near a node, the
geometric fiber is isomorphic to the stack quotient
\[
[\spec \big(\bk[x,y]/(xy)\big) \big/ \mu_k]
\]
where $\mu_k$ is some cyclic group with the action
$\zeta(x,y) = (\zeta\cdot x, \zeta^{-1}\cdot y)$.
Given a twisted curve as above, by \cite[4.11]{AbVi02} the coarse
space $\ul{C} \to \ul{S}$ is a family of $n$-pointed usual pre-stable
curves over $\ul{S}$ with the markings determined by the images of
$\{p_i\}$.
The \emph{genus} of the twisted curve $\ul{\cC}$ is defined as the
genus of the corresponding coarse pre-stable curve $\ul{C}$.

When there is no danger of confusion, we will simply write
$\ul{\cC} \to \ul{S}$, and the terminology twisted curves and pre-stable curves are interchangable.

\subsubsection{Logarithmic curves}
\label{sss:log-curves}
An \emph{$n$-pointed log curve} over a fine and saturated log scheme $S$ in the sense of \cite{Ol07} consists of
\[
(\pi\colon \cC \to S, \{p_i\}_{i=1}^n)
\]
such that
\begin{enumerate}
 \item The underlying data $(\ul{\cC} \to \ul{C} \to \ul{S}, \{p_i\}_{i=1}^n)$ is a twisted $n$-pointed curve over $\ul{S}$.
 \item $\pi$ is a proper, logarithmically smooth, and integral morphism of fine and saturated logarithmic stacks.
 \item If $\ul{U} \subset \ul{\cC}$ is the non-singular locus of $\ul{\pi}$, then $\ocM_{\cC}|_{\ul{U}} \cong \pi^*\ocM_{S}\oplus\bigoplus_{i=1}^{n}\NN_{p_i}$ where $\NN_{p_i}$ is the constant sheaf over $p_i$ with fiber $\NN$.
\end{enumerate}

For simplicity, we may refer to $\pi\colon \cC \to S$ as a log curve
when there is no danger of confusion.
The \emph{pull-back} of a log curve $\pi\colon \cC \to S$ along an
arbitrary morphism of fine and saturated log schemes $T \to S$ is the
log curve $\pi_T\colon \cC_T:= \cC\times_S T \to T$ with the fiber
product taken in the category of fine and saturated log stacks.

Given a log curve $\cC \to S$, we associate the \emph{log cotangent bundle} $\omega^{\log}_{\cC/S} := \omega_{\uC/\uS}(\sum_i p_i)$ where $\omega_{\uC/\uS}$ is the relative dualizing line bundle of the
underlying family $\uC \to \uS$.

\subsection{Logarithmic $R$-maps as logarithmic fields}\label{ss:log-fields}

In this subsection, we reformulate the notion of a log $R$-map in
terms of the more concrete notion of spin-maps with fields.
This will be useful for relating to previous constructions in GLSM (see Section \ref{sec:examples}), and for some of the proofs in
Section~\ref{sec:properties}.
\begin{definition}\label{def:spin}
  Let $\ul{g}\colon \ul{\cC} \to \ul{\cX}$ be a pre-stable map
  over $\ul{S}$.
  An \emph{$r$-spin structure} of $\ul{g}$ is a line bundle $\cL$ over
  $\ucC$ together with an isomorphism
  \[\cL^{\otimes r} \cong \omega^{\log}_{\ul{\cC}/\ul{S}}\otimes g^*\bL^{\vee}.\]
  The pair $(\ul{g}, \cL)$ is called an {\em $r$-spin map}.
\end{definition}

Given a log map $g\colon \cC \to \cX$ over $S$ and an $r$-spin
structure $\cL$ of the underlying map $\ul{g}$, we introduce a
weighted projective stack bundle over $\uC$:
\begin{equation}\label{equ:proj-bundle}
  \ul{\cP}_{\cC} := \PP^\bw\left(\bigoplus_{i > 0} (g^*(\bE_i^\vee) \otimes \cL^{\otimes i}) \oplus \cO\right)
\end{equation}
where $\bw$ indicates the weights of the $\GG_m$-action as in
\eqref{equ:universal-proj}.
The Cartier divisor $\infty_{\cP} \subset \ul{\cP}_{\cC}$ defined by the vanishing of the last coordinate, 
is called the \emph{infinity
  hyperplane}.
Let $\cM_{\infty_{\cP}}$ be the log structure on $\ul{\cP}_{\cC}$
associated to the Cartier divisor $\infty_{\cP}$, see \cite{Ka88}. Form the log stack
\[
\cP_{\cC} := (\ul{\cP}_{\cC}, \cM_{\cP_\cC} := \cM_{\cC}|_{\ul{\cP}_{\cC}}\oplus_{\cO^*}\cM_{\infty_{\cP}}),
\]
with the natural projection $\cP_{\cC} \to \cC$.

\begin{definition}
  \label{def:log-field}
  A \emph{log field} over an $r$-spin map $(g, \cL)$ is a section
  $\rho\colon \cC \to \cP_{\cC}$ of $\cP_{\cC} \to \cC$.
  The triple $(g, \cL, \rho)$ over $S$ is called an \emph{$r$-spin map
    with a log field}.

  The \emph{pull-back} of an $r$-spin map with a log field is
  defined as the pull-back of log maps.
\end{definition}

We now show that the two notions --- log $R$-maps and pre-stable maps
with log fields --- are equivalent.

\begin{proposition}\label{prop:map-field-equiv}
Fix a log map $g\colon \cC \to \cX$ over $S$, and consider the following diagram of solid arrows
\[
\xymatrix{
\cC \ar@{-->}[rrd]  \ar@/^2ex/@{-->}[rrrd] \ar@/^4ex/[rrrrd]^{g} \ar@/_6ex/[rrrdd]_{\omega^{\log}_{\cC/S}}  &&& \\
&& \fP  \ar[r] & \fX \ar[d]^{\zeta}  \ar[r] & \cX  \\
&& & \BC  &
}
\]
We have the following equivalences:
\begin{enumerate}
\item The data of an $r$-spin map $(g, \cL)$ is equivalent to
   a morphism $\cC \to \fX$ making the above diagram
  commutative.
\item The data of a log field $\rho$ over a given $r$-spin map
  $(g, \cL)$ is equivalent to giving a log $R$-map
  $f\colon \cC \to \fP$ making the above diagram commutative.
\end{enumerate}
\end{proposition}
\begin{proof}
  The first equivalence follows from Definition \ref{def:spin} and \eqref{diag:spin-target}.
  
  Note that $(g, \cL)$ induces a commutative diagram of solid arrows with all squares cartesian: 
    \begin{equation}\label{diag:map-field}
    \xymatrix{
      \cP_\cC \ar[d] \ar[r] &\fP_{\cC} \ar[r] \ar[d] & \fP \ar[d] \\
      \cC \ar[r] \ar[rd]_{=} \ar@/^1pc/@{-->}[u]^{\rho}& \fX_{\cC} \ar[r] \ar[d] & \fX \ar[d] \\
      & \cC \ar[r]^{\omega_{\cC/S}^{\log}} &  \BC.
    }
  \end{equation}
  Thus (2) follows from the universal property of cartesian squares.
\end{proof}

\begin{definition}\label{def:field-stability}
  An $r$-spin map with a log field is \emph{stable} if the
  corresponding $R$-map is stable.
\end{definition}

Let $f\colon \cC \to \fP$ be a logarithmic $R$-map over $S$, and
$\rho\colon\cC \to \cP_\cC$ be the corresponding logarithmic field.
Using \eqref{diag:map-field}, we immediately obtain
\[
  f^* \cO(\ttwist\infty_{\fP}) = \rho^* \cO(\ttwist\infty_{\cP_{\cC}}),
\]
hence the following equivalent description of the stability condition:
\begin{corollary}\label{cor:field-stability}
  A pre-stable map with log field $(g, \cL, \rho)$ over $S$ is
  stable iff the corresponding $R$-map $f$ is representable, and  if for a sufficiently small $\delta_0 \in (0,1)$ there exists $k_0 > 1$ such that for any pair $(k,\delta)$ satisfying $k > k_0$ and $\delta_0 > \delta > 0$, the following holds
  \begin{equation}\label{equ:field-stability}
    (\omega_{\cC/S}^{\log})^{1 + \delta} \otimes g^* \cH^{\otimes k} \otimes \rho^* \cO(\ttwist\infty) > 0.
  \end{equation}
\end{corollary}

\begin{remark}\label{rem:r-spin-stability}
  The condition \eqref{equ:field-stability} is compatible
  with the stability of log $r$-spin fields in
  \cite[Definition~4.9]{CJRS18P}.
  Let $\cX = \spec \bk$ and $\rho\colon \cC \to \cP_\cC$ be a log $r$-spin
  field over $S$ as in \cite{CJRS18P}.
  The stability of $\rho$ is equivalent to
  \begin{align*}
    0 & < \omega^{\log}_{\cC/S}\otimes \rho^*\cO(k\cdot \zero_{\cP}) \\
   &= \omega^{\log}_{\cC/S}\otimes \cL^{\otimes k}\otimes\rho^*\cO(k\cdot \infty_{\cP}) \\
   &= \left((\omega^{\log}_{\cC/S})^{1+\frac{r}{k}}\otimes\rho^*\cO(r\cdot \infty_{\cP})\right)^{\otimes \frac{k}{r}}
  \end{align*}
for $k \gg 0$. Now replacing $\frac{r}{k}$ by $\delta$ in
\[
(\omega^{\log}_{\cC/S})^{1+\frac{r}{k}}\otimes\rho^*\cO(r\cdot \infty_{\cP}) > 0,
\]
we recover \eqref{equ:field-stability} as desired.
\end{remark}

\subsection{The structure of the infinity divisor}

For later use, we would like to study the structure of $\infty_{\fP}$.

Let $\bw$ and $\bw'$ be two weights as in \eqref{equ:proj-bundle} such that $\bw'$ corresponds to $\etwist = \frac{1}{d}$. Consider $\bw_{\infty}$ (resp.\ $\bw'_{\infty}$) obtained by removing the weight of the $\cO$ factor from $\bw$ (resp.\ $\bw'$).
Since $\gcd(\bw'_{\infty}) = 1$, we observe that
\[
\infty_{\fP'} = \PP^{\bw'_{\infty}}\big(\bigoplus_i\bE^{\vee}_{i,\fX}\otimes \Sp^{\otimes i}\big) \cong \PP^{\bw'_{\infty}}\big(\bigoplus_i\bE^{\vee}_{i,\fX}\big).
\]
In particular, there is a cartesian diagram
\begin{equation}\label{diag:infinity-pullback}
\xymatrix{
\infty_{\fP'} \ar[rr] \ar[d] && \infty_{\cX} := \PP^{\bw'_{\infty}}\big(\bigoplus_i\bE^{\vee}_{i}\big) \ar[d] \\
\fX \ar[rr] && \cX.
}
\end{equation}
To fix the notation, denote by $\cO_{\infty_{\cX}}(1)$ the tautological line bundle over $\infty_{\cX}$ associated to the upper right corner, and by $\cO_{\infty_{\fP'}}(1)$ the pull-back of $\cO_{\infty_{\cX}}(1)$ via the top horizontal arrow. Let $\cO_{\fP'}(1)$ be the tautological line bundle associated to the expression of $\fP'$ as in \eqref{equ:universal-proj}.

Let $\ell = \gcd(\bw_{\infty})$. Observe that
$
\ul{\fP} \to \ul{\fP}'
$
is an $\ell$-th root stack along $\infty_{\fP'}$. Thus, $\infty_{\fP}$ parameterizes $\ell$-th roots of the normal bundle $N_{\infty_{\fP'}/\fP'}$ over $\infty_{\fP'}$. In particular, the morphism
\[
\infty_{\fP} \to \infty_{\fP'}
\]
is a $\mu_{\ell}$-gerbe.

As shown below, the small number ``$\delta$'' in the stability condition \eqref{equ:hyb-stability} plays an important role in stabilizing components in $\infty_{\fP}$.

\begin{proposition}\label{prop:curve-in-infinity}
Consider an underlying $R$-map
\[
\xymatrix{
 && \infty_{\fP} \ar[d]^{\zeta\circ\fp} \\
\uC \ar[rru]^{\ul{f}} \ar[rr]_{\omega^{\log}_{\cC/S}}&& \BC
}
\]
over a geometric point. Consider the following commutative diagram
\begin{equation}\label{equ:map-to-rigid-infinity}
\xymatrix{
\uC \ar[r] \ar@/^3ex/[rr]^{\ul{f}'} \ar@/_3ex/[rrr]_{\ul{f}_{\cX}} & \infty_{\fP} \ar[r] & \infty_{\fP'} \ar[r] & \infty_{\cX}
}
\end{equation}
Then we have
\begin{equation}\label{equ:curve-in-infinity}
  \ul{f}^* \cO_{\fP}(\ttwist \infty_{\fP})
  = (\omega^{\log}_{\uC})^{\vee}\otimes \ul{f}^*\bL\otimes (\ul{f}_\cX)^*\cO_{\infty_{\cX}}(\frac{r}{d}).
\end{equation}
Furthermore, we have
\begin{equation}\label{equ:infinity-stability}
(\omega_{\uC}^{\log})^{1 + \delta} \otimes (\ft \circ f)^* \cH^{\otimes k} \otimes f^* \cO(\ttwist\infty_{\fP}) > 0
\end{equation}
if and only if the coarse of $\ul{f}_{\cX}$ is stable in the usual sense.
\end{proposition}
\begin{proof}
Recall $\bw$ corresponds to the choice $\etwist = \frac{\ell}{d}$. We have
\[
\ul{f}^* \cO_{\fP}(\ttwist \infty_{\fP}) = (\ul{f}')^*\cO_{\fP'}(\frac{r}{d}\cdot \infty_{\fP'}).
\]
Since $\cO_{\fP'}(\infty_{\fP'}) \cong \cO_{\fP'}(1)$, we calculate
\[
(\ul{f}')^*\cO_{\fP'}(\infty_{\fP'})|_{\infty_{\fP'}} \cong (\ul{f}')^*\cO_{\infty_{\fP'}}(1)\otimes\cL^{\otimes -d} = (\ul{f}_{\cX})^*\cO_{\infty_{\cX}}(1)\otimes \cL^{\otimes -d}.
\]
Equation \eqref{equ:curve-in-infinity} is proved by combining the above calculation and Definition \ref{def:spin}.

Now using \eqref{equ:curve-in-infinity}, we obtain
  \begin{multline}
    \label{eq:positive-in-infinity}
    (\omega_{\ul{\cC}}^{\log})^{1 + \delta} \otimes (\ft \circ f)^* \cH^{\otimes k} \otimes\ul{f}^* \cO(\ttwist\infty_{\fP}) \\
    \cong (\omega_{\ul{\cC}}^{\log})^\delta \otimes (\ft \circ f)^* \cH^{\otimes k} \otimes (\ft \circ f)^*\bL\otimes (f_\cX)^*\cO_{\infty_{\cX}}(\frac{r}{d}),
  \end{multline}
  Let $\ul{\cZ} \subset \ul{\cC}$ be an irreducible component.
  Note the \eqref{equ:infinity-stability} holds over $\ul{\cZ}$ for
  $k \gg 0$ unless $\ft \circ f$ contracts $\cZ$ to a point.
  Suppose we are in the latter situation, hence both
  $(\ft \circ f)^* \cH^{\otimes k}$ and $(\ft \circ f)^*\bL$ have
  degree zero over $\ul{\cZ}$.
  Since $\ul{f}_{\cX}^*\cO_{\infty_\cX}(1)|_{\ul{\cZ}}$ has
  non-negative degree and $1 \gg\delta > 0$,
  \eqref{equ:infinity-stability} holds if and only if either
  $\ul{f}_{\cX}(\ul{\cZ})$ is not a point, or
  $\omega^{\log}_{\ul{\cC}}|_{\cZ}$ is positive.
  This proves the second statement.
\end{proof}

\begin{corollary}\label{cor:rat-bridge-stability}
  Let $\ul{f}\colon \ul{\cC} \to \ul{\fP}$ be
  an underlying R-map.
  Then a rational bridge $\cZ \subset \ul{\cC}$ fails to satisfy the
  stability condition \eqref{equ:hyb-stability} if and only if
  $\deg \ul{f}^*\cO_{\fP}(\infty_{\fP})|_{\cZ} = 0$ and
  $\deg (\ft \circ f)^*\cH|_{\cZ} = 0$.
\end{corollary}
\begin{proof}
  Suppose $\cZ$ is unstable.
  Then \eqref{equ:hyb-stability} implies that
  $\deg (\ft \circ f)^*\cH^{\otimes k}|_{\cZ} = 0$ for any $k \gg 0$,
  hence $\deg (\ft \circ f)^*\cH|_{\cZ} = 0$.
  Since $\omega^{\log}_{\cC/S}|_{\cZ} = \cO_{\cZ}$, we have
  $\deg \ul{f}^*\cO_{\fP}(\infty_{\fP})|_{\cZ} \leq 0$.
  It is clear that
  $\deg \ul{f}^*\cO_{\fP}(\infty_{\fP})|_{\cZ} \geq 0$ if
  $f(\cZ) \not\subset \infty_{\fP}$.
  If $f(\cZ) \subset \infty_{\fP}$, then \eqref{equ:curve-in-infinity}
  implies $\deg \ul{f}^*\cO_{\fP}(\infty_{\fP})|_{\cZ} \geq 0$.
  Thus we have $\deg \ul{f}^*\cO_{\fP}(\infty_{\fP})|_{\cZ} = 0$.

  The other direction follows immediately from
  \eqref{equ:hyb-stability}.
\end{proof}

\subsection{The combinatorial structures}

The \emph{minimality} or \emph{basicness} of stable log maps, which
plays a crucial role in constructing the moduli of stable log maps,
was introduced in \cite{AbCh14, Ch14, GrSi13}.
Based on their construction, a modification called \emph{minimality
  with uniform maximal degeneracy} has been introduced in
\cite{CJRS18P} for the purpose of constructing reduced virtual cycles.
We recall these constructions for later reference.

\subsubsection{Degeneracies and contact orders}\label{ss:combinatorics}
We fix a log $R$-map $f\colon \cC \to \fP$ over $S$. Consider the induced morphism of characteristic sheaves:
\begin{equation}\label{equ:combinatorics}
f^{\flat} \colon f^*\ocM_{\fP} \to \ocM_{\cC}.
\end{equation}
Note that characteristic sheaves are constructible. We recall the following terminologies.\\

\noindent{(1)\em Degeneracies of irreducible components.}
An irreducible component $\cZ\subset \cC$ is called \emph{degenerate} if $(f^*\ocM_{\fP})_{\eta_\cZ}\cong \NN$ where $\eta_\cZ \in \cZ$ is the generic point, and \emph{non-degenerate} otherwise.
Equivalently $\cZ$ is degenerate iff $f(\cZ) \subset \infty_{\fP}$.
For a degenerate $\cZ$, write
\[
e_\cZ =\ \bar{f}^{\flat}(1)_{\eta_\cZ} \in \ocM_{\cC, \eta_\cZ} = \ocM_{S}
\]
and call it the \emph{degeneracy} of $\cZ$.  If $\cZ$ is non-degenerate, set $e_\cZ = 0$.

An irreducible component $\cZ$ is called \emph{a maximally degenerate component}, if $e_{\cZ'} \poleq e_{\cZ}$ for any irreducible component $\cZ'$. Here for $e_1, e_2 \in \ocM_{S}$, we define $e_1 \poleq e_2$ iff $(e_2 - e_2) \in \ocM_S$.

\noindent{(2)\em The structure at markings.}
Let $p \in \cZ$ be a marked point. Consider
\[
(f^*\ocM_{\fP})_{p} \stackrel{\bar{f}^{\flat}}{\longrightarrow} \ocM_{\cC,p} \cong \ocM_S\oplus \NN \longrightarrow  \NN
\]
where the arrow on the right is the projection.
If $(f^*\ocM_{\fP})_{p} \cong \NN$ or equivalently $f(p) \in \infty_{\fP}$, we denote by
$c_p \in \ZZ_{\geq 0}$ the image of $1 \in \NN$ via the above
composition, and $c_p = 0$ otherwise.
We call $c_p$ the \emph{contact order} at the marking $p$. Contact orders are a generalization of tangency multiplicities in the log setting.\\

\noindent{(3)\em The structure at nodes.}
Define the \emph{natural partial order} $\poleq$ on the set of irreducible components of $\cC$ such that $\cZ_1 \poleq \cZ_2$ iff $(e_{\cZ_2} - e_{\cZ_1}) \in \ocM_S$.

Let $q \in \cC$ be a node joining two irreducible components $\cZ_1$ and $\cZ_2$ with $\cZ_1 \poleq \cZ_2$.
Then \'etale locally at $q$, \eqref{equ:combinatorics} is of the form
\[
(\bar{f}^{\flat})_q \colon   (f^*\ocM_{\fP})_q \to \ocM_{\cC,q} \cong \ocM_{S}\oplus_{\NN}\NN^2,
\]
where the two generators $\sigma_1$ and $\sigma_2$ of $\NN^2$
correspond to the coordinates of $\cZ_1$ and $\cZ_2$ at $q$
respectively, and the arrow $\NN := \langle\ell_q\rangle \to \NN^2$ is
the diagonal $\ell_q \mapsto \sigma_1 + \sigma_2$.
If $(f^*\ocM_{\fP})_q \cong \NN$ or equivalently
$f(q) \in \infty_{\fP}$, we have
\[
(\bar{f}^{\flat})_q (1) = e_{\cZ_1} + c_q \cdot \sigma_1,
\]
where the non-negative integer $c_q$ is called the \emph{contact order} at $q$. In this case, we have a relation between the two degeneracies
\begin{equation}\label{equ:edge-relation}
e_{\cZ_1} + c_q \cdot \ell_q = e_{\cZ_2}.
\end{equation}
If $(f^*\ocM_{\fP})_q$ is trivial, then we set the contact order $c_q = 0$. Note that in this case $e_{\cZ_1} = e_{\cZ_2} = 0$, and \eqref{equ:edge-relation} still holds.

\subsubsection{Minimality}

We recall the construction of minimal monoids in \cite{Ch14, AbCh14,
  GrSi13}.
The \emph{log combinatorial type} of the $R$-map $f$ consists of:
\begin{equation}\label{equ:combinatorial-type}
G = \big(\ul{G},  V(G) = V^{n}(G) \cup V^{d}(G), \poleq, (c_i)_{i\in L(G)}, (c_l)_{l\in E(G)} \big)
\end{equation}
where
\begin{enumerate}[(i)]
 \item $\ul{G}$ is the dual intersection graph of the underlying curve $\ul{\cC}$.

 \item $V^{n}(G) \cup V^{d}(G)$ is a partition of $V(G)$ where $V^{d}(G)$ consists of vertices of with non-zero degeneracies.

 \item $\poleq$ is the natural partial order on the set $V(G)$.

 \item We associate to a leg $i\in L(G)$ the contact order $c_i \in \NN$ of the corresponding marking $p_i$.

 \item We associate to an edge $l\in E(G)$ the contact order $c_l \in \NN$ of the corresponding node.
\end{enumerate}

We introduce a variable $\ell_l$ for each edge $l \in E(G)$, and a variable $e_v$ for each vertex $v \in V(G)$. Denote by $h_l$ the relation
$ e_{v'} = e_v + c_l\cdot\ell_l
$
for each edge $l$ with the two ends $v \poleq v'$ and contact order $c_l$. Denote by $h_v$ the following relation
$
e_v = 0
$
for each $v \in V^{n}(G)$. Consider the following abelian group
\[
\mathcal{G} = \left(\big(\bigoplus_{v \in V(G)} \ZZ e_v\big) \oplus \big( \bigoplus_{l \in E(G)} \ZZ \rho_l \big) \right) \big/ \langle h_v, h_l \ | \ v\in V^{d}(G), \ l \in E(G) \rangle
\]
Let $\mathcal{G}^{t} \subset \mathcal{G}$ the torsion subgroup. Consider the following composition
\[
 \big( \bigoplus_{v \in V(G)} \NN e_v \big) \oplus \big( \bigoplus_{l \in E(G)} \NN \rho_l\big) \to \mathcal{G} \to \mathcal{G}/\mathcal{G}^{t}
\]
Let $\oM(G)$ be the smallest submonoid that is saturated in $\mathcal{G}/\mathcal{G}^{t}$, and contains the image of the above composition. We call $\oM(G)$ the \emph{minimal or basic monoid} associated to $G$.

Recall from \cite[Proposition 3.4.2]{Ch14}, or \cite[Proposition 2.5]{CJRS18P} that there is a canonical map of monoids
\begin{equation}\label{equ:minimal}
 \oM(G) \to \ocM_S
\end{equation}
induced by sending $e_v$ to the degeneracy of the component associated to $v$, and sending $\ell_l$ to the element $\ell_{q}$ as in \eqref{equ:edge-relation} associated to $l$. In particular, the monoid $\oM(G)$ is fine, saturated, and sharp.

\begin{definition}\label{def:minimal}
A log $R$-map is \emph{minimal} or \emph{basic} if over each geometric fiber, the natural morphism \eqref{equ:minimal} is an isomorphism.
\end{definition}

\subsubsection{Logarithmic $R$-map with uniform maximal degeneracy}
\label{sss:UMD}

\begin{definition}
A log $R$-map is said to have \emph{uniform maximal degeneracy} if there exists a maximal degenerate component over each geometric fiber, see Section \ref{ss:combinatorics} (1).
\end{definition}

Let $f \colon \cC \to \fP$ be a log $R$-map over a geometric log point $S$, and $G$ be its log combinatorial type. Assume that $f$ has uniform maximal degeneracy, and denote by $V_{\max} \subset V(G)$ the collection of vertices with the maximal degeneracy. We call $(G, V_{\max})$ the \emph{log combinatorial type with uniform maximal degeneracy}, and form the corresponding minimal monoid below.

Consider the torsion-free abelian group
\[
\big( \oM(G)^{gp}\big/ \sim \big)^{tf}
\]
where $\sim$ is given by the relations $(e_{v_1} - e_{v_2}) = 0$ for any $v_1, v_2 \in V_{\max}$. By abuse of notation, we may use $e_v$ for the image of the degeneracy of the vertex $v$ in $\big( \oM(G)^{gp}\big/ \sim \big)^{tf}$. Thus, for any $v \in V_{\max}$ their degeneracies in $\big( \oM(G)^{gp}\big/ \sim \big)^{tf}$ are identical, denoted by $e_{\max}$. Let $\oM(G,V_{\max})$ be the saturated submonoid in $\big( \oM(G)^{gp}\big/ \sim \big)^{tf}$ generated by
\begin{enumerate}
\item the image of $\oM(G) \to \big( \oM(G)^{gp}\big/ \sim \big)^{tf}$, and
\item the elements $(e_{\max} - e_v)$ for any $v \in V(G)$.
\end{enumerate}
By \cite[Proposition 3.7]{CJRS18P}, there is a natural morphism of monoids $\oM(G) \to \oM(G,V_{\max})$ which fits in a commutative diagram
\[
\xymatrix{
\oM(G) \ar[r] \ar[rd]_{\phi} & \oM(G,V_{\max}) \ar[d]^{\phi_{\max}}\\
& \ocM_S
}
\]
We call $\oM(G,V_{\max})$ the \emph{minimal monoid with uniform maximal degeneracy} associated to $(G, V_{\max})$, or simply the \emph{minimal monoid} associated to $(G,V_{\max})$.

\begin{definition}\label{def:umd-minimal}
A log $R$-map is \emph{minimal with uniform maximal degeneracy} if over each geometric fiber the morphism $\phi_{\max}$ is an isomorphism.
\end{definition}

Note that in general a log $R$-map minimal with
uniform maximal degeneracy does not need to be minimal in the sense of
Definition \ref{def:minimal}.

\subsubsection{The universal logarithmic target}\label{sss:universal-log-target}
Consider the log stack $\cA$ with the underlying stack $[\A^1/\GG_m]$ and log structure induced by its toric boundary. It parameterizes Deligne--Faltings log structures of rank one \cite[A.2]{Ch14}. Thus there is a canonical strict morphism of log stacks
\begin{equation}\label{equ:universal-log-target}
\fP \to \cA.
\end{equation}
Let $\infty_{\cA} \subset \cA$ be the strict closed substack, then $\infty_{\fP} = \infty_{\cA}\times_{\cA}\fP$.

Given any log R-map $f \colon \cC \to \fP$, we obtain a log map
$f' \colon \cC \to \cA$ via composing with
\eqref{equ:universal-log-target}.
Then $f'$ and $f$ share the same log combinatorial type (with uniform
maximal degeneracy) since
\[
(f')^*\cM_{\cA} \cong f^*\cM_{\fP} \ \ \ \mbox{and} \ \ \ f^{\flat} = (f')^{\flat}.
\]
This point of view will be used later in our construction. 

\subsection{The evaluation morphism  of the underlying structure}\label{ss:underlying-evaluation}
Denote by $\fP_\bk := \fP\times_{\BC}\spec \bk$
where the arrow on the right is the universal $\GG_m$-torsor.
Let $\cI\fP_\bk$ be the cyclotomic inertia stack of $\fP_{\bk}$ \cite[Definition 3.1.5]{AGV08}.
Then $\cI\infty_{\fP_\bk} = \cI\fP_\bk\times_{\cA}\infty_{\cA}$ is the
cyclotomic inertia stack of $\infty_{\fP_\bk}$ equipped with the
pull-back log structure from $\cA$.

\begin{lemma}\label{lem:underlying-evaluation}
Let $f\colon \cC \to \fP$ be a log $R$-map over $S$, and $p \subset \cC$ be a marking. Then the restriction $f|_{p}$ factors through $\fP_\bk \to \fP$. Furthermore, $f$ is representable along $p$ if the induced morphism $p \to \fP_\bk$ is representable.
\end{lemma}
\begin{proof}
Since $\omega^{\log}_{\cC/S}|_{p} \cong \cO_{p}$, the composition $\ul{p} \to \ul{\cC} \to \BC$ factors through $\spec \bk \to \BC$. This proves the statement.
\end{proof}

Consider the universal gerbes $\cI\fP_\bk \to \ocI\fP_\bk$ and
$\cI\infty_{\fP_\bk} \to \ocI\infty_{\fP_\bk}$ in $\fP_\bk$ and
$\infty_{\fP_\bk}$ \cite[Section 3]{AGV08}.
Let $f\colon \cC \to \fP$ be a log $R$-map over $S$ with constant
contact order $c_i$ along its $i$-th marking $p_i \subset \cC$.
Write $\ocI^i = \ocI\fP_\bk$ if $c_i = 0$, and
$\ocI^i = \ocI\infty_{\fP_\bk}$, otherwise.
By the above lemma, the restriction $f|_{p_i}$ induces the
\emph{$i$-th evaluation morphism of the underlying structures}
\begin{equation}\label{equ:underlying-evaluation}
\ev_i\colon \ul{S} \to \ocI^{i}
\end{equation}
such that $p_i \to \ul{S}$ is given by the pull-back of the universal gerbe over $\ocI^{i}$. Thus, connected components of $\ocI\fP_\bk \cup \ocI\infty_{\fP_\bk}$ provide discrete data for log $R$-maps. Note that $\ocI\fP_\bk \cup \ocI\infty_{\fP_\bk}$ is smooth provided that $\fP \to \BC$, hence $\fP_\bk$ is smooth.

\begin{definition}\label{def:hyb-sector}
  A \emph{log sector} $\gamma$ is a connected component of
  either $\ocI\fP_\bk$ or $\ocI\infty_{\fP_\bk}$.
  It is \emph{narrow} if gerbes parameterized by $\gamma$ all avoids $\infty_{\fP_\bk}$. 
  
  A {\em sector of compact type} is a connected component of $\ocI\zero_{\fP_\bk}$. In particular all narrow sectors are of compact type.
\end{definition}

Due to the fiberwise $\CC^*_\omega$-action of $\fP \to \fX$, it is easy to
see that a sector is narrow iff it parameterizes gerbes in
$\zero_{\fP_\bk}$.
Thus, the above definition is compatible with
\cite[Definition~4.1.3]{FJR18}.
Furthermore, since $\zero_{\fP_\bk}$ and $\infty_{\fP_\bk}$ are
disjoint, the compact-type condition forces the contact order to be
trivial.

\subsection{The stack of logarithmic $R$-maps}\label{ss:hybrid-stack}

The discrete data of a log $R$-map $f\colon\cC \to \fP$
consists of the genus $g$, and the curve class $\beta \in H_2(\cX)$ of
$\ft \circ f$.
Furthermore, each marking has {\em discrete data} given by its contact
order $c$ and the log sector $\gamma$.
Let $\ddata = \{(\gamma_i, c_i)\}_{i=1}^n$ be the collection of discrete data at all markings where $n$ is the number of markings.

Denote by $\SH_{g, \ddata}(\fP, \beta)$ the stack of stable $R$-maps
over the category of logarithmic schemes with discrete data $g$,
$\beta$, $\ddata$.
Let $\UH_{g, \ddata}(\fP, \beta)$
be the category of objects with uniform maximal
degeneracy.
There is a tautological morphism \cite[Theorem~3.14]{CJRS18P}
\begin{equation}\label{equ:forget-uniform-degeneracy}
\UH_{g, \ddata}(\fP, \beta) \to \SH_{g, \ddata}(\fP, \beta).
\end{equation}
which is representable, proper, log \'etale, and surjective. Furthermore, \eqref{equ:forget-uniform-degeneracy} restricts to the identity over the open substack parameterizing log $R$-maps with images in $\fP^{\circ}$.

\begin{theorem}\label{thm:representability}
The categories $\SH_{g, \ddata}(\fP, \beta)$ and $\UH_{g, \ddata}(\fP, \beta)$ are represented by proper log Deligne--Mumford stacks.
\end{theorem}
\begin{proof}
Since \eqref{equ:forget-uniform-degeneracy} is representable and proper, it suffices to verify the statement for
$\SH_{g, \ddata}(\fP, \beta)$, which will be done in Section \ref{sec:properties}. A key to the representability  is the fact discovered in \cite{AbCh14, Ch14,GrSi13} that the underlying stack $\ul{\SH_{g, \ddata}(\fP, \beta)}$ is the stack of minimal objects in Definition \ref{def:minimal}, and $\ul{\UH_{g, \ddata}(\fP, \beta)}$ is the stack of minimal objects in Definition \ref{def:umd-minimal}, see also \cite[Theorem 2.11]{CJRS18P}.
\end{proof}

\subsection{Change of twists}\label{ss:change-twists}

This section studies R-maps under the change of twists in preparation
of the proof of the Change of twist theorem
(Theorem~\ref{thm:red-ind-twists}).
The reader may skip this section on first reading, and return when
studying the proof of Theorem~\ref{thm:red-ind-twists}.

Consider two twisting choices
$\etwist_1, \etwist_2 \in \frac{1}{d}\cdot \ZZ$ such that
$\frac{\etwist_1}{\etwist_2} \in \ZZ$.
Let $\fP_1$ and $\fP_2$ be the hybrid targets corresponding to the
choices of $\etwist_1$ and $\etwist_2$ respectively as in
\eqref{equ:universal-proj}.
Then there is a cartesian diagram of log stacks
\begin{equation}\label{diag:targets-twists}
\xymatrix{
\fP_1 \ar[rr] \ar[d] &&  \fP_2 \ar[d] \\
\cA_1 \ar[rr]^{\nu} && \cA_2
}
\end{equation}
where $\cA_{1}$ and $\cA_2$ are two copies of $\cA$, the vertical arrows are given by \eqref{equ:universal-log-target}, and $\nu$ is the morphism induced by $\NN \to \NN, 1 \mapsto \frac{\etwist_1}{\etwist_2}$ on the level of characteristic monoids.  Note that the top is the $\frac{\etwist_1}{\etwist_2}$-th root stack along $\infty_{\fP_2}$ in $\fP_2$, and is compatible with arrows to $\BC$.

\begin{proposition}\label{prop:change-twists}
Let $f'\colon \cC' \to \fP_1$ be a stable log $R$-map over $S$. Then the composition $\cC' \to \fP_1 \to \fP_2$ factors through a stable log $R$-map $f\colon \cC \to \fP_2$ over $S$ such that
\begin{enumerate}
 \item The morphism $\cC' \to \cC$ induces an isomorphism of their coarse curves, denoted by $C$.
 \item The underlying coarse morphisms of $\cC' \to C\times_{\BC}\fP_1$ and $\cC \to C\times_{\BC}\fP_2$ are isomorphic.
 \item If $f'$ has uniform maximal degeneracy, so does $f$.
\end{enumerate}
Furthermore, this factorization is unique up to a unique isomorphism.
\end{proposition}
\begin{proof}
  Consider the stable log map
  $\cC' \to \fP_{1,C} := \fP_1\times_{\BC}C$ induced by $f'$.
  By \cite[Proposition~9.1.1]{AbVi02}, the underlying map of the
  composition $\cC' \to \fP_{1,C} \to \fP_{2,C} := \fP_2\times_{\BC}C$
  factors through a stable map $\ul{\cC} \to \ul{\fP_{2,C}}$ which
  yields an induced underlying $R$-map $\ul{f}\colon \ul{\cC} \to \fP_{2}$.

  We first construct the log curve $\cC \to S$. Let $\cC^{\sharp} \to S^{\sharp}$ be the canonical log structure associated to the underlying curve. Since $\fP_{1,C} \to \fP_{2,C}$ is quasi-finite, the morphism $\ul{\cC'} \to \ul{\cC}$ induces an isomorphism of coarse curves.  Thus we obtain a log morphism $\cC' \to \cC^{\sharp}$ over $S \to S^{\sharp}$. This yields the log curve $\cC := S\times_{S^{\sharp}}\cC^{\sharp} \to S$.

  Next we show that $f'$ descends to a log map $f\colon \cC \to \fP_2$.
  Since the underlying structure $\ul{f}$ has already being
  constructed, by \eqref{diag:targets-twists} it suffices to show that
  the morphism $h'\colon \cC' \to \cA_1$ induced by $f'$ descends to
  $h\colon \cC \to \cA_2$ with $\ul{h}$ induced by $\ul{f}$.
  Since $\cA$ is an Artin cone, it suffices to check on the level of
  characteristic sheaves over the log \'etale cover $\cC' \to \cC$,
  i.e.\ we need to construct the dashed arrow making the following
  commutative
  \[
    \xymatrix{
      \ocM_{\cA_2}|_{\cC'} \ar@{-->}[r]^{\bar{h}^{\flat}} \ar@{^{(}->}[d] & \ocM_{\cC}|_{\cC'} \ar@{^{(}->}[d] \\
      \ocM_{\cA_1} \ar[r]^{(\bar{h}')^{\flat}} & \ocM_{\cC'}.
    }
  \]
  Thus, it suffices to consider the case where $\ul{S}$ is a geometric point.

  Note that both vertical arrows are injective, hence we may view the monoids on the top as the submonoids of the bottom ones.
  Let $\delta_1$ and $\delta_2$ be a local generator of $\cM_{\cA_1}$ and $\cM_{\cA_2}|_{\cC'}$ respectively. Denote by $\bar{\delta}_1 \in \ocM_{\cA_1}$ and $\bar{\delta}_1 \in \ocM_{\cA_2}|_{\cC'}$ the corresponding elements. Since $\bar{\delta}_2 \mapsto \frac{\etwist_1}{\etwist_2}\cdot \bar{\delta}_1$, it suffices to show that $m := (\bar{h}')^{\flat}(\frac{\etwist_1}{\etwist_2}\cdot \bar{\delta}_1) \in  \ocM_{\cC}|_{\cC'}$.

  Indeed, the morphism
  $\ul{\cC'} \to \ul{\cA_1}\times_{\ul{\cA_2}}\ul{\cC}$ lifting the
  identity of $\ul{\cC}$ is representable.
  Hence along any marking, the morphism $\ul{\cC'} \to \ul{\cC}$ is a
  $\rho$-th root stack with $\rho | \frac{\etwist_1}{\etwist_2}$.
  And along each node, the morphism $\ul{\cC'} \to \ul{\cC}$ is a
  $\rho$-th root stack with $\rho | \frac{\etwist_1}{\etwist_2}$ on
  each component of the node.
  By the definition of log curves, we have
  $\frac{\etwist_1}{\etwist_2}\cdot \ocM_{\cC'} \subset
  \ocM_{\cC}|_{\cC'}$.
  This proves $m \in \ocM_{\cC}|_{\cC'}$ as needed for constructing
  $h$ hence $f$.

Finally, consider any component $Z \subset \cC$ and the unique component $Z' \subset \cC'$ dominating $Z$. Then we have $e_{Z} = \frac{\etwist_1}{\etwist_2}\cdot e_{Z'}$ where $e_Z, e_Z' \in \ocM_S$ are the degeneracies of $Z$ and $Z'$ respectively. Therefore (3) holds, since if $Z'$ is maximally degenerate, so is $Z$.
\end{proof}

Consider log R-maps $f'$ and $f$ as in Proposition
\ref{prop:change-twists}.
Let $\ddata' = \{(\gamma'_i, c'_i)\}_{i=1}^n$ (resp.\
$\ddata = \{(\gamma_i, c_i)\}_{i=1}^n$) be the discrete data of $f'$
(resp.\ $f$) along markings.
Observe that $(\gamma_i, c_i)$ is uniquely determined by
$(\gamma'_i, c'_i)$ as follows.
First, since $\fP_{1,\bk} \to \fP_{2,\bk}$ is the $\frac{a_1}{a_2}$-th
root stack along $\infty_{\fP_{2,\bk}}$, the sector $\gamma_i$ is uniquely
determined by $\gamma'_i$ \cite[Section~1.1.10]{AbFa16}.
Then the morphism $\cC' \to \cC$ is an $\varrho_i$-th root stack along
the $i$-th marking for some $\varrho_i | \frac{a_1}{a_2}$ uniquely
determined by the natural morphism $\gamma'_i \to \gamma_i$
\cite[Lemma~1.1.11]{AbFa16}.
The contact orders $c_i$ and $c_i'$ are then related by
\begin{equation}\label{equ:changing-contact-order}
c_i = \frac{a_1}{a_2}\cdot \frac{c'_i}{\varrho_i}.
\end{equation}

\begin{corollary}\label{cor:changing-twists}
There are canonical morphisms
\[
\SH_{g,\ddata'}(\fP_1, \beta) \to \SH_{g,\ddata}(\fP_2,\beta) \ \ \ \mbox{and} \ \ \ \UH_{g,\ddata'}(\fP_1, \beta) \to \UH_{g,\ddata}(\fP_2,\beta).
\]
For convenience, we denote both morphisms by $\nu_{\etwist_1/\etwist_2}$ when there is no danger of confusion.
\end{corollary}

\section{A tale of two virtual cycles}
\label{sec:tale}

This section forms the heart of this paper.
We first introduce the canonical perfect obstruction theory and
virtual cycle in Section~\ref{ss:canonical-obs}, and prove a change of
twist theorem in this setting in
Section~\ref{ss:can-theory-ind-twist}.
We then introduce the compact type locus, its canonical virtual cycle
(Section~\ref{ss:compact-type}) and the superpotentials
(Section~\ref{ss:superpotential}), in preparation for defining the
canonical cosection in Section~\ref{ss:can-cosection}.
This allows us to construct the reduced theory in
Section~\ref{ss:reduced}.
We then prove the comparison theorems in
Sections~\ref{ss:comparison-1}--\ref{ss:red-theory-ind-twist}.

The first time reader may skip Sections~\ref{ss:can-theory-ind-twist}
and \ref{ss:red-theory-ind-twist} related to the change of twists.
In addition, under the further simplification that the set $\Sigma$ of markings is
empty, the reader may skip Sections~\ref{ss:compact-type},
\ref{ss:modify-target} and \ref{sss:superpotential-modified-target}.
In this situation, the sup- or subscripts, ``$\cpt$'', ``$\reg$'' and
``$-$'' may be dropped.

\subsection{The canonical theory}\label{ss:canonical-obs}

For the purposes of perfect obstruction theory and virtual fundamental
classes, we impose in this section:

\begin{assumption}\label{assu:smooth-target}
$\cX$ is smooth.
\end{assumption}
The assumption implies that $\fP \to \BC$ is log smooth with the smooth underlying morphism.

To simplify notations, we introduce
\[
\UH := \UH_{g, \ddata}(\fP, \beta), \ \ \ \  \ \ \SH := \SH_{g, \ddata}(\fP, \beta),
\]
for stacks of log R-maps as in Section \ref{ss:hybrid-stack}.
We also introduce
\[
\fU:= \fU_{g,\vec{c}}(\cA), \ \ \ \  \ \ \fM := \fM_{g,\vec{c}}(\cA)
\]
where $\fM_{g,\vec{c}}(\cA)$ (resp.\ $\fU_{g,\vec{c}}(\cA)$) is the stack parameterizing log
maps (resp.\ with uniform maximal degeneracy) to $\cA$ of genus $g$
and with contact orders $\vec{c}$ induced by $\ddata$.
These stacks fit in a cartesian diagram
\[
\xymatrix{
\UH \ar[rr]^{F} \ar[d] && \SH \ar[d] \\
\fU \ar[rr] && \fM
}
\]
where the vertical arrows are canonical strict morphisms by Section \ref{sss:universal-log-target}, the bottom is given by \cite[Theorem 3.14]{CJRS18P}, and the top is \eqref{equ:forget-uniform-degeneracy}.

Let $\bullet$ be one of the stacks $\UH, \SH, \fU$ or $\fM$, and $\pi_{\bullet}\colon \cC_{\bullet} \to \bullet$ be the universal curve. Denote by $\cP_{\bullet} := \cC_{\bullet}\times_{\BC}\fP$ where $\cC_{\bullet} \to \BC$ is induced by $\omega^{\log}_{\cC_{\bullet}/\bullet}$. Let $f_{\bullet}\colon \cC_{\bullet} \to \cP_{\bullet}$ be the  section induced by the universal log $R$-map for $\bullet = \UH$ or $\SH$.

Consider the commutative diagram
\[
\xymatrix{
\cC_{\SH} \ar@/^1pc/[rrd]^{=} \ar[rd]^{f_{\SH}} \ar@/_1pc/[ddr]_{f}&&& \\
& \cP_{\SH} \ar[r] \ar[d] & \cC_{\SH} \ar[r]^{\pi_{\SH}} \ar[d] & \SH \ar[d] \\
& \cP_{\fM} \ar[r] & \cC_{\fM} \ar[r]_{\pi_{\fM}} & \fM
}
\]
where the three vertical arrows are strict, and the two squares are
Cartesian.
We use $\LL$ to denote the log cotangent complexes in the sense of
Olsson \cite{LogCot}.
The lower and upper triangle yield
\[
\LL_{f_{\SH}} \to f^*_{\SH}\LL_{\cP_{\SH}/\cP_{\fM}}[1] \cong \pi^*_{\SH}\LL_{\SH/\fM}[1] \qquad \mbox{and} \qquad \LL_{f_{\SH}} \cong f^*_{\SH}\LL_{\cP_{\SH}/\cC_{\SH}}[1],
\]
respectively. 
Hence we obtain
\[f^*_{\SH}\LL_{\cP_{\SH}/\cC_{\SH}} \to \pi^*_{\SH}\LL_{\SH/\fM}.\]
Tensoring both sides by the dualizing complex
$\omega_{\pi_{\SH}}^\bullet = \omega_{\cC_{\SH}/\SH}[1]$ and applying
$\pi_{\SH,*}$, we obtain
\[
\pi_{\SH,*}\big(f^*_{\SH}\LL_{\cP_{\SH}/\cC_{\SH}}\otimes\omega_{\pi_{\SH}}^\bullet\big) \to \pi_{\SH,*}\pi^{!}_{\SH} \LL_{\SH/\fM} \to \LL_{\SH/\fM}
\]
where the last arrow follows from the fact that $\pi_{\SH,*}$ is left
adjoint to $\pi^{!}(-) := \omega_{\pi_{\SH}}^\bullet\otimes\pi^*(-)$.
Further observe that
$\LL_{\cP_{\SH}/\cC_{\SH}} = \Omega_{\fP/\BC}|_{\cP_\SH}$ is the log
cotangent bundle. Hence, we obtain
\begin{equation}\label{equ:log-POS-R}
\varphi^{\vee}_{\SH/\fM}\colon \EE^{\vee}_{\SH/\fM} := \pi_{\SH,*}\big(f^*_{\SH}\Omega_{\fP/\BC}\otimes\omega_{\pi_{\SH}}^\bullet\big) \to \LL_{\SH/\fM}.
\end{equation}
The same proof as in \cite[Proposition 5.1]{CJRS18P} shows that
$\varphi^{\vee}_{\SH/\fM}$ is a perfect obstruction theory of
$\SH \to \fM$ in the sense of \cite{BeFa97}.
Recall that $\fM$ is log smooth and equi-dimentional
\cite[Proposition~2.13]{CJRS18P}.
Denote by $[\SH]^{\vir}$ the virtual cycle given by the virtual
pull-back of the fundamental class $[\fM]$ using
$\varphi^{\vee}_{\SH/\fM}$.

Pulling back $\varphi^{\vee}_{\SH/\fM}$ along $\UH \to \SH$, we obtain a perfect obstruction theory of $\UH \to \fU$:
\begin{equation}\label{equ:log-POS}
\varphi^{\vee}_{\UH/\fU}\colon \EE^{\vee}_{\UH/\fU} := \pi_{\UH,*}\big(f^*_{\UH}\Omega_{\fP/\BC}\otimes\omega_{\pi_{\UH}}^\bullet\big) \to \LL_{\UH/\fU}.
\end{equation}
A standard calculation shows that $\EE_{\UH/\fU} = \pi_{\UH,*}\big(f^*_{\UH}T_{\fP/\BC}\big)$.

Let $[\UH]^{\vir}$ be the corresponding virtual cycle.
Since $\fU \to \fM$ is proper and birational by
\cite[Theorem~3.17]{CJRS18P}, by the virtual push-forward of
\cite{Co06, Ma12} the two virtual cycles are related by
\begin{equation}\label{equ:push-log-virtual-cycle}
F_*[\UH]^{\vir} = [\SH]^{\vir}.
\end{equation}

\subsection{Independence of twists I: the case of the canonical theory}\label{ss:can-theory-ind-twist}

In this section, using the results from
Section~\ref{ss:change-twists}, we study the behavior of the canonical
virtual cycle under the change of twists.

\begin{proposition}\label{prop:can-ind-twists}
  Given the situation as in Corollary \ref{cor:changing-twists}, we
  have the following push-forward properties of virtual cycles
\begin{enumerate}
 \item $\nu_{\etwist_1/\etwist_2,*}[\SH_{g,\ddata'}(\fP_1, \beta)]^\vir = [\SH_{g,\ddata}(\fP_2, \beta)]^\vir$,
 \item $\nu_{\etwist_1/\etwist_2,*}[\UH_{g,\ddata'}(\fP_1, \beta)]^\vir = [\UH_{g,\ddata}(\fP_2, \beta)]^\vir$.
\end{enumerate}
\end{proposition}

\begin{proof}
  We will only consider (1). Statement (2) can be proved identically
  by considering only log maps with uniform maximal degeneracy, thanks
  to Proposition \ref{prop:change-twists} (3).

Since $\fP_1 \to \fP_2$ is a log \'etale birational modification, (1) follows from a similar proof as in \cite{AbWi18} but in a simpler situation except that we need to take into account orbifold structures. In what follows, we will only specify the differences, and refer to \cite{AbWi18} for complete details.

First, consider the stack $\fM' := \fM_{g,\vec{c}'}'(\cA_1 \to \cA_2)$, the analogue of the one in \cite[Proposition 1.6.2]{AbWi18}, parameterizing commutative diagrams
\begin{equation}\label{eq:middle-stack}
\xymatrix{
\cC' \ar[r] \ar[d] & \cA_1 \ar[d] \\
\cC \ar[r] & \cA_2
}
\end{equation}
where $\cC' \to \cC$ is a morphism of log curves over $S$ inducing an
isomorphism of underlying coarse curves, the top and bottom are log maps with discrete data along markings given by $\vec{c}' = \{(r'_i, c'_i)\}$ (see Lemma \ref{lem:lift-root}) and $\vec{c} = \{(r_i, c_i)\}$ respectively, and the induced morphism
$\cC' \to \cC\times_{\cA_2}\cA_1$ is representable, hence is stable.

We first show that $\fM'$ is algebraic.
Indeed, let $\fM_{g,\vec{c}}(\cA)$ be the stack of genus $g$ log maps
to $\cA$ with discrete data $\vec{c}$ along markings.
Let $\fM_1$ be the stack parameterizing sequences
$\cC' \to \cC \to \cA_2$ where $\cC' \to \cC$ is a morphism of genus
$g$, $n$-marked log curves over $S$ with isomorphic underlying coarse
curves, and $\varrho_i$-th root stack along the $i$-th marking for each
$i$ \eqref{equ:changing-contact-order}.
$\fM_1$ is algebraic as the morphism $\fM_1 \to \fM_{g,\vec{c}}(\cA_2)$ defined by
\begin{equation}\label{eq:construct-mid-stack}
[\cC' \to \cC \to \cA_2] \mapsto [\cC \to \cA_2]
\end{equation}
is algebraic and $\fM_{g,\vec{c}}(\cA_2) = \fM_{g,\vec{c}}(\cA)$ is
algebraic.
Now $\fM'$ is given by the open substack of
$\fM_1\times_{\fM_{g,\vec{c}''}(\cA_2)}\fM_{g,\vec{c}'}(\cA_1)$ where
the representability of $\cC' \to \cC\times_{\cA_2}\cA_1$ holds.
Here
$\fM_{g,\vec{c}'}(\cA_1) = \fM_{g, \vec{c}'}(\cA) \to
\fM_{g,\vec{c}''}(\cA_2)$ is given by composing log maps to $\cA_1$
with $\cA_1 \to \cA_2$ hence
$\vec{c}'' = \{(r'_i, \frac{a_1}{a_2}\cdot c'_i)\}$, and
$\fM_1 \to \fM_{g,\vec{c}''}(\cA_2)$ is given by
$[\cC' \to \cC \to \cA_2] \mapsto [\cC' \to \cA_2]$.

Next, by Proposition~\ref{prop:change-twists} and
\eqref{diag:targets-twists}, we obtain the commutative diagram
\[
\xymatrix{
\SH_{g,\ddata'}(\fP_1, \beta) \ar[rr] \ar[d]^{G'_1} \ar@/_2pc/[dd]_{G_1}&& \SH_{g,\ddata}(\fP_2, \beta) \ar[d]^{G_2} \\
\fM' \ar[rr]_{F_2} \ar[d]^{F_1} && \fM_{g,\vec{c}}(\cA_2) \\
\fM_{g,\vec{c}'}(\cA_1) &&
}
\]
where we define a morphism $F_1 \colon \eqref{eq:middle-stack} \mapsto [\cC' \to \cA_1]$ and a proper morphism $F_2 \colon \eqref{eq:middle-stack} \mapsto [\cC \to \cA_2]$, and the square is cartesian. 

Since the horizontal arrows in \eqref{diag:targets-twists} are logarithmic modifications in the sense of \cite{AbWi18}, the same proof as in \cite[Lemma 4.1, Section 5.1]{AbWi18} shows that $\fM' \to \fM_{g,\vec{c}'}(\cA_1)$ is strict and \'etale. Using the identical method as in \cite[Section 6.2]{AbWi18}, one construct a perfect obstruction theory of $G'_1$ which is identical to the one of $G_1$ as in \eqref{equ:log-POS-R}. 

Furthermore Lemma \ref{lem:lift-root} and the same proof as in \cite[Proposition 5.2.1]{AbWi18} imply that $\fM' \to \fM_{g,\vec{c}'}(\cA_1)$ and $\fM' \to \fM_{g,\vec{c}}(\cA_2)$ are both birational. Finally, following the same lines of proof as in  \cite[Section 6]{AbWi18} and using Costello's virtual push-forward \cite[Theorem 5.0.1]{Co06}, we obtain (1).
\end{proof}

Since log maps to $\cA$ are unobstructed \cite[Proposition 1.6.1]{AbWi18} (see also \cite[Proposition 2.13]{CJRS18P}), the discrete data along markings can be determined by studying the following non-degenerate situation.

\begin{lemma}\label{lem:lift-root}
  Let $f \colon \cC \to \cA_2$ be a log map with discrete data
  $(r_i, c_i)$ at the $i$-th marking.
  Assume that no component of $\cC$ has image entirely contained in
  $\infty_{\cA_2}$.
  Let $\cC' \to \cC$ be obtained by taking the $\varrho_i$-th root
  along the $i$-th marking for each $i$. Then
\begin{enumerate}
\item $f \colon \cC \to \cA_2$ lifts to $f' \colon \cC' \to \cA_1$ if and only if $\frac{\etwist_1}{\etwist_2} | c_i \cdot \varrho_i$. In this case, the lift is unique up to a unique isomorphism.
\item Furthermore, the induced $\cC' \to \cC\times_{\cA_2}\cA_1$ by
  $f'$ is representable if and only if
  $\varrho_i = \frac{a_1/a_2}{\gcd(c_i,a_1/a_2)}$.
  In this case, let $(r'_i, c'_i)$ be the discrete data of $f_i$ at
  the $i$-th marking.
  Then for each $i$ we have
\[
r'_i = \varrho_i \cdot r_i \ \ \ \mbox{and} \ \ \ c'_i = \frac{c_i}{\gcd(c_i,\etwist_1/\etwist_2)}.
\]
\end{enumerate}
\end{lemma}
\begin{proof}
  Finding a lift $f'$ amounts to finding
  $\cC' \to \cC\times_{\cA_2}\cA_1$ that lifts the identity
  $\cC \to \cC$.
  Thus, both (1) and (2) follow from \cite[Lemma~1.3.1]{AbFa16} and
  \cite[Theorem~3.3.6]{Ca07}.
\end{proof}

\subsection{The compact type locus and its canonical virtual cycle}\label{ss:compact-type}
We next introduce the closed substack over which the reduced theory will be constructed.

\subsubsection{The logarithmic evaluation stacks}\label{sss:log-ev-stack}
Let $\fY = \fM$ (resp.\ $\fU$), and $\scrY = \SH$ (resp.\ $\UH$) with the strict
canonical morphism $\scrY \to \fY$.
The \emph{$i$-th evaluation stack} $\fY^{\ev}_i$ associates to any
$\fY$-log scheme $S$ the category of commutative diagrams:
\begin{equation}\label{diag:log-evaluation}
\xymatrix{
p_i \ar[rr] \ar[d] && \fP_{\bk} \ar[d] \\
\cC \ar[rr]^{\omega^{\log}_{\cC/S}}  &&\cA 
}
\end{equation}
where $\cC \to \cA$ is the log map over $S$ given by $S \to \fY$,
$p_i \subset \cC$ is the $i$-th marking with the pull-back log
structure from $\cC$, and the top horizontal arrow is representable. 

There is a canonical strict morphism $ \fY^{\ev}_i \to \fY $
forgetting the top arrow in \eqref{diag:log-evaluation}.  
By Lemma~\ref{lem:underlying-evaluation}, the morphism $\scrY \to \fY$
factors through the \emph{$i$-th evaluation morphism}
\[
\fev_i \colon \scrY \to \fY^{\ev}_i.
\]
For the reduced theory, we introduce substacks
\[
\fY^{\cpt}_i \subset \fY^{\mathring{\ev}}_i \subset \fY^{\ev}_i,
\]
where $\fY^{\mathring{\ev}}_i$ parameterizes diagrams \eqref{diag:log-evaluation} whose images at the $i$-th markings avoid $\infty_{\cA}$ (or equivalently avoids $\infty_{\fP}$) and $\fY^{\cpt}_i$ parameterizes diagrams \eqref{diag:log-evaluation} whose images at the $i$-th markings are contained in $\zero_{\fP}$. Recall that $(\gamma_i, c_i)$ are the sector and contact order at the $i$-th marking, see Section \ref{ss:hybrid-stack}.  

\begin{proposition}\label{prop:compact-evaluation}
Both $\fY^{\cpt}_i$ and  $\fY^{\mathring{\ev}}_i$ are log algebraic stacks. Furthermore, we have
\begin{enumerate}
 \item If $c_i > 0$, then $\fY^{\cpt}_i = \fY^{\mathring{\ev}}_i = \emptyset$.
 \item The strict morphisms $\fY^{\cpt}_{i} \to \fY$ and $\fY^{\mathring{\ev}}_{i} \to \fY$ are smooth.
\end{enumerate}
\end{proposition}

\begin{remark}
 $\fY^{\ev}_i$ is also algebraic, but we do not need this fact here.
\end{remark}

\begin{proof}
(1) follows from the definition of $\fY^{\cpt}_i$ and $\fY^{\mathring{\ev}}_i$. We now assume $c_i = 0$.

Let $\fY_{i,0} \subset \fY$ be the open dense sub-stack over which the
image of $p_i$ avoids $\infty_{\cA}$.
Let $\ocI \fP_\bk^\circ \subset \ocI \fP_\bk$ be the open substack
parameterizing gerbes avoiding $\infty_{\fP_{\bk}}$.
By \eqref{diag:log-evaluation}, it follows that
\begin{equation}\label{equ:open-type-loci}
\fY^{\mathring{\ev}}_{i} = \fY_{i,0}\times \ocI \fP_\bk^\circ
\end{equation}
hence $\fY^{\mathring{\ev}}_{i}$ is algebraic. Similarly $\fY^{\cpt}_i$ is a closed substack of $\fY^{\mathring{\ev}}_{i}$ given by 
\begin{equation}\label{equ:compact-type-loci}
\fY^{\cpt}_{i} = \fY_{i,0}\times \ocI \zero_{\fP_\bk},
\end{equation}
hence is also algebraic.
(2) follows from the smoothness of $\ocI \zero_{\fP_\bk}$ and
$\ocI \fP_\bk^\circ$.
\end{proof}

Consider the following fiber products both taken over $\fY$:
\begin{equation}\label{equ:compact-type-universal-stack}
\fY^{\cpt} := \prod_i \fY^{\cpt}_{i} \ \ \  \mbox{and} \ \ \ \fY^{\mathring{\ev}} := \prod_i \fY^{\mathring{\ev}}_i,
\end{equation}
where $i$ runs through all markings with contact order zero.
Consider fiber products
\begin{equation}\label{equ:compact-type-stack}
\scrY^\cpt := \scrY \times_{\fY^{\ev}}\fY^{\cpt} \ \ \ \mbox{and} \ \ \   \scrY^{\mathring{\ev}}:= \scrY \times_{\fY^{\ev}}\fY^{\mathring{\ev}}.
\end{equation}
Then $\scrY^{\mathring{\ev}} \subset \scrY$ (resp.\
$\scrY^\cpt \subset \scrY$) is the open (resp.\ closed) sub-stack
parameterizing stable log R-maps whose images at markings with the zero contact order avoid
$\infty_{\fP}$ (resp.\ in $\zero_{\fP}$).

\subsubsection{The canonical perfect obstruction theory of $\scrY^\cpt \to \fY^{\cpt}$}
Consider the universal map and projection over $\fY^{\mathring{\ev}}$ respectively:
\[
\fev \colon \cup_i p_i \to \fP \ \ \ \mbox{and} \ \ \ \pi_{\ev}\colon \cup_{i}p_i \to \fY^{\mathring{\ev}}.
\]
By \eqref{equ:open-type-loci}, \eqref{equ:compact-type-universal-stack} and \cite[Lemma 3.6.1]{AGV08}, we have an isomorphism of vector bundles
\[
\varphi^{\vee}_{\fY^{\mathring{\ev}}/\fY}\colon \EE^{\vee}_{\fY^{\mathring{\ev}}/\fY} := (\pi_{\ev,*} \fev^*T_{\fP/\BC})^{\vee} \stackrel{\cong}{\longrightarrow} \LL_{\fY^{\mathring{\ev}}/\fY}.
\]

The perfect obstruction theory \eqref{equ:log-POS} restricts to a relative perfect obstruction theory 
\[
\varphi^{\vee}_{\scrY^{\mathring{\ev}}/\fY}\colon \EE^{\vee}_{\scrY^{\mathring{\ev}}/\fY} := \EE^{\vee}_{\SH/\fM}|_{\scrY^{\mathring{\ev}}} \to \LL_{\scrY^{\mathring{\ev}}/\fY}.
\]
A standard construction as in  \cite[A.2]{BrLe00}  or \cite[Proposition 4.4, Lemma 4.5]{ACGS20P} yields a morphism of triangles
\begin{equation}\label{diag:compatible-obs}
\xymatrix{
\ev^*\EE^{\vee}_{\fY^{\mathring{\ev}}/\fY} \ar[r] \ar[d]_{\varphi^{\vee}_{\fY^{\mathring{\ev}}/\fY}} & \EE^{\vee}
_{\scrY^{\mathring{\ev}}/\fY} \ar[r] \ar[d]_{\varphi^{\vee}_{\scrY^{\mathring{\ev}}/\fY}} & \EE^{\vee}_{\scrY^{\mathring{\ev}}/\fY^{\mathring{\ev}}} \ar[d]_{\varphi^{\vee}_{\scrY^{\mathring{\ev}}/\fY^{\mathring{\ev}}}} \ar[r]^-{[1]} & \\
\ev^*\LL_{\fY^{\mathring{\ev}}/\fY} \ar[r] & \LL_{\scrY^{\mathring{\ev}}/\fY} \ar[r] & \LL_{\scrY^{\mathring{\ev}}/\fY^{\mathring{\ev}}} \ar[r]^-{[1]} &
}
\end{equation}
where $\varphi^{\vee}_{\scrY^{\mathring{\ev}}/\fY^{\mathring{\ev}}}$ is a perfect obstruction theory of $\scrY^{\mathring{\ev}} \to \fY^{\mathring{\ev}}$ of the form
\[
\varphi^{\vee}_{\scrY^{\mathring{\ev}}/\fY^{\mathring{\ev}}}\colon \EE^{\vee}_{\scrY^{\mathring{\ev}}/\fY^{\mathring{\ev}}} := \pi_{\scrY^{\mathring{\ev}},*} f_{\scrY^{\mathring{\ev}}}^*(T_{\fP/\BC}(- \Sigma))^{\vee} \to \LL_{\scrY^{\mathring{\ev}}/\fY^{\mathring{\ev}}}.
\]
Here $\Sigma$ is the sum or union of all markings with the zero contact order. Thus, the two perfect obstruction theories $\varphi^{\vee}_{\scrY^{\mathring{\ev}}/\fY}$ and $\varphi^{\vee}_{\scrY^{\mathring{\ev}}/\fY^{^{\mathring{\ev}}}}$ are compatible in the sense of \cite{BeFa97}.

Pulling back $\varphi^{\vee}_{\scrY^{\mathring{\ev}}/\fY^{\mathring{\ev}}}$ to $\scrY^{\cpt}$ we obtain the {\em canonical perfect obstruction theory} of $\scrY^{\cpt} \to \fY^{\cpt}$
\begin{equation}\label{equ:obs-compact-evaluation}
\varphi^{\vee}_{\scrY^{\cpt}/\fY^{\cpt}}\colon \EE^{\vee}_{\scrY^{\cpt}/\fY^{\cpt}} := \pi_{\scrY^{\cpt},*} f_{\scrY^{\cpt}}^*(T_{\fP/\BC}(- \Sigma))^{\vee} \to \LL_{\scrY^{\cpt}/\fY^{\cpt}}.
\end{equation}
Denote by $[\scrY^{\cpt}]^{\vir}$ the {\em canonical virtual cycle} of $\scrY^{\cpt}$ defined via \eqref{equ:obs-compact-evaluation}. 

\subsubsection{The compact type locus}\label{sss:compact-type-loci}
We call $\scrY^\cpt \subset \scrY$ the {\em compact type locus} if the contact orders of all markings are equal to zero. 
For the purpose of constructing the reduced virtual cycles over the compact type locus, we will impose for the rest of this section that 

\begin{assumption}\label{assu:zero-contacts}
All contact orders are equal to zero.
\end{assumption}

This assumption is needed in our construction to apply the cosection localization of Kiem-Li \cite{KiLi13}. 
In this case $\ddata$ is the same as a collection of log sectors which
are further restricted to be sectors of compact type (see
Definition~\ref{def:hyb-sector}).
Note that if all sectors are narrow, then $\scrY^\cpt = \scrY$.

\subsection{The superpotentials}\label{ss:superpotential}
Our next goal is to construct the reduced virtual cycle $[\scrY^{\cpt}]^{\red}$ for $\scrY^{\cpt} = \UH^{\cpt}$. The superpotential is a key ingredient which we discuss now.

\subsubsection{The definition}
A \emph{superpotential} is a morphism of stacks
\[
W\colon \fP^{\circ} \to \LR
\]
over $\BC$. Equivalently, $W$ is a section of the line bundle $\LR|_{\fP^{\circ}}$ over $\fP^{\circ}$.

Pulling-back $W$ along the universal torsor $\spec\bk \to \BC$, we obtain a $\CC^*_\omega$-equivariant function
\[
W_\bk \colon \fP^{\circ}_\bk \to \bk
\]
which recovers the information of $W$. 
Denote by $\crit(W_{\bk}) \subset \fP^{\circ}_{\bk}$ the {\em critical locus} of the holomorphic function $W_\bk$. It descends to a closed substack $\crit(W) \subset \fP^{\circ}$. 

\begin{definition}
  We call $\crit(W)$ the {\em critical locus} of the superpotential
  $W$. 
  We say that $W$ has \emph{proper critical locus} if $\crit(W_{\bk})$
  is proper over $\bk$, or equivalently $\crit(W)$ is proper over
  $\BC$.
\end{definition}

Let $\fX_\bk := \fX\times_{\BC}\spec\bk$ where the left arrow is
given by $\zeta$ in \eqref{equ:hyb-target} and the right one is the
universal torsor.
Since $\fP^{\circ}_\bk$ is a vector bundle over $\fX_\bk$, the
critical locus of $W_\bk$, if proper, is necessarily supported on
the fixed locus $\zero_{\fP_\bk} \subset \fP_\bk$ of the
$\CC^*_\omega$-action.

\subsubsection{The extended superpotential}
To extend $W$ to $\fP$, we first observe the following:

\begin{lemma}\label{lem:pole-of-potential}
Suppose there exists a non-zero superpotential $W$. Then the order of poles of $W_{\bk}$ along $\infty_{\fP_{\bk}}$ is the positive integer $\ttwist = a \cdot r$.
\end{lemma}
\begin{proof}
The existence of non-zero $W$ implies that there is a sequence of non-negative integers $k_i$ such that
\[
r = \sum_i k_i \cdot i,
\]
where $i$ runs through the grading of non-trivial $\bE_i$. The integrality of $\ttwist$ follows from the choices of $a$, and the order of poles of $W_{\bk}$ follows from the choices of weights $\bw$ in \eqref{equ:universal-proj}. 
\end{proof}

Consider the  $\PP^1$-bundle over $\BC$
\[
\PP_\omega = \PP(\LR\oplus\cO).
\]
We further equip $\PP_\omega$ with the log structure given by its reduced
infinity divisor $\infty_{\PP_\omega} := \PP_\omega \setminus \vb(\LR)$.
The superpotential $W$ extends to a rational map of log stacks
$\oW\colon \fP \dashrightarrow \PP_\omega$ over $\BC$ with the
indeterminacy locus $\overline{(W^{-1}(\zero_{\LR}))}\cap \infty_{\fP}$ by Lemma \ref{lem:pole-of-potential}. 

Equivalently, $\oW$ can be viewed as a rational section of $\LR|_{\fP^{\circ}}$ extending $W$, and having poles along $\infty_{\fP}$ of order $\ttwist$.

\subsubsection{The twisted superpotential}
Next, we discuss how to extend the superpotential $W$ across the
boundary.
This will be shown to be the key to extending cosections to the
boundary of the log moduli stacks.
It should be noticed that the non-empty indeterminacy locus of $\oW$
is a new phenomenon compared to the $r$-spin case \cite{CJRS18P}, and
requires a somewhat different treatment as shown below.

Consider the log \'etale morphism of log stacks
\begin{equation}\label{equ:partial-exp}
\cA^e \to \cA\times\cA
\end{equation}
given by the blow-up of the origin of $\cA\times\cA$.
Denote by $\fP^e$ and $\PP^e_\omega$ the pull-back of
\eqref{equ:partial-exp} along the following respectively
\[
\fP\times\cA_{\max} \stackrel{(\cM_{\fP}, id)}{\longrightarrow} \cA\times\cA_{\max} \ \ \ \mbox{and} \ \ \ \PP_\omega\times\cA_{\max} \stackrel{(\cM_{\PP_\omega}, \nu_{\ttwist})}{\longrightarrow} \cA\times \cA
\]
Here $\cA_{\max} = \cA$,  
and $\nu_{\ttwist}$ is the degree $\ttwist$ morphism induced by $\NN \to \NN, \ 1 \mapsto \ttwist$ on the level of characteristics. Recall from Lemma \ref{lem:pole-of-potential} that $\ttwist$ is a positive integer given $W \neq 0$. 

Denote by $\infty_{\fP^e} \subset \fP^e$ and $\infty_{\PP^e} \subset \PP^e_\omega$ the proper transforms of $\infty_{\fP}\times\cA_{\max}$ and $\infty_{\PP_{\omega}}\times\cA_{\max}$ respectively. Consider
\begin{equation}\label{equ:remove-infinity}
\fP^{e,\circ} := \fP^{e}\setminus \infty_{\fP^e} \ \ \ \mbox{and} \ \ \ \PP^{e,\circ}_\omega := \PP^e_\omega \setminus \infty_{\PP^e_\omega}.
\end{equation}
We obtain a commutative diagram with rational horizontal maps
\[
\xymatrix{
\fP^{e,\circ} \ar[d] \ar@{-->}[rr]^{\oW^{e,\circ}} && \PP_\omega^{e,\circ} \ar[d] \\
\fP\times\cA_{\max} \ar@{-->}[rr]^{\oW\times id} && \PP_\omega\times\cA_{\max}
}
\]

\begin{lemma}
There is a canonical surjective log morphism
\[
\fc\colon \PP^{e,\circ}_\omega \to \vb\big(\LR\boxtimes\cO_{\cA_{\max}}(\ttwist\Delta_{\max})\big)
\]
by contracting the proper transform of $\PP_\omega\times \Delta_{\max}$ where $\Delta_{\max} \subset \cA_{\max}$ is the closed point, and the target of $\fc$ is equipped with the pull-back log structure from $\cA_{\max}$.
\end{lemma}
\begin{proof}
This follows from a local coordinate calculation.
\end{proof}

\begin{proposition}\label{prop:twisted-potential}
The composition $\tW := \fc\circ\oW^{e,\circ}$ is a surjective morphism that contracts the proper transform of $\fP\times\Delta_{\max}$.
\end{proposition}
\begin{proof}
  A local calculation shows that the proper transform of
  $\fP\times\Delta_{\max}$ dominates the proper transform of
  $\PP^{e,\circ}_\omega\times\Delta_{\max}$, hence is contracted by $\fc$.
  The surjectivity of $\oW^{e,\circ}$ follows from the pole order of Lemma \ref{lem:pole-of-potential} and the above construction. Hence the surjectivity in the statement follows from the surjectivity of $\fc$ in the above lemma.

  It remains to show that $\tW$ is well-defined everywhere. Let
  $E^{\circ} \subset \fP^{e,\circ}$ be the exceptional divisor of
  $\fP^{e,\circ} \to \fP\times \cA$.
  Then $E^{\circ} \cong N_{\infty_{\fP/\fP}}$ is the total space
  of the normal bundle.
  The indeterminacy locus of $\oW^{e,\circ}$ is the fiber of
  $E^{\circ} \to \infty_{\fP}$ over
  $\overline{(W^{-1}(\zero_{\LR}))}\cap \infty_{\fP}$.
  One checks that $\tW$ contracts the indeterminacy locus of
  $\oW^{e,\circ}$ to the zero section of its target.
\end{proof}

\begin{definition}\label{def:twisted-potential-non-deg}
  We call $\tW$ the \emph{twisted superpotential}.\footnote{This is
    different from the ``twisted superpotential'' used in the
    physics literature \cite[(2.27)]{Wi93}.}
  It is said to have \emph{proper critical locus} if the vanishing
  locus of the log differential $\diff \tW$, defined as a closed
  strict substack of $\fP^{e,\circ}$, is proper over
  $\BC \times \cA_{\max}$.
\end{definition}

\begin{proposition}
  $\tW$ has proper critical locus iff $W$ has proper critical locus.
\end{proposition}
\begin{proof}
Since $W$ is the fiber of $\tW$ over the open dense point of $\cA_{\max}$, one direction is clear. We next assume that $W$ has proper critical locus.

  Consider the substack $\fP^{e,*} \subset \fP^{e,\circ}$ obtained by
  removing the zero section $\zero_{\fP^{e}}$ and the proper transform of
  $\fP\times_{\cA_{\max}}\Delta_{\max}$.
  Since the proper transform of $\fP\times_{\cA_{\max}} \Delta_{\max}$ is proper over $\BC\times\cA_{\max}$, it suffices to show that the morphism
  \[
  \tW|_{\fP^{e,*}}\colon \fP^{e,*} \to \vb\big(\LR\boxtimes\cO_{\cA_{\max}}(\ttwist\Delta_{\max})\big)
  \]
  has no critical points fiberwise over $\BC\times\cA_{\max}$, as
  otherwise the critical locus would be non-proper due to the
  $\CC^*$-scaling of $\fP^{e,*}$.

On the other hand, $\fP^{e,*}$ can be expressed differently as follows
  \[
  \fP^{e,*} = \vb\big(\bigoplus_{i > 0}(\bE^{\vee}_{i,\fX}\otimes\Sp^{\otimes i}\boxtimes \cO_{\cA_{\max}}(\etwist i \Delta_{\max}))\big)\setminus \zero
  \]
  where $\zero$ is the zero section of the corresponding vector bundle.
  Note that $W$ induces a morphism over $\BC\times\cA_{\max}$:
  \[
 \vb\big(\bigoplus_{i > 0}(\bE^{\vee}_{i,\fX}\otimes\Sp^{\otimes i}\boxtimes \cO(\etwist i \Delta_{\max}))\big) \to \vb\big(\LR\boxtimes\cO(\ttwist \Delta_{\max}) \big).
  \]
  whose restriction to $\fP^{e,*}$ is precisely $\tW|_{\fP^{e,*}}$.
  Since $\crit(W) $ is contained in the zero section, $\tW|_{\fP^{e,*}}$ has no
  critical points on $\fP^{e,*}$.
\end{proof}

\subsection{The canonical cosection}\label{ss:can-cosection}
Next we construct the canonical cosection for the moduli of log
R-maps.
For this purpose, we adopt the assumptions in Section
\ref{sss:compact-type-loci} by assuming all contact orders are zero
and working with the compact type locus for the rest of this section.
Furthermore, in order for the canonical cosection to behave well along
the boundary of the moduli, it is important to work with log R-maps
with uniform maximal degeneracy, see Section \ref{sss:UMD}.
As already exhibited in the r-spin case \cite{CJRS18P}, this will be
shown to be the key to constructing the reduced theory in later
sections in the general case.

\subsubsection{Modifiying the target}\label{ss:modify-target}

We recall the short-hand notation $\fU^{\cpt}$ and $\UH^\cpt$ as in
\eqref{equ:compact-type-stack} and
\eqref{equ:compact-type-universal-stack}. Consider the universal log $R$-map and the projection over $\UH^\cpt$ respectively:
\[
f_{\UH^\cpt}\colon \cC_{\UH^\cpt} \to \fP \ \ \ \mbox{and} \ \ \ \pi\colon \cC_{\UH^\cpt} \to \UH^\cpt.
\]
Denote by
$f_{\UH^\cpt}\colon \cC_{\UH^\cpt} \to \cP_{\UH^\cpt} :=
\fP\times_{\BC}\cC_{\UH^\cpt}$ again for the corresponding section.
To obtain a cosection, we modify the target $\cP_{\UH^\cpt}$ as
follows.

Consider $\fX_{\UH^\cpt} := \cC_{\UH^\cpt}\times_{\omega^{\log},\BC,\zeta}\fX$.
Recall $\Sigma$ is the sum of all markings. We define $\cP_{\UH^\cpt,-}$ to be the log stack with the underlying stack
\begin{equation}\label{equ:modified-target}
\ul{\cP_{\UH^\cpt,-}} := \ul{\PP}^{\bw}\left(\bigoplus_{i > 0}(\bE^{\vee}_i|_{\fX_{\UH^\cpt}}\otimes\cL_{\fX}^{\otimes i}|_{\fX_{\UH^\cpt}}(-\Sigma))\oplus \cO_{\fX_{\UH^\cpt}} \right).
\end{equation}
The log structure on $\cP_{\UH^\cpt,-}$ is defined to be the direct
sum of the log structures from the curve $\cC_{\UH^\cpt}$ and the
Cartier divisor $\infty_{\cP_{\UH^\cpt,-}}$
similar to $\cP_{\UH^\cpt}$ in Section \ref{ss:canonical-obs}.

Denote by $\cP^{\circ}_{\UH^\cpt,-} = \cP_{\UH^\cpt,-} \setminus \infty_{\cP_{\UH^\cpt,-}}$. We have a morphism of vector bundles over $\fX_{\UH^\cpt}$
\[
\cP^{\circ}_{\UH^\cpt,-} \to \cP^{\circ}_{\UH^\cpt}
\]
which contracts the fiber over $\Sigma$, and is
isomorphic everywhere else.
This extends to a birational map
\[
\cP_{\UH^\cpt,-} \dashrightarrow \cP_{\UH^\cpt}
\]
whose indeterminacy locus is precisely
$\infty_{\cP_{\UH^\cpt,-}}|_{\Sigma}$. Denote by 
\[
\cP_{\UH^\cpt,\reg} = \cP_{\UH^\cpt,-} \setminus \infty_{\cP_{\UH^\cpt,-}}|_{\Sigma}.
\]

\begin{lemma}
There is a canonical factorization
\begin{equation}\label{equ:hmap-modify-target}
\xymatrix{
\cC_{\UH^\cpt} \ar[rr]^{f_{\UH^\cpt}} \ar[rd]_{f_{\UH^\cpt,-}} && \cP_{\UH^\cpt} \\
&\cP_{\UH^\cpt,\reg} \ar[ru]&
}
\end{equation}
\end{lemma}
\begin{proof}
  Note that $f_{\UH^\cpt,-}$ and $f_{\UH^\cpt}$ coincide when restricted
  away from $\Sigma$.
  The lemma follows from the constraint 
  $f_{\UH^\cpt}(\Sigma) \subset \zero_{\cP_{\UH^\cpt}}$ of the compact type locus. 
\end{proof}

The following lemma will be used to show the compatibility of perfect
obstruction theories constructed in \eqref{equ:obs-compact-evaluation}
and in \cite{CJW19P}.
\begin{lemma}
There is a canonical exact sequence 
\begin{equation}\label{equ:modified-tangent}
0 \to f^*_{\UH^\cpt}T_{\fP/\BC}(-\Sigma) \to f^*_{\UH^\cpt,-}T_{\cP_{\UH^\cpt,\reg}/\cC_{\UH^\cpt}} \to T_{\fX/\BC}|_{\Sigma} \to 0.
\end{equation}
\end{lemma}
\begin{proof}
Consider the following commutative diagram of solid arrows
\[
\xymatrix{
0 \ar[r] & T_{\fP/\fX}|_{\cC_{\UH^\cpt}}(-\Sigma) \ar[r] \ar[d]^{\cong} & T_{\fP/\BC}|_{\cC_{\UH^\cpt}}(-\Sigma) \ar[r] \ar@{-->}[d]  & T_{\fX/\BC}|_{\cC_{\UH^\cpt}}(-\Sigma) \ar[r] \ar@{_{(}->}[d] & 0 \\
0 \ar[r] & T_{\cP_{\UH^\cpt,\reg}/\fX_{\UH^\cpt}}|_{\cC_{\UH^\cpt}} \ar[r] \ar@{_{(}->}[d] & T_{\cP_{\UH^\cpt,\reg}/\cC_{\UH^\cpt}}|_{\cC_{\UH^\cpt}} \ar[r] \ar@{_{(}->}[d]  & T_{\fX/\BC}|_{\cC_{\UH^\cpt}} \ar[r] \ar[d]^{\cong} \ar[r] & 0 \\
0 \ar[r] & T_{\fP/\fX}|_{\cC_{\UH^\cpt}} \ar[r] & T_{\fP/\BC}|_{\cC_{\UH^\cpt}} \ar[r] & T_{\fX/\BC}|_{\cC_{\UH^\cpt}} \ar[r] & 0
}
\]
where the horizontal lines are exact, the top exact sequence is the twist of the bottom one, and the lower middle vertical arrow is induced by \eqref{equ:hmap-modify-target}. Note that the sheaves in the first two columns are naturally viewed as sub-sheaves of $T_{\fP/\BC}|_{\cC_{\UH^\cpt}}$. The injection $T_{\fX/\BC}|_{\cC_{\UH^\cpt}}(-\Sigma) \hookrightarrow T_{\fX/\BC}|_{\cC_{\UH^\cpt}}$ on the upper right corner can be viewed as an inclusion of quotients by the same sub-bundle
\[
T_{\fP/\BC}|_{\cC_{\UH^\cpt}}(-\Sigma) \big/T_{\fP/\fX}|_{\cC_{\UH^\cpt}}(-\Sigma)  \subset T_{\cP_{\UH^\cpt,\reg}/\cC_{\UH^\cpt}}|_{\cC_{\UH^\cpt}}\big/T_{\fP/\fX}|_{\cC_{\UH^\cpt}}(-\Sigma),
\]
which lifts to $T_{\fP/\BC}|_{\cC_{\UH^\cpt}}(-\Sigma) \subset T_{\cP_{\UH^\cpt,\reg}/\cC_{\UH^\cpt}}|_{\cC_{\UH^\cpt}}$ by Lemma~\ref{lem:comm-alg} below. This defines the dashed arrow.

Finally, \eqref{equ:modified-tangent} follows from combinning the following exact sequence to the top two rows in the above commutative diagram
\[
0 \to T_{\fX/\BC}|_{\cC_{\UH^\cpt}}(-\Sigma) \to T_{\fX/\BC}|_{\cC_{\UH^\cpt}} \to T_{\fX/\BC}|_{\Sigma} \to 0.
\]
\end{proof}

\begin{lemma}
  \label{lem:comm-alg}
  Suppose $R$ is a commutative ring, and $A, B, C$ are submodules of
  an $R$-module $M$ satisfying $A \subset B$, $A \subset C$ and
  $B/A \subset C/A$ as submodules of $M/A$.
  Then $B \subset C$ as submodules of $M$.
\end{lemma}
\begin{proof}
The proof is left to the reader. 
\end{proof}

\subsubsection{The boundary of the moduli stacks}
Recall from \cite[Section 3.5]{CJRS18P} that the maximal degeneracy induces canonical morphisms to $\cA_{\max}$
\[
\UH^\cpt \to \fU^{\cpt} \to \fU \to \cA_{\max}.
\]
Consider the Cartier divisors 
\[
\Delta_{\fU} = \Delta_{\max}\times_{\cA_{\max}}\fU \subset \fU \ \  \ \mbox{and} \ \ \ \Delta_{\fU^{\cpt}} = \Delta_{\max}\times_{\cA_{\max}}\fU^{\cpt} \subset \fU^{\cpt}
\] 
and their pre-image $\Delta_{\UH^\cpt} \subset \UH^\cpt$. Hence we have the line bundle 
\[
\bL_{\max} = \cO_{\fU^{\cpt}}(-\Delta_{\fU^{\cpt}}) = \cO_{\cA_{\max}}(-\Delta_{\max})|_{\fU^{\cpt}}.
\]

\begin{definition}
We call $\Delta_{\fU}$ (resp.\ \ $\Delta_{\UH^\cpt}$ and
$\Delta_{\fU^{\cpt}} $) the \emph{boundary of maximal degeneracy} of
$\fU$ (resp.\ \ $\UH^\cpt$ and $\fU^{\cpt}$).  
We further introduce the \emph{interiors}
\begin{equation}\label{equ:interior-stack}
\IR^\cpt := \UH^\cpt\setminus (\Delta_{\UH^\cpt}) \ \ \ \mbox{and} \ \ \ \IU^{\cpt} := \fU^{\cpt}\setminus (\Delta_{\fU^{\cpt}})
\end{equation}
\end{definition}
By construction, $\IR^\cpt$ (resp.\ \ $\IU^{\cpt}$) parameterizes stable
log $R$-maps (resp.\ log maps) whose image avoids $\infty_{\fP}$
(resp.\ avoids $\infty_{\cA}$).
In this case, $\IU^{\cpt}$ is the stack of pre-stable curves since all
maps to $\cA$ factor through its unique open dense point.
In particular, $\IU^{\cpt}$ is smooth and log smooth.

\subsubsection{The twisted superpotential over the modified target}\label{sss:superpotential-modified-target}

Consider the two morphisms
\[
\cP_{\UH^\cpt,\reg} \to \cA\times\cA_{\max} \ \ \ \mbox{and} \ \ \ \cP_{\UH^\cpt} \to \cA\times\cA_{\max}
\]
where the morphisms to the first copy $\cA$ are induced by their infinity divisors. Pulling back \eqref{equ:partial-exp} along the above two morphisms, we obtain
\[
\cP^e_{\UH^\cpt,\reg} \to \cP_{\UH^\cpt,\reg}  \ \ \ \mbox{and} \ \ \ \cP^e_{\UH^\cpt} \to \cP_{\UH^\cpt}
\]
Further removing the proper transforms of
their infinity divisors from both, we obtain
$\cP^{e,\circ}_{\UH^\cpt,\reg}$ and $\cP^{e,\circ}_{\UH^\cpt}$.
Note that that
$\cP^{e,\circ}_{\UH^\cpt} \cong
\fP^{e,\circ}\times_{\BC}\cC_{\UH^\cpt}$.
Consider the short-hand
\begin{equation}\label{equ:twisted-omega}
\tomega := \omega_{\cC_{\UH^\cpt}/\UH^\cpt}\otimes \pi^*\bL_{\max}^{-\ttwist} \ \ \ \mbox{and} \ \ \ \tomega_{\log} := \omega^{\log}_{\cC_{\UH^\cpt}/\UH^\cpt}\otimes  \pi^*\bL_{\max}^{-\ttwist}|_{\UH^\cpt}
\end{equation}
with the natural inclusion $\tomega \to \tomega_{\log}$.

\begin{lemma}\label{lem:potential-restrict}
There is a commutative diagram
\[
\xymatrix{
\cP^{e,\circ}_{\UH^\cpt,\reg} \ar[rr]^{\tcW_{-}} \ar[d] && \tomega \ar[d] \\
\cP^{e,\circ}_{\UH^\cpt} \ar[rr]^{\tcW} && \tomega_{\log}
}
\]
where the two vertical arrows are the natural inclusions, $\tcW$ is the pull-back of $\tW$, and the two horizontal arrows are isomorphic away from the fibers over $\Sigma$.
\end{lemma}
\begin{proof}
It suffices to construct the following commutative diagram
\[
\xymatrix{
\cP^{e,\circ}_{\UH^\cpt,\reg} \ar[rr]^{\tcW'_{-}} \ar[d] && \tomega_{\fX_{\UH^\cpt}} \ar[d] \\
\cP^{e,\circ}_{\UH^\cpt} \ar[rr]^{\tcW'} && \tomega_{\log, \fX_{\UH^\cpt}}
}
\]
where the right vertical arrow is the pull-back of
$\tomega \to \tomega_{\log}$ along
$\fX_{\UH^\cpt} \to \cC_{\UH^\cpt}$.
By Proposition~\ref{prop:twisted-potential}, the composition
\[
  \cP^{e,\circ}_{\UH^\cpt,\reg} \to \cP^{e,\circ}_{\UH^\cpt} \to \tomega_{\log, \fX_{\UH^\cpt}}
\]
contracts the fiber of $\cP^{e,\circ}_{\UH^\cpt,\reg} \to \fX_{\UH^\cpt}$
over $\Sigma$ to the zero section of $\tomega_{\log, \fX_{\UH^\cpt}}$, hence 
factors through
$\tomega_{\fX_{\UH^\cpt}} \cong
\tomega_{\log,\fX_{\UH^\cpt}}(-\Sigma)$.
\end{proof}

\subsubsection{The relative cosection}
By \cite[Lemma 3.18]{CJRS18P}, \eqref{equ:hmap-modify-target} canonically lifts to a commutative triangle
\begin{equation}\label{equ:hmap-e-target}
\xymatrix{
\cC_{\UH^\cpt} \ar[rr]^{f_{\UH^\cpt}} \ar[rd]_{f_{\UH^\cpt,-}} && \cP^{e,\circ}_{\UH^\cpt} \\
&\cP^{e,\circ}_{\UH^\cpt,\reg} \ar[ru]&
}
\end{equation}
where the corresponding arrows are denoted again by
$f_{\UH^\cpt}$ and $f_{\UH^\cpt,-}$.
Now we have
\[
f^*_{\UH^\cpt,-}\diff\tcW_- \colon  f^*_{\UH^\cpt,-}T_{\cP^{e,\circ}_{\UH^\cpt,\reg}/\cC_{\UH^\cpt}} \longrightarrow (\tcW_-\circ f_{\UH^\cpt,-})^*T_{\tomega/\cC_{\UH^\cpt}} \cong \tomega.
\]

By \eqref{equ:modified-tangent}, we have a composition
\begin{equation}\label{eq:composing-cosection}
f^*_{\UH^\cpt}T_{\fP/\BC}(-\Sigma) \longrightarrow f^*_{\UH^\cpt,-}T_{\cP^{e,\circ}_{\UH^\cpt,\reg}/\cC_{\UH^\cpt}} \to \tomega,
\end{equation}
again denoted by $f^*_{\UH^\cpt,-}\diff\tcW_-$. Pushing forward along $\pi$ and using \eqref{equ:obs-compact-evaluation}, we have
\[
\sigma^{\bullet}_{\UH^\cpt/\fU^{\cpt}} := \pi_*\big(f^*_{\UH^\cpt,-}\diff\tcW_-\big) \colon \EE_{\UH^\cpt/\fU^{\cpt}} \longrightarrow \pi_*\tomega \cong \pi_*\omega_{\cC_{\UH^\cpt}/\UH^\cpt}\otimes\bL_{\max}^{-\ttwist}|_{\UH^\cpt}.
\]
where the isomorphism follows from the projection formula and \eqref{equ:twisted-omega}.

Finally, taking the first cohomology we obtain the \emph{canonical cosection}:
\begin{equation}\label{equ:canonical-cosection}
\sigma_{\UH^\cpt/\fU^{\cpt}} \colon \obs_{\UH^\cpt/\fU^{\cpt}} := H^1(\EE_{\UH/\fU^{\cpt}}) \longrightarrow \bL_{\max}^{-\ttwist}|_{\UH^\cpt}.
\end{equation}

\subsubsection{The degeneracy locus of $\sigma_{\UH^\cpt/\fU^{\cpt}}$}

Denote by $\IR_W$ the stack of $R$-maps in $\UH^\cpt$
which factor through $\crit(W)$.
Since $\crit(W)$ is a closed sub-stack of $\fP$, $\IR_W$ is a strict closed substack of $\IR^\cpt$. The stack $\UH^{\cpt}$ plays a key role in the following crucial result.

\begin{proposition}\label{prop:cosection-degeneracy-loci}
  Suppose $W$ has proper critical locus.
  Then the degeneracy locus of $\sigma_{\UH^\cpt/\fU^{\cpt}}$ is
  supported on $\IR_W \subset \UH^\cpt$.
\end{proposition}
\begin{proof}
  It suffices to check the statement at each geometric point.
  Let $f\colon \cC \to \fP$ be a stable log $R$-map given by a
  geometric point $S \to \UH^\cpt$.
  Following the same line of proof as in \cite{CJW19P}, consider the
  cosection:
\[
\sigma_S:= \sigma_{\UH^\cpt/\fU^{\cpt}}|_S \colon H^1(f^*T_{\fP/\BC}(-\Sigma)) \to H^1(\tomega|_{\cC})
\]
Applying Serre duality and taking dual, we have
\[
\sigma^{\vee}_S \colon H^0(\omega_{\cC/S}\otimes\tomega^{\vee}|_{\cC}) \to H^0(\omega_{\cC/S}\otimes f^*\Omega_{\fP/\BC}(\Sigma)).
\]
Note that $\omega_{\cC/S}\otimes\tomega^{\vee}|_{\cC} =\bL_{\max}^{\ttwist}|_{\cC} \cong \cO_{\cC}$. Thus $\sigma_S$ degenerates iff
\[
id\otimes\big(f^*_{-}\diff\tcW_-\big)^{\vee} \colon \omega_{\cC/S}\otimes\tomega^{\vee}|_{\cC} \to \omega_{\cC/S}\otimes f^*\Omega_{\fP/\BC}(\Sigma)
\]
degenerates which translates to the vanishing of
\begin{equation}\label{equ:cosection-degeneracy}
\big(f^*_{-}\diff\tcW_-\big) \colon   f^*T_{\fP/\BC}(-\Sigma) \to \cO_{\cC}
\end{equation}

Note that away from markings, $\tcW_-$ is the same as $\tcW$ which is
the pull-back of $\tW$.
If $S \not\in \Delta_{\UH^{\cpt}}$, then
\eqref{equ:cosection-degeneracy} degenerates iff $f$ factors through
$\crit(W)$.
Consider a geometric point $S \in \Delta_{\UH^{\cpt}}$.
By \cite[Lemma~3.18~(2)]{CJRS18P}, $\cC$ has at least one component
$\cZ$ whose image via $f_{-}$ is contained in the exceptional locus of
$\cP^{e,\circ}_{\UH^\cpt,\reg} \to  \cP_{\UH^\cpt,\reg}$.
Because $\tcW$ has proper critical locus,
\eqref{equ:cosection-degeneracy} is non-zero along $\cZ$.
This completes the proof.
\end{proof}

\subsection{The reduced theory}\label{ss:reduced}

Next we fix a $W$ hence $\tW$ with proper critical loci,
and apply the general machinery in Section~\ref{sec:POT-reduction} to construct
the reduced theory.

\subsubsection{The twisted Hodge bundle}

Consider
\[
\tomega_{\fU^{\cpt}} := \omega_{\cC_{\fU^{\cpt}}/\fU^{\cpt}} \otimes \pi^*_{\fU^{\cpt}}\bL^{-\ttwist}_{\max}
\]
and its direct image cone
$ \fH :=\mathbf{C}(\pi_{\fU,*}\tomega_{\fU^{\cpt}}) $
as in \cite[Definition 2.1]{ChLi12}.
It is an algebraic stack over $\fU^{\cpt}$ parameterizing sections of
$\tomega_{\fU^{\cpt}}$ \cite[Proposition 2.2]{ChLi12}.
Indeed, $\fH$ is the total space of the vector bundle
\[
R^0\pi_{\fU^{\cpt},*} \tomega_{\fU^{\cpt}} \cong R^0\pi_{\fU^{\cpt},
  *}\omega_{\fU^{\cpt}}\otimes \bL^{- \ttwist}_{\max}|_{\fU^{\cpt}}
 \]
 over $\fU^{\cpt}$ by \cite[Section 5.3.5]{CJRS18P}. We further equip $\fH$ with the log structure pulled back from $\fU^{\cpt}$.

By \cite[Proposition 2.5]{ChLi12}, $\fH \to \fU^{\cpt}$ has a perfect obstruction theory
\begin{equation}\label{equ:Hodge-perfect-obs}
\varphi_{\fH/\fU^{\cpt}} \colon \TT_{\fH/\fU^{\cpt}} \to \EE_{\fH/\fU^{\cpt}} := \pi_{\fH,*}\tomega_{\fH}.
\end{equation}
By projection formula, we have
\begin{equation}\label{equ:fake-obs}
H^1(\EE_{\fH/\fU^{\cpt}}) = R^1\pi_{\fH,*} \tomega_{\fH} = R^1\pi_{\fH, *}\omega_{\fH}\otimes \bL^{-\ttwist}_{\max}|_{\fH} \cong \bL^{-\ttwist}_{\max}|_{\fH}.
\end{equation}

Let $\bs_{\fH}\colon \cC_{\fH} \to \vb(\tomega_{\fH})$ be the universal section over $\fH$. The morphism $\UH^\cpt \to \fU^{\cpt}$ factors through the tautological morphism
\[
  \UH^\cpt \to \fH
\]
such that $\bs_{\fH}|_{\UH^\cpt} = \tcW_{-} \circ f_{\UH^\cpt,-}$.

\subsubsection{Verifying assuptions in Section \ref{ss:reduction-set-up}}\label{sss:verify-reduction-assumption}

First, the sequence \eqref{equ:stacks-reduction} in consideration is
\[
  \UH^\cpt \to \fH \to \fU^{\cpt}
\]
with the perfect obstruction theories $\varphi_{\UH^\cpt/\fU^{\cpt}}$
in \eqref{equ:obs-compact-evaluation} and $\varphi_{\fH/\fU^{\cpt}}$
in \eqref{equ:Hodge-perfect-obs}.
Choose the Cartier divisor $\Delta = \ttwist \Delta_{\fU^{\cpt}}$ with the pre-images
$\ttwist\Delta_{\UH^\cpt} \subset \UH^\cpt$ and $\ttwist\Delta_{\fH} \subset \fH$.
Thus we have the two term complex
$\FF = [\cO_{\fU^{\ev}_{0}} \stackrel{\epsilon}{\to}
\bL^{-\ttwist}_{\max}]$ in degrees $[0,1]$.
The commutativity of \eqref{diag:compatible-POT} is verified in Lemma
\ref{lem:obs-commute} below, and the sujectivity of
\eqref{equ:general-cosection} along $\Delta_{\fU_0}$ follows from
Proposition \ref{prop:cosection-degeneracy-loci}.

\begin{lemma}\label{lem:obs-commute}
There is a canonical commutative diagram
\begin{equation}\label{diag:rel-obs-commute}
\xymatrix{
\TT_{\UH^\cpt/\fU^{\cpt}} \ar[rr] \ar[d]_{\varphi_{\UH^\cpt/\fU^{\cpt}}} && \TT_{\fH/\fU^{\cpt}}|_{\UH^\cpt} \ar[d]^{\varphi_{\fH/\fU^{\cpt}}|_{\UH^\cpt}} \\
\EE_{\UH^\cpt/\fU^{\cpt}} \ar[rr]^{\sigma^{\bullet}_{\fU^{\cpt}}} && \EE_{\fH/\fU^{\cpt}}|_{\UH^\cpt}
}
\end{equation}
where the two vertical arrows are the perfect obstruction theories.
\end{lemma}
\begin{proof}
Similarly as in Section \ref{ss:modify-target}, we may construct the log weighted projective bundle
\[
\cP_{\fU^{\cpt}} \to \fX_{\fU^{\cpt}} := \cC_{\fU^{\cpt}}\times_{\BC} \fX
\]
and its modification $\cP^{e,\circ}_{\fU^{\cpt},\reg}$ with the pull-backs
\[
\cP_{\fU^{\cpt}}\times_{\fX_{\UH^\cpt}}\fX_{\fU^{\cpt}} \cong \cP_{\UH^\cpt} \ \ \ \mbox{and} \ \ \ \cP^{e,\circ}_{\fU^{\cpt},\reg}\times_{\fX_{\UH^\cpt}}\fX_{\fU^{\cpt}} \cong \cP^{e,\circ}_{\UH^\cpt,\reg}.
\]
We may also define the line bundle $\tilde{\omega}_{\fU^{\cpt}}$ over $\cC_{\fU^{\cpt}}$ similar to \eqref{equ:twisted-omega}.
The same proof as in Lemma \ref{lem:potential-restrict} yields a morphism
\[
\tcW_{\fU^{\cpt},-} \colon \cP^{e,\circ}_{\fU^{\cpt},\reg} \to \tilde{\omega}_{\fU^{\cpt}}
\]
which pulls back to $\tcW_{-}$ over $\UH^\cpt$. We obtain a commutative diagram
\[
\xymatrix{
 \cC_{\UH^\cpt} \ar[rr] \ar[d]_{f_{\UH^\cpt,-}} && \cC_\fH \ar[d]^{\bs_\fH} \\
 \cP^{e,\circ}_{\fU^{\cpt},\reg}   \ar[rr]^{\tcW_{\fU^{\cpt},-}} && \tilde{\omega}_{\fU^{\cpt}}
      }
\]
where by abuse of notations the two vertical arrows are labeled by the morphisms inducing them. This leads to a commutative diagram of log tangent complexes
\[
    \xymatrix{
      \pi^* \TT_{\UH^\cpt/\fU^{\cpt}} \cong \TT_{\cC_{\UH^\cpt}/\cC_{\fU^{\cpt}}} \ar[rr] \ar[d] &&       \pi^* \TT_{\cC_{\fH}/\cC_{\fU^{\cpt}}}|_{\cC_{\UH^\cpt}} \cong \TT_{\fH/{\fU^{\cpt}}}|_{\cC_{\UH^\cpt}} \ar[d]\\
      (f_{\UH^\cpt,-})^* \TT_{\cP^{e,\circ}_{\fU^{\cpt},\reg}/\cC_{\fU^{\cpt}}} \ar[rr]^{(\diff \tcW_{\fU^{\cpt},-})|_{\cC_{\UH^\cpt}}}  && (\bs_\fH)^* \TT_{\tilde{\omega}_{\fU^{\cpt}}/\cC_{\fU^{\cpt}}}|_{\cC_{\UH^\cpt}}
      }
\]
Diagram  \eqref{diag:rel-obs-commute} follows from first applying $\pi_*$ to the above diagram and then using adjunction.
\end{proof}

\subsubsection{The reduced perfect obstruction theory}\label{sss:reduced-theory}
Applying Theorem \ref{thm:reduction} to the situation above, we obtain the \emph{reduced perfect obstruction theory}
\begin{equation}\label{equ:red-OPT}
\varphi^{\red}_{\UH^\cpt/\fU^{\cpt}} \colon \TT_{\UH^\cpt/\fU^{\cpt}} \to \EE^{\red}_{\UH^\cpt/\fU^{\cpt}}
\end{equation}
and the \emph{reduced cosection}
\[
\sigma^{\red}_{\fU^{\cpt}} \colon H^1(\EE^{\red}_{\UH^\cpt/\fU^{\cpt}}) \to \cO_{\UH^\cpt}
\]
with the following properties
\begin{enumerate}
 \item The morphism $\varphi_{\UH^\cpt/\fU^{\cpt}}$ factors through $\varphi^{\red}_{\UH^\cpt/\fU^{\cpt}}$ such that
 \[
 \varphi_{\UH^\cpt/\fU^{\cpt}}|_{\UH^\cpt\setminus\Delta_{\UH^\cpt}} = \varphi^{\red}_{\UH^\cpt/\fU^{\cpt}}|_{\UH^\cpt\setminus\Delta_{\UH^\cpt}}
 \]

 \item $\sigma^{\red}_{\fU^{\cpt}}$ is surjective along $\Delta_{\UH^\cpt}$, and satisfies
\[
\sigma^{\red}_{\fU^{\cpt}}|_{\UH^\cpt\setminus\Delta_{\UH^\cpt}} = \sigma_{\fU^{\cpt}}|_{\UH^\cpt\setminus\Delta_{\UH^\cpt}}.
\]
\end{enumerate}

The virtual cycle $[\UH^\cpt]^{\red}$ associated to $\varphi^{\red}_{\UH^\cpt/\fU^{\cpt}}$ is called the {\em reduced virtual cycle} of $\UH^\cpt$. We emphasize that the reduced theory depends on the superpotential $W$.

\subsubsection{The cosection localized virtual cycle of $\IR^\cpt$}\label{sss:cosection-localized-class}
Recall from Proposition \ref{prop:cosection-degeneracy-loci} that the
degeneracy loci of $\sigma_{\UH^\cpt/\fU^{\cpt}}$ are supported along
the proper substack $\IR_W \subset \IR^\cpt$.
We have canonical embeddings
\[
\iiota\colon \IR_W \hookrightarrow \IR^\cpt \ \ \ \mbox{and} \ \ \ \iota\colon \IR_W \hookrightarrow \UH^\cpt.
\]

Since $\IU^{\cpt}$ is smooth, we are in the situation of Section \ref{ss:general-absolut-theory}. Applying Theorem \ref{thm:generali-localized-cycle}, we obtain the \emph{cosection localized virtual cycle}
\begin{equation}\label{equ:localized-cycle}
[\IR^\cpt]_{\sigma} \in A_*(\IR_W)
\end{equation}
with the property that $\iiota_*[\IR^\cpt]_{\sigma} = [\IR^\cpt]^{\vir}$. Since the canonical theory, the reduced theory, and their cosections all agree over $\IR^\cpt$, the existence of the cycle $[\IR^\cpt]_{\sigma}$  does \emph{not} require the compactification $\UH^\cpt$ of $\IR^\cpt$.

\subsection{The first comparison theorem}\label{ss:comparison-1}
We now show that the reduced virtual cycle and the cosection localized virtual cycle agree.
\begin{theorem}\label{thm:reduced=local}
$\iota_*[\IR^\cpt]_{\sigma}  = [\UH^\cpt]^{\red}$
\end{theorem}
\begin{proof}
Since $\UH^\cpt$ is of finite type, replacing $\fU$ by an open set containing the image of $\UH^\cpt$, we may assume that $\fU$ hence $\fU^{\cpt}$ is also of finite type. By \cite[Lemma 5.25]{CJRS18P}, there is a birational projective resolution $\fr\colon \tU \to \fU$ which restricts to the identity on $\IU = \fU\setminus(\Delta_{\max}|_{\fU})$. Let
\begin{equation}\label{equ:resolution}
\tU^{\cpt} =\fU^{\cpt}\times_{\fU}\tU \to \fU^{\cpt} \ \ \ \mbox{and} \ \ \ \tUH^\cpt = \UH\times_{\fU}\tU \to \UH.
\end{equation}
By abuse of notations, both morphisms are denoted by $\fr$ when there
is no danger of confusion.
Then the two morphisms restrict to the identity on $\IU^{\cpt}$ and
$\IR^{\cpt}$ respectively.
Furthermore, $\tU^{\cpt} \to \fU^{\cpt}$ is a birational projective
resolution by Proposition \ref{prop:compact-evaluation}.

Let $(\varphi^{\red}_{\tUH^\cpt/\tU^{\cpt}}, \sigma_{\tU^{\cpt}})$ be the pull-back of $(\varphi^{\red}_{\UH^\cpt/\fU^{\cpt}}, \sigma_{{\fU}^{\cpt}})$ along $\fr$. Then $\varphi^{\red}_{\tUH^\cpt/\tU^{\cpt}}$ defines a perfect obstruction theory of $\tUH^\cpt \to \tU^{\cpt}$ hence a virtual cycle $[\tUH^\cpt]^{\red}$. By the virtual push-forward of \cite{Co06, Ma12}, we have
\begin{equation}\label{equ:vc-push-forward-along-resolution}
\fr_*[\tUH^\cpt]^{\red} = [\UH^\cpt]^{\red}.
\end{equation}

On the other hand, since $(\varphi^{\red}_{\tUH^\cpt/\tU^{\cpt}}, \sigma_{\tU^{\cpt}})$ is the pull-back of $(\varphi^{\red}_{\UH^\cpt/\fU^{\cpt}}, \sigma_{{\fU}^{\cpt}})$, the same properties listed in Section \ref{sss:reduced-theory} also pull back to $(\varphi^{\red}_{\tUH^\cpt/\tU^{\cpt}}, \sigma_{\tU^{\cpt}})$. Since $\tU^{\cpt}$ is smooth, Theorem \ref{thm:generali-localized-cycle} implies
\begin{equation}\label{equ:reduced=local-resolution}
\iota_*[\tUH^\cpt]_{\sigma_{\tU^{\cpt}}} = [\UH^\cpt]^{\red}.
\end{equation}
Since $\fr$ does not modify the interior $\IR^\cpt$ and $\IU^{\cpt}$, we have
\begin{equation}\label{equ:local=local}
[\tUH^\cpt]_{\sigma_{\tU^{\cpt}}} = [\IR^\cpt]_{\sigma_{{\fU}^{\cpt}}}.
\end{equation}
Finally, \eqref{equ:vc-push-forward-along-resolution},
\eqref{equ:reduced=local-resolution}, and \eqref{equ:local=local}
together imply the statement.
\end{proof}

\subsection{The second comparison theorem}\label{ss:comparison-2}
By Section \ref{sss:verify-reduction-assumption} and Theorem \ref{thm:boundary-cycle} (1), we obtain a factorization of perfect obstruction theories of $\Delta_{\UH^\cpt} \to \Delta_{\fU^{\cpt}}$
\[
 \xymatrix{
 \TT_{\Delta_{\UH^\cpt}/\Delta_{\fU^{\cpt}}} \ar[rr]^{\varphi_{\Delta_{\UH^\cpt}/\Delta_{\fU^{\cpt}}}} \ar[rd]_{\varphi^{\red}_{\Delta_{\UH^\cpt}/\Delta_{\fU^{\cpt}}}} && \EE_{\Delta_{\UH^\cpt}/\Delta_{\fU^{\cpt}}} \\
 &\EE^{\red}_{\Delta_{\UH^\cpt}/\Delta_{\fU^{\cpt}}} \ar[ru]&
 }
\]
where the top is the pull-back of \eqref{equ:red-OPT}.
Let $[\Delta_{\UH^\cpt}]^{\red}$ be the \emph{reduced boundary virtual
  cycle} associated to
$\varphi^{\red}_{\Delta_{\UH^\cpt}/\Delta_{\fU^{\cpt}}}$. We then
have:

\begin{theorem}\label{thm:comparison-2}
$[\UH^\cpt]^{\vir} = [\UH^\cpt]^{\red} + \ttwist [\Delta_{\UH^\cpt}]^{\red}$.
\end{theorem}
\begin{proof}
The pull-back $\varphi_{\tUH^\cpt/\tU^{\cpt}} := \varphi_{\UH^\cpt/{\fU}^{\cpt}}|_{\tUH^\cpt}$ defines a perfect obstruction theory of $\tUH^\cpt \to \tU^{\cpt}$ with the corresponding virtual cycle $[\tUH^\cpt]^{\vir}$. Applying the virtual push-forward \cite{Co06, Ma12}, we have
\begin{equation}\label{equ:con-vc-push-forward}
\fr_*[\tUH^\cpt] = [\UH^\cpt].
\end{equation}

Consider the resolution \eqref{equ:resolution}, and write
\[
\Delta_{\tU^{\cpt}} = \Delta_{\fU^{\cpt}}\times_{\fU^{\cpt}}\tU^{\cpt} \ \ \ \mbox{and} \ \ \ \Delta_{\tUH^\cpt} = \Delta_{\tU^{\cpt}}\times_{\tU^{\cpt}}\tUH^\cpt.
\]
Applying Theorem  \ref{thm:boundary-cycle} to the data $(\tUH^\cpt, \ttwist\Delta_{\tUH^\cpt}, \varphi_{\tUH^\cpt/\tU^{\cpt}}, \sigma_{\tU^{\cpt}})$, we obtain the reduced boundary cycle $[\ttwist\Delta_{\tUH^\cpt}]^{\red} = \ttwist[\Delta_{\tUH^\cpt}]^{\red}$ and the following relation
\begin{equation}
[\tUH^\cpt] = [\tUH^\cpt]^{\red} + \ttwist \cdot [\Delta_{\tUH^\cpt}]^{\red}.
\end{equation}
Applying $\fr_*$ and using \eqref{equ:vc-push-forward-along-resolution} and \eqref{equ:con-vc-push-forward}, we have
\[
[\UH^\cpt] = [\UH^\cpt]^{\red} + \ttwist \cdot \fr_*[\Delta_{\tUH^\cpt}]^{\red}.
\]
It remains to verify that $[\Delta_{\UH^\cpt}]^{\red} = \fr_*[\Delta_{\tUH^\cpt}]^{\red}$.

Recall the degeneracy loci $\IR_{W} \subset \IR^\cpt$ of
$\sigma_{\tU^{\cpt}}$.
Write $V = \UH^\cpt \setminus \IR_{W}$ and
$\widetilde{V} = \tUH^\cpt\setminus \IR_{W}$.
In the same way as in \eqref{equ:t-red-POT} we construct the totally
reduced perfect obstruction theory $\EE^{\tred}_{V/\fU^{\cpt}}$ for
$V \to \fU^{\cpt}$ which pulls back to the totally reduced perfect
obstruction theory $\EE^{\tred}_{\widetilde{V}/\tU^{\cpt}}$ for
$\widetilde{V} \to \tU^{\cpt}$.
Let $[V]^{\tred}$ and $[\widetilde{V}]^{\tred}$ be the corresponding
virtual cycles.
Then the virtual push-forward implies
$\fr_*[\widetilde{V}]^{\tred} = [V]^{\tred}$.
We calculate
\[
\ttwist\cdot \fr_*[\Delta_{\tUH^\cpt}]^{\red} = \fr_* i^![\widetilde{V}]^{\tred} = i^![V]^{\tred} = \ttwist \cdot [\Delta_{\UH^\cpt}]^{\red}
\]
where the first and the last equalities follow from \eqref{equ:tred=bred}, and the middle one follows from the projection formula. This completes the proof.
\end{proof}

\subsection{Independence of twists II: the case of the reduced theory}\label{ss:red-theory-ind-twist}

In this section, we complete the proof of the change of twists theorem.

Consider the two targets $\fP_1$ and $\fP_2$ as in Section \ref{ss:change-twists}. Since $\fP_1 \to \fP_2$ is isomorphic along $\zero_{\fP_1} \cong \zero_{\fP_2}$ and $\ddata$ is a collection compact type sectors, the morphism in Corollary \ref{cor:changing-twists} restricts to
\[
\nu_{\etwist_1/\etwist_2} \colon \UH^{\cpt}_1 := \UH^{\cpt}_{g,\ddata}(\fP_1,\beta) \to \UH^{\cpt}_2 :=\UH^{\cpt}_{g,\ddata}(\fP_2,\beta).
\]
We compare the virtual cycles:

\begin{theorem}\label{thm:red-ind-twists}
\begin{enumerate}
\item
  $\nu_{{\etwist_1/\etwist_2},*}[\UH^{\cpt}_1]^{\red} =
  [\UH^{\cpt}_2]^{\red}$
\item
  $\nu_{{\etwist_1/\etwist_2},*}[\UH^{\cpt}_1]^\vir =
  [\UH^{\cpt}_2]^\vir$.
\item
$
\nu_{{\etwist_1/\etwist_2},*}[\Delta_{\UH^{\cpt}_1}]^{\red} = \frac{\etwist_2}{\etwist_1} \cdot [\Delta_{\UH^{\cpt}_2}]^{\red}.
$
\end{enumerate}
\end{theorem}
\begin{proof}
By Theorem \ref{thm:reduced=local}, both $[\UH^{\cpt}_1]^{\red}$ and $[\UH^{\cpt}_2]^{\red}$ are represented by the same cosection localized virtual cycle contained in the common open set $\IR^\cpt$ of both $\UH^{\cpt}_1$ and $\UH^{\cpt}_2$, hence are independent of the choices of $\etwist_i$. This proves the part of (1).

We can prove (2) similarly as in Proposition \ref{prop:can-ind-twists}. The only modification needed is to work over the log evaluation stack in Section \ref{sss:log-ev-stack}.

Finally, (3) follows from (1), (2) and Theorem \ref{thm:comparison-2}.
\end{proof}

\section{Examples}
\label{sec:examples}

\subsection{Gromov--Witten theory of complete intersections}\label{ss:examples-GW}

One of the most direct application of log GLSM is to study the
Gromov--Witten theory of complete intersections, and more generally,
zero loci of non-degenerate sections of vector bundles.
Here, the most prominent examples are quintic threefolds in $\PP^4$.

The input of this log GLSM is given by a proper smooth
Deligne--Mumford stack $\cX$ with a projective coarse moduli, a vector
bundle $\bE = \bE_1$ over $\cX$, a section $s \in H^0(\bE)$ whose zero
locus $\cZ$ is smooth of codimension $\rk \bE$. In this case we may choose
$\bL = \cO_{\cX}$, $r = 1$, and may choose $\etwist = 1$ for
simplicity.
Then the universal targets are
$\fP = \PP(\bE^\vee \otimes \LR \oplus \cO)$ and
$\fP^\circ = \vb(\bE^\vee \otimes \LR)$.
We may also view them as the quotients of $\PP(\bE^\vee \oplus \cO)$
and $\vb(\bE^\vee)$ under the $\CC^*_\omega = \CC^*$-scalar
multiplication on $\bE^\vee$.
By Proposition \ref{prop:map-field-equiv}, the data of a stable R-map
$f\colon \cC\ \to \fP^{\circ}$ with compact type evaluation over $S$
is equivalent to a stable map $g\colon \cC \to \cX$ over $S$ together
with a section $\rho \in H^0(\omega_\cC \otimes g^*(\bE^\vee))$.
Thus $\SR^{\cpt}_{g, \ddata}(\fP^{\circ},\beta)$ is the same as the
\emph{moduli space of stable maps to $\cX$ with $p$-fields} studied in
\cite{ChLi12, KiOh18P, ChLi18P, CJW19P}.

In this situation, the superpotential
\begin{equation*}
  W\colon \vb(\bE^\vee \boxtimes \LR) \to \vb(\LR)
\end{equation*}
is defined as the pairing with $s$.
It has proper critical locus whenever $\cZ$ is smooth of expected
dimension \cite[Lemma 2.2.2]{CJW19P}, and then the degeneracy locus $\IR_W$ is supported on 
$\scrM_{g, \ddata}(\cZ, \beta)$ embedded in the subset
$\scrM_{g, \ddata}(\cX, \beta) \subset \SR^{\cpt}_{g,
  \ddata}(\fP^{\circ},\beta)$, which is defined by log $R$-maps
mapping into $\zero_\fP$.
Recall that $\ddata$ is a collection of connected components of the
inertia stack of $\cX$.
The moduli space $\scrM_{g, \ddata}(\cZ, \beta)$ parameterizes stable
maps $\cC \to \cZ$ such that the composition $\cC \to \cZ \to \cX$ has
curve class $\beta$, and sectors $\ddata$.
In particular, $\scrM_{g, \ddata}(\cZ, \beta)$ is a disjoint union
parameterized by curve classes $\beta'$ on $\cZ$ such that
$\iota_* \beta' = \beta$ under the inclusion
$\iota\colon \cZ \to \cX$.
Combining Theorem \ref{thm:reduced=local} with the results in \cite{ChLi12, KiOh18P, ChLi18P}, and more
generally in \cite{CJW19P, Pi20P}, we obtain:

\begin{proposition}
  \label{prop:glsm-gw}
  In the above setting, we have
  \begin{equation*}
    [ \UH_{g, \ddata}(\fP, \beta)]^{\red}
    = (-1)^{\rk(\bE)(1 - g) + \int_\beta c_1(\bE) - \sum_{j = 1}^n \age_j(\bE)} [\scrM_{g, \ddata}(\cZ, \beta)]^\vir,
  \end{equation*}
  where $\age_j(\bE)$ is the age of $\bE|_\cC$ at the $j$th
  marking (see \cite[Section~7]{AGV08}).
\end{proposition}

Therefore Gromov--Witten invariants of $\cZ$ (involving only cohomology classes
from $\cX$) can be computed in terms of log GLSM
invariants.

\begin{proof}
  We will show that the perfect obstruction theory and cosection used
  in this paper are compatible with those in \cite{CJW19P}.
  Recall the notations
  $\IR^\cpt = \SR^\cpt_{g, \ddata}(\fP^\circ, \beta)$ and
  $\IU^{\cpt} = \fU^{\cpt}\setminus(\Delta_{\fU^{\cpt}})$ from
  \eqref{equ:interior-stack}.
  Note that $\IU^{\cpt} = \IU \times (\ocI\cX)^n$ where $ \ocI\cX$ is
  the rigidified cyclotomic inertia stack of $\cX$ as in
  \cite[3.4]{AGV08}, and $\IU$ is simply the moduli of twisted curves.
  Note that we have a morphism of distinguished triangles over
  $\IR^\cpt$
\[
\xymatrix{
 \TT_{\IR^\cpt/\IU^{cpt}} \ar[r] \ar[d] & \TT_{\IR^\cpt/\IU} \ar[r] \ar[d] & T_{(\ocI\cX)^n}|_{\IR^\cpt} \ar[d]^{\cong} \ar[r] &\\
 \pi_{\IR^\cpt,*}f^*_{\IR^\cpt}T_{\fP/\BC}(-\Sigma) \ar[r]  &  \pi_{\IR^\cpt,*}f^*_{\IR^\cpt,-}T_{\cP_{\IR^\cpt,\reg}/\cC_{\IR^\cpt}} \ar[r] &   \pi_{\IR^\cpt,*}T_{\fX/\BC}|_{\Sigma} \ar[r] & \\
 }
\]
where the left vertical arrow is the restriction of the perfect obstruction theory \eqref{equ:obs-compact-evaluation} to $\IR^\cpt$, the middle vertical arrow is precisely the perfect obstruction theory \cite[(18)]{CJW19P}, the vertical arrow on the right follows from \cite[Lemma 3.6.1]{AGV08}, and the bottom is obtained by applying the derived pushforward $\pi_{\IR^\cpt,*}$ to \eqref{equ:modified-tangent}. Thus, the perfect obstruction theory defined in this paper is compatible  with that of \cite{CJW19P}, hence they define the same absolute perfect obstruction theory of $\IR^\cpt$.

Now applying $R^1\pi_{\IR^\cpt,*}$ to the composition \eqref{eq:composing-cosection}, we have 
\[
\xymatrix{
R^1\pi_{\IR^\cpt,*}f^*_{\IR^\cpt}T_{\fP/\BC}(-\Sigma) \ar[r]^{\cong \ \ } \ar[rd] & R^1\pi_{\IR^\cpt,*}f^*_{\IR^\cpt,-}T_{\cP_{\IR^\cpt,\reg}/\cC_{\IR^\cpt}} \ar[d]\\
& \cO_{\IR^\cpt}
}
\]
where the horizontal isomorphism follows from the compatibility of perfect obstruction theories above, the vertical arrow on the right is the relative cosection \cite[(25)]{CJW19P}, and the skew arrow is the relative cosection \eqref{equ:canonical-cosection} restricted to the open substack $\IR^\cpt$. This means that the cosections in this paper over  $\IR^\cpt$ is identical to the cosections in \cite{CJW19P}. Therefore, the statement follows from \cite[Theorem 1.1.1]{CJW19P}.
\end{proof}

\subsection{FJRW theory}

We discuss in this section how our set-up includes all of FJRW theory,
which is traditionally \cite{FJR13} stated in terms of a
quasi-homogeneous polynomial $W$ defining an isolated singularity at
the origin, and a diagonal symmetry group $G$ of $W$.

We first recall a more modern perspective on the input data for the
FJRW moduli space following \cite[Section~2.2]{FJR18} and
\cite[Section~3]{PoVa16}.
Fix an integer $N$, a finite subgroup $G \subset \Gm^N$, and positive
integers $c_1, \dotsc, c_N$ such that $\gcd(c_1, \dotsc, c_N) = 1$. 
Let $\CC^*_R$ be the one dimensional sub-torus
$\{(\lambda^{c_1}, \dotsc, \lambda^{c_N})\} \subset \Gm^N$, and assume
that $G \cap \CC^*_R$ is a cyclic group of order $r$, which is usually
denoted by $\langle J\rangle$.
Consider the subgroup $\Gamma = G \cdot \CC^*_R \subset \Gm^N$.
There is a homomorphism
$\zeta\colon \Gamma \to \CC^*_\omega \cong \Gm$ defined by
$G \mapsto 1$ and
$(\lambda^{c_1}, \dotsc, \lambda^{c_N}) \mapsto \lambda^r$.
\begin{definition}
  A {\em $\Gamma$-structure} on a twisted stable curve $\cC$ is a commutative
  diagram
  \begin{equation*}
    \xymatrix{
      & \mathbf{B\Gamma} \ar[d] \\
      \cC \ar[r] \ar[ur] & \BC.
    }
  \end{equation*}
  A {\em $\Gamma$-structure with fields} \cite{CLL15} is a commutative diagram
  \begin{equation*}
    \xymatrix{
      & [\CC^N / \Gamma] \ar[d] \\
      \cC \ar[r] \ar[ur] & \BC.
    }
  \end{equation*}
\end{definition}
\begin{remark}
  A special case of FJRW theory is the $r$-spin theory, whose logarithmic
  GLSM was discussed in \cite{CJRS18P}.
  In this case, $N = 1$, $\CC^*_R = \Gamma$, and
  $G = \mu_r \subset \CC^*_R$ is the subgroup of $r$th roots of unity.
  \end{remark}

\begin{lemma}
  \label{lem:hybrid-vs-CLL}
  There is hybrid target data (as in Section~\ref{ss:target-data})
  such that there is a commutative diagram
  \begin{equation*}
    \xymatrix{
      [\CC^N/\Gamma] \ar[r]^{\sim} \ar[dr] & \fP^\circ \ar[d] \\
      & \BC.
      }
  \end{equation*}
\end{lemma}
\begin{proof}
  This is a special case of the following
  Lemma~\ref{lem:hybrid-match}.
\end{proof}

There are several constructions of the FJRW virtual cycle in full
generality \cite{CKL18, FJR08, KL18, PoVa16}.
The construction closest to ours, and which we will follow here, is
the approach \cite{CLL15} using cosection localized virtual classes
for the special case of narrow insertions at all markings.

In the FJRW situation, by Lemma~\ref{lem:hybrid-vs-CLL}, the moduli space $\SR^{\cpt}_{g, \ddata}(\fP^\circ, \beta)$ of stable $R$-maps is
the same as the moduli of $G$-spin curves with fields in \cite{CLL15}.
Indeed, $\cX$ is a point, and all compact-type
sectors are narrow.
In this case, Proposition~\ref{prop:compact-evaluation} (1) implies that 
$\SR^{\cpt}_{g, \ddata}(\fP^\circ, \beta) = \SR_{g, \ddata}(\fP^\circ, \beta)$.

The perfect obstruction theories in this paper are constructed in a
slightly different way to the ones in \cite{CLL15} or \cite{CJRS18P} in
that construct it relative to a moduli space of twisted curves instead
of a moduli space of $G$-spin curves.
These constructions are related via a base-change of obstruction
theories as in \cite[Lemma~A.2.2]{CJW19P}, and in particular give rise
to the same virtual class.
Given a superpotential $W$ with proper critical locus, the cosection
constructed in Section~\ref{sss:cosection-localized-class} is easily
seen to agree with the one in \cite{CLL15}.
Therefore, $[\IR^\cpt]^\vir$ is the FJRW virtual class, and log GLSM
recovers FJRW theory in the narrow case.

\subsection{Hybrid models}\label{ss:ex-hyb-model}

The hybrid GLSM considered in the literature \cite{CFGKS18P, Cl17, FJR18} 
fit neatly into our set-up, and they form a generalization
of the examples of the previous sections.
In this paper though, we restrict ourselves to the case of compact
type insertions, and to the $\infty$-stability in order to include non-GIT quotients.

The input data of a hybrid GLSM is the following: Let
$G \subset \Gm^{K + N}$ be a sub-torus, and $\theta\colon G \to \CC^*$
be a character such that the stable locus $\CC^{K, s}$ and the
semi-stable locus $\CC^{K, ss}$ for the $G$-action on
$\CC^K = \CC^K \times \{0\}$ agree, and that
$\CC^{K+N, ss} = \CC^{K,ss}\times \CC^{N}$.
Then $[\CC^{K, ss} \times \CC^N / G]$ is the total space of a vector
bundle $\bE^\vee$ on a Deligne--Mumford stack
$\cX = [\CC^{K, ss} / G]$.
Furthermore, assume that there is a one-dimensional subtorus
$\CC^*_R = \{(1, \dotsc, 1, \lambda^{c_1}, \dotsc, \lambda^{c_N})\}
\subset \Gm^{K + N}$ such that $c_i > 0$ for all $i$, and
$G \cap \CC^*_R \cong \ZZ/r\ZZ$.
Let $\Gamma = G\cdot \CC^*_R$, and define
$\zeta\colon \Gamma \to \CC^*_\omega$ via $G \mapsto 1$ and
$(\lambda^{c_1}, \dotsc, \lambda^{c_N}) \mapsto \lambda^r$.

Given this set-up, the moduli space of $\infty$-stable LG quasi-maps
\cite[Definition~1.3.1]{CFGKS18P} is the same as the moduli space of
$R$-maps to the target $[\CC^{K, ss} \times \CC^N/\Gamma] \to \BC$.
Analogously to the previous section, we have the following:
\begin{lemma}
  \label{lem:hybrid-match}
  There is hybrid target data (as in Section~\ref{ss:target-data})
  such that there is a commutative diagram
  \begin{equation*}
    \xymatrix{
      [\CC^{K, ss} \times \CC^N/\Gamma] \ar[r]^-{\sim} \ar[dr] & \fP^\circ \ar[d] \\
      & \BC.
      }
  \end{equation*}
\end{lemma}
\begin{proof}
  Choose a splitting
  $\Gm^{K + N} \cong T \times \CC^*_R$ into
  tori.
  Let $H$ be the projection of $G$ to $T$.
  Then there is an isomorphism $\Gamma \cong H \times \CC^*_R$
  defined by the projections, and the homomorphism
  $\zeta\colon \Gamma \to \CC^*_\omega$ becomes of the form
  $(\lambda, h) \mapsto \lambda^r \chi(h)$ for the character
  $\chi := \zeta|_{H} \colon H \to \CC^*_\omega$.

  Set $\cX = [\CC^{K, ss}/H]$ and let $\bL$ be the line bundle induced
  by $\chi$.
  Then $[\CC^{K, ss} \times \CC^N / H]$ is a rank $N$ vector bundle
  over $\cX$ with the splitting $\bE = \oplus_j \bE^{\vee}_{c_j}$
  according to the weights $c_j$ of the $\CC^*_R$-action.

  Consider $\fX := [\CC^{K, ss}/ \Gamma] \cong \mathbf{BC}^*_R \times \cX \to \BC \times \cX$ induced by the line bundle
  $\cL_R^{\otimes r} \boxtimes \bL$ and the identity on the second
  factor.
  Here, $\cL_R$ is the universal line bundle on $\mathbf{BC}^*_R$.
  The universal spin structure $\Sp$ is the pull-back of $\cL_R$. We then have $\fP^{\circ} \cong [\vb(\oplus_i \bE^{\vee}_i)/\CC^*_R] \to \BC$ which is the same as $[\CC^{K, ss} \times \CC^N / \Gamma] \to \BC$.
\end{proof}

It is a straightforward verification that the hybrid GLSM virtual
cycles constructed in our paper agree with those constructed in
\cite{Cl17, FJR18}.
Indeed, the absolute perfect obstruction theory and cosection for
$\IR^{\cpt}$ constructed in this paper agree with the ones in the
literature (to see this, we again need the base-change lemma
\cite[Lemma~A.2.2]{CJW19P}).
We leave the comparison to \cite{CFGKS18P} for a future work.

\section{Properties of the stack of stable logarithmic $R$-maps}
\label{sec:properties}
In this section, we establish Theorem \ref{thm:representability}.

\subsection{The representability}

For convenience, we prove the algebraicity of the stack $\SH(\fP)$ of
all log R-maps with all possible discrete data since the discrete data
specifies open and closed components, and the stability is an open
condition. 

Consider the stack of underlying R-maps $\fS(\ul{\fP}/\BC)$ which associates to any scheme $\ul{S}$ the category of commutative diagrams
\[
\xymatrix{
\uC\ar[rd]_{\omega^{\log}_{\ucC/\ul{S}}} \ar[r]^{\ul{f}} & \ul{\fP} \ar[d] \\
& \BC
}
\]
where $\uC \to \ul{S}$ is a family of twisted curves. 
As proved in \cite{AbCh14, Ch14,GrSi13}, the tautological morphism
$\SH(\fP, \beta) \to \fS(\ul{\fP}/\BC)$ is represented by log
algebraic spaces, see also \cite[Theorem~2.11]{CJRS18P}.
To show that $\SH(\fP, \beta)$ is algebraic, it remains to prove the
algebraicity of $\fS(\ul{\fP}/\BC)$. 
Now consider the tautological morphism
\[
\fS(\ul{\fP}/\BC) \to \fMtw
\]
where $\fMtw$ is the stack of twisted pre-stable curves.
For any morphism $\ul{S} \to \fMtw$, the corresponding pre-stable
curve $\ul{\cC} \to \ul{S}$ defines a fiber product
$\ul{\fP}\times_{\BC}\ul{\cC}$.
For any $\ul{T} \to \ul{S}$, the fiber product
$$\fS_{\ul{S}}(\ul{T}) := \ul{T}\times_{\fM(\cX)}\fS(\ul{\fP}/\BC)(\ul{S})$$
parameterizes sections of the projection $\ul{\fP}\times_{\BC}\ul{\cC}_{\ul{T}} \to \ul{\cC}_{\ul{T}} := \ul{\cC}\times_{\ul{S}}\ul{T}$.
Note that the composition $\ul{\fP}\times_{\BC}\ul{\cC} \to \ul{\cC} \to \ul{S}$ is proper and of Deligne--Mumford type. 
Since being a section is an open condition, the stack $\fS_{\ul{S}}$ is an open substack of the stack parameterizing pre-stable maps to the family of targets $\ul{\fP}\times_{\BC}\ul{\cC} \to \ul{S}$, which is algebraic by the algebraicity of Hom-stacks in \cite[Theorem 1.2]{HaRy19}.
Hence, $\fS_{\ul{S}}$ is algebraic over $\ul{S}$.
This proves the algebraicity of $\fS(\ul{\fP}/\BC)$.  

\subsection{Finiteness of automorphisms}

We now verify that $\SH_{g, \ddata}(\fP, \beta)$ is of
Deligne--Mumford type.
Let $f\colon \cC \to \fP$ be a pre-stable $R$-map.
An automorphism of $f$ over $\ul{S}$ is an automorphism of the log
curve $\cC \to S$ over $\ul{S}$ which fixes $f$.
Denote by $\Aut(f/S)$ the sheaf of automorphism groups of $f$ over
$\ul{S}$.
Since the underlying stack $\ul{\SH_{g, \ddata}(\fP, \beta)}$
parameterizes minimal objects in Definition \ref{def:minimal}, it
suffices to consider the following:

\begin{proposition}\label{prop:finite-auto}
  Assume $f$ as above is minimal and stable, and that $\ul{S}$ is a
  geometric point.
  Then $\Aut(f/S)$ is a finite group.
\end{proposition}
\begin{proof}
  By \cite{Ch14} and \cite{GrSi13}, it suffices to show that the
  automorphism group of the underlying objects are finite, see also
  \cite[Proposition~2.10 and 3.13]{CJRS18P}.
  By abuse of notation, we leave out the underlines, and assume all
  stacks and morphisms are equipped with the trivial logarithmic
  structures.

  Since the dual graph of $\cC$ has finitely many automorphisms, it
  suffices to consider the case that $\cC$ is irreducible.
  After possibly taking normalization, and marking the preimages of
  nodes, we may further assume that $\cC$ is smooth.
  Suppose $f$ has infinite automorphisms.
  Then we have either $\cC$ is smooth and rational, and the total
  number of markings is less than 3, or $\cC$ is an unmarked genus one
  curve.
  In both cases, the morphism $g := \ft \circ f\colon \cC \to \cX$
  contracts the curve to a point $x \in \cX$.

  We first consider the cases that $\cC$ is rational with two
  markings, or it is of genus one without any markings.
  In both cases, we have $\omega^{\log}_{\cC/S} \cong \cO_{\cC/S}$.
  Thus the morphism $\cC \to \BC$ induced by $\omega^{\log}_{\cC/S}$
  factors through the universal quotient $\spec \bk \to \BC$.
  We obtain a commutative diagram
  \begin{equation}\label{diag:finite-auto-trivial-omega}
    \xymatrix{
      \cC \ar@/_3ex/[rd] \ar[r]_{f_\bk} \ar@/^3ex/[rr]^{f} & \fP_\bk \ar[r] \ar[d] & \fP \ar[d] \\
      & \spec \bk \ar[r] & \BC
    }
  \end{equation}
  where the square is cartesian. Since the automorphism group of $f$
  is infinite, the automorphism group of $f_\bk$ is infinite as well.
  Thus $f_\bk$ contracts $\cC$ to a point of the Deligne--Mumford stack
  $\fP_\bk$.
  Then we have
  \[
    \deg\big(f_\bk^*\cO(\infty_{\fP_\bk})\big) = \deg\big( f^*\cO(\infty_{\fP})\big) = 0
  \]
  which contradicts the stability of $f$ as in \eqref{equ:hyb-stability}.

  Now assume that $\cC$ is rational with at most one marking.
  Suppose there is no point $q \in \cC$ such that $f(q) \in \zero_\fP$.
  Let $f_{\cX}\colon \cC \to \infty_{\cX}$ be the composition
  $\cC \to \fP \setminus \zero_\fP \to \infty_{\fP} \to \infty_{\cX}$
  where $\fP \setminus \zero_\fP \to \infty_{\fP}$ is the projection from
  $\zero_\fP$ to $\infty_\fP$, see
  Proposition~\ref{prop:curve-in-infinity}.
  Since automorphisms of $f$ fix $f_{\cX}$, the map $f_{\cX}$
  contracts $\cC$ to a point of $\infty_{\cX}$, hence
  $\deg \big(f_{\cX}^*\cO_{\infty_{\fP'}}(\frac{r}{d})\big) = 0$.
  Proposition \ref{prop:curve-in-infinity} immediately leads to a
  contradiction to the stability condition \eqref{equ:hyb-stability}.
  Thus there must be a point $q \in \cC$ such that $f(q) \in \zero_\fP$.

  On the other hand, since $\omega^{\log}_{\cC/S} < 0$ and
  $\deg g^*\cH = 0$, by the stability condition
  \eqref{equ:hyb-stability} we must have
  $\deg\big( f^* \cO(\infty_{\fP})\big) > 0$. Thus $\cC$
  intersects $\infty_{\fP}$ properly at its unique marking, denoted by
  $\sigma$, as the morphism $f$ comes from a log map.
  Clearly, $q \neq \sigma$.

  Consider the $\GG_m$-invariant open subset
  $U = \cC \setminus \{q\}$.
  Note that $\omega_U^{\log}$ is $\GG_m$-equivariantly trivial.
  We thus arrive at the same diagram
  \eqref{diag:finite-auto-trivial-omega} with $\cC$ replaced by $U$.
  The infinite automorphism group implies that $f_\bk|_{U}$ is constant.
  On the other hand, the image of $U$ must intersect $\infty_{\fP_\bk}$
  properly.
  This is not possible!
\end{proof}

\subsection{Boundedness}
We next show that the stack $\SH_{g, \ddata}(\fP, \beta)$ is of finite type. Consider the following composition
\begin{equation}\label{equ:take-curve}
\SH_{g, \ddata}(\fP, \beta) \to \fS(\ul{\fP}/\BC) \to \fM_{g,n}
\end{equation}
where $\fM_{g,n}$ is the stack of genus $g$, $n$-marked pre-stable curves, the first arrow is obtained by removing log structures, and the second arrow is obtained by taking coarse source curves. We divide the proof into two steps.

\subsubsection{The composition \eqref{equ:take-curve} is of finite type.}
Let $T \to \fM_{g,n}$ be a morphism from a finite type scheme $T$, and $C \to T$ be the universal curve. Since the question is local on $\fM_{g,n}$, it suffices to prove that
\[
\SH_T := \SH_{g, \ddata}(\fP, \beta)\times_{\fM_{g,n}} T \to \fS_{T} := \fS(\ul{\fP}/\BC)\times_{\fM_{g,n}} T \to T
\]
is of finite type.

For any object $(f\colon \cC_S \to \ul{\fP}) \in \fS_T(S)$, let $C_T$ be
the pull-back of $C \to T$ via $S \to T$.
Then $\cC_S \to C_T$ is the coarse morphism.
Note that $\omega^{\log}_{C_T/T}$ pulls back to
$\omega^{\log}_{\cC_S/S}$.
We thus obtain a commutative diagram of solid arrows with the unique
square cartesian:
\begin{equation}\label{diag:factor-through-coarse}
\xymatrix{
&& \ul{\fP}_{T} \ar[rr] \ar[d] && \ul{\fP} \ar[d] \\
\cC_S \ar[rr] \ar@{-->}[rru]^{\tilde{f}} \ar@/_1pc/[rrrr]_{\omega^{\log}_{\cC_S/S}} && C \ar[rr]^{\omega^{\log}_{C/T}} && \BC
}
\end{equation}
Then it follows that $f$ factors through a unique dashed arrow
$\tilde{f}$ making the above diagram commutative.

Note that $\ul{\fP}_T \to T$ is a family of proper Deligne--Mumford
stacks with projective coarse moduli spaces over $T$.
Let $\tilde{\beta}$ be the curve class of the fiber of
$\ul{\fP}_T \to T$ corresponding to objects in $\SH_T$.
Note that $\tilde{\beta}$ is uniquely determined by the curve class
$\beta$ in $\cX$ and the contact orders.
Thus, the morphism $\SH_T \to \fS_T$ factors through the open substack
$\fS_T(\tilde{\beta}) \subset \fS_T$ with the induced maps with curve
class $\tilde{\beta}$.

First, note that the morphism $\SH_T \to \fS_T(\tilde{\beta})$ is of
finite type. Indeed, using the same proof as in
\cite[Lemma~4.15]{CJRS18P}, one shows that the morphism
$\SH_T \to \fS_T(\tilde{\beta})$ is combinatorially finite
(\cite[Definition 2.14]{CJRS18P}), hence is of finite type by
\cite[Proposition 2.16]{CJRS18P}.

Now let $\scrM_{g,n}(\ul{\fP}_T/T,\tilde{\beta})$ be the stack of
genus $g$, $n$-marked stable maps to the family of targets
$\ul{\fP}_T/T$ with curve class $\tilde{\beta}$.
Then $\fS_T(\tilde{\beta})$ is identified with the locally closed
sub-stack of $\scrM_{g,n}(\ul{\fP}_T/T,\tilde{\beta})$ which for any
$T$-scheme $S$ associates the category of stable maps
$\tilde{f}\colon \cC_S \to \ul{\fP}_T$ over $S$ such that the induced
map $C_S \to C$ from the coarse curve $C_S$ of $\cC_S$ to $C \to T$ is
stable, and is compatible with the marked points.
Since $\scrM_{g,n}(\ul{\fP}_T/T,\tilde{\beta})$ is of finite type over
$T$ by \cite[Theorem 1.4.1]{AbVi02}, $\fS_T(\tilde{\beta})$ is of
finite type.

\subsubsection{The image of \eqref{equ:take-curve} is of finite type}
It remains to show that the image of \eqref{equ:take-curve} is contained in a finite type sub-stack of $\fM_{g,n}$. For this purpose, it suffices to bound the number of unstable components of the source curves in $\SH_{g, \ddata}(\fP, \beta)$.

Let $f\colon \cC \to \fP$ be an $R$-map corresponding to a geometric
log point $S \to \SH_{g, \ddata}(\fP, \beta)$.
Observe that the number
\begin{equation}\label{equ:boundedness-degree}
  d_{\beta}:= \deg \omega^{\log}_{\cC/S}\otimes {f}^*\cO(\ttwist \infty_{\fP})
\end{equation}
is a constant only depending on the choice of genus $g$, the orbifold structure at markings, and the contact orders. Let $\cZ \subset \cC$ be an irreducible component. Denote by
\[
d_{\beta,\cZ} := \deg \big( \omega^{\log}_{\cC/S}\otimes {f}^*\cO(\ttwist \infty_{\fP})\big)|_{\cZ}.
\]
Let $g := \ft \circ f$ be the pre-stable map underlying $f$.
An irreducible component $\cZ\subset \cC$ is called \emph{$g$-stable}
if $(\deg g^*\cH)|_{\cZ} > 0$ or $\cZ$ is a stable component of the
curve $\cC$.
Otherwise, $\cZ$ is called \emph{$g$-unstable}.

Suppose $\cZ$ is $g$-unstable.
Then by the stability condition \eqref{equ:hyb-stability} and
$(\deg g^*\cH)|_{\cZ} = 0$, we have
\[
  d_{\beta,\cZ} \ge \deg \big( (\omega^{\log}_{\cC/S})^{1 + \delta}\otimes {f}^*\cO(\ttwist \infty_{\fP})\big)|_{\cZ} > 0.
\]
Note that $\fP_\bk$ is a proper Deligne--Mumford stack, and the stack of
cyclotomic gerbes in $\fP_\bk$ is of finite type, see
\cite[Definition~3.3.6, Corollary~3.4.2]{AGV08}.
Thus there exists a positive integer $\lambda$ such that if $\fP_\bk$
has a cyclotomic gerbes banded by $\mu_k$, then $k | \lambda$.
Since $\ul{f}$ factors through a representable morphism $\tilde{f}$ in
\eqref{diag:factor-through-coarse}, we have
$d_{\beta,\cZ} \ge \frac{1}{\lambda}$.

Now we turn to considering $g$-stable components.
Since the genus is fixed, and the orbifold structure of $\cZ$ is
bounded, the number of $g$-stable components is bounded.
Let $\cZ$ be an $g$-stable component.
We have the following two possibilities.

Suppose $f(\cZ) \not\subset \infty_{\fP}$, hence
$\deg {f}^*\cO(\ttwist \infty_{\fP})\big)|_{\cZ} \geq 0$.
Then we have $d_{\beta,\cZ} \geq -1$ where the equality holds only if
$\cZ$ is a rational tail.

Now assume that $f(\cZ) \subset \infty_{\fP}$.
By Proposition~\ref{prop:curve-in-infinity}, we have
\[
  d_{\beta,\cZ} = \deg g^*\bL\otimes (\ul{f}_{\cX})^*\cO_{\infty_{\cX}}(\frac{r}{d}).
\]
Since $\deg (\ul{f}_{\cX})^*\cO_{\infty_{\cX}}(\frac{r}{d}) \geq 0$
and $\deg g^*\bL$ is bounded below by some number only depending on
$\bL$ and the curve class $\beta$, we conclude that $d_{\beta,\cZ}$ is
also bounded below by some rational number independent the choices of
$\cZ$.

Finally, note that
\[
  d_{\beta} = \sum_{\cZ\colon\text{ $g$-stable}} d_{\beta,\cZ} + \sum_{\cZ\colon\text{ $g$-unstable}} d_{\beta,\cZ}.
\]
The above discussion implies that the first summation can be bounded below by a number only depending on the discrete data $\beta$, and each term in the second summation is a positive number larger than $\frac{1}{\lambda}$. We thus conclude that the number of irreducible components of the source curve $\cC$ is bounded.

This finishes the proof of the boundedness.


\subsection{The set-up of the weak valuative criterion}\label{ss:valuative-set-up}
Let $R$ be a discrete valuation ring (DVR), $K$ be its quotient field, $\fm \subset R$ be the maximal ideal, and $k = R/\fm$ the residue field. Denote by $\ul{S} = \spec R$, $\ul{\eta} = \spec K$ and $\ul{s} = \spec k$.
Our next goal is to prove the weak valuative criterion of stable log R-maps.

\begin{theorem}\label{thm:weak-valuative}
  Let $f_{\eta}\colon \cC_{\eta} \to \fP$ be a minimal log $R$-map
  over a log $K$-point $\eta$.
  Possibly replacing $R$ by a finite extension of DVRs, there is a
  minimal log $R$-map $f_S \colon \cC \to \fP$ over $S$ extending
  $f_{\eta}$ over $\ul{\eta}$.
  Furthermore, the extension $f_S$ is unique up to a unique
  isomorphism.
\end{theorem}

We will break the proof into several steps.
Since the stability condition in Section \ref{sss:stability} 
are constraints only on the level of underlying structures, by the relative properness of log maps over underlying stable maps \cite{Ch14, AbCh14, GrSi13}, see also \cite[Proposition 2.17]{CJRS18P}, it suffices to prove the existence
and uniqueness of an underlying stable $R$-map
$\ul{f}\colon \ul{\cC} \to \fP$ extending $f_{\eta}$ over $S$,
possibly after a finite extension of the base.

Since the focus is now the underlying structure, we will leave out the
underlines to simplify the notations, and assume all stacks are
equipped with the trivial logarithmic structure for the rest of this
section.

Normalizing along nodes of $\cC_{\eta}$, possibly taking further base
change, and marking the points from the nodes, we obtain a possibly
disjoint union of smooth curves $\cC_{\eta}^{n}$.
Observe that $\cC_{\eta}^n \to \BC$ induced by
$\omega^{\log}_{\cC_{\eta}^n}$ factors through the corresponding
$\cC_{\eta} \to \BC$ before taking normalization.
Thus, we may assume that $\cC_{\eta}$ is smooth and irreducible.
It is important to notice that every isolated intersection of
$\cC_{\eta}$ with $\infty_{\fP}$ via $f$ is marked.
This will be crucial in the proof below.

\subsection{The separatedness}
We first verify the uniqueness in Theorem \ref{thm:weak-valuative}.
The strategy is similar to \cite[Section~4.4.5]{CJRS18P} but in a more
complicated situation.

\subsubsection{Reduction to the comparison of coarse curves}

Let $f_i\colon \cC_i \to \fP$ be a stable underlying R-map over $S$
extending $f_{\eta}$ for $i=1,2$.
Let $\cC_i \to C_i$ and $\cC_{\eta} \to C_{\eta}$ be the corresponding
coarse moduli.
By \eqref{diag:factor-through-coarse}, the morphism $f_i$ factors
through a twisted stable map
$\cC_i \to \cP_i := \ul{\fP}\times_{\BC}C_i$, where $\fP_i$ is a
proper Deligne--Mumford stack over $S$.
By the properness of \cite[Theorem~1.4.1]{AbVi02}, to show that $f_1$
and $f_2$ are canonically isomorphic, it suffices to show that the two
coarse curves $C_1$ and $C_2$ extending $C_{\eta}$ are canonically
isomorphic.

\subsubsection{Merging two maps}
Let $C_3$ be a family of prestable curves over $\spec R$ extending
$C_{\eta}$ with dominant morphisms $C_3 \to C_i$ for $i=1,2$.
We may assume $C_3$ has no rational components with at most two
special points contracted in both $C_1$ and $C_2$ by further
contracting these components.

Let $\cC_3 \to \cC_1\times \cC_2\times C_3$ be the family of twisted
stable maps over $\spec R$ extending the obvious one
$\cC_\eta \to \cC_1\times \cC_2\times C_3$.
Observe that the composition
$\cC_3 \to \cC_1\times \cC_2\times C_3 \to C_3$ is the coarse moduli
morphism.
Indeed, if there is a component of $\cC_3$ contracted in $C_3$, then
it will be contracted in both $\cC_1$ and $\cC_2$ as well.

Set $U_i^{(0)} = C_3$ for $i=1,2$. Let $U_i^{(k+1)}$ be obtained by
removing from $U_i^{(k)}$ the rational components with precisely one
special point in $U_i^{(k)}$ and that are contracted in $C_i$.
Note that these removed rational components need not be proper, and
their closure may have more than one special points in $C_3$.
We observe that this process must stop after finitely many steps.
Denote by $U_i \subset C_3$ the resulting open subset.

\begin{lemma}\label{lem:cover}
$U_1 \cup U_2 = C_3$.
\end{lemma}
\begin{proof}
Suppose $z \in C_3 \setminus (U_1\cup U_2) \neq \emptyset$. Then there is a tree of rational curves in $C_3$ attached to $z$ and contracted in both $C_1$ and $C_2$. This contradicts the assumption on $C_3$.
\end{proof}

We then construct an underlying $R$-map $f_3\colon \cC_3 \to \fP$ by
merging $f_1$ and $f_2$ as follows.
Denote by $\cU_i := \cC_3\times_{C_3}U_i$ for $i=1,2$.
Note that $U_{i} \to C_i$ hence $\cU_i \to \cC_i$ contracts only
rational components with precisely two special points in $U_i$.
In particular, we have
$\omega^{\log}_{\cC_3/S}|_{\cU_i} = \omega^{\log}_{\cC_i/S}|_{\cU_i}$.
This leads to a commutative diagram below
\begin{equation}\label{diag:merging-maps}
\xymatrix{
 &&&& \fP \ar[d] \\
\cU_i \ar@/^1.5pc/[urrrr]^{f_3|_{\cU_i}} \ar[rr] \ar@/_1pc/[rrrr]_{\omega^{\log}_{\cC_3/S}|_{\cU_i}} && \cC_i \ar@/^.5pc/[rru]^{f_i} \ar[rr]^{\omega^{\log}_{\cC_i/S}} && \BC.
}
\end{equation}
where $f_{3}|_{\cU_i}$ is given by the obvious composition. We then observe that the two morphisms $f_3|_{\cU_1}$ and $f_3|_{\cU_2}$ coincide along $\cU_1\cap \cU_2$. Indeed, both morphisms $f_3|_{\cU_i}$ restrict to $f_{\eta}$ over the open dense set $\cC_{\eta} \subset \cU_i$, and $\fP \to \BC$ is proper of Deligne--Mumford type. Thus, $f_3|_{\cU_1}$ and $f_3|_{\cU_2}$ can be glued to an underlying $R$-map $f_3\colon \cC_3 \to \fP$ over $S$.

\subsubsection{Comparing the underlying $R$-maps}
Denote by $\ol{\cU_{i,s}}$ the closure of the closed fiber $\cU_{i,s}$ in $\cC_3$.

\begin{lemma}\label{lem:degree-comparison}
Notations as above, we have
  \begin{equation}\label{eq:degree-comparison}
    \deg \left(\omega^{\log}_{\cC_3/S} \otimes f_{3}^*\cO(\ttwist\infty_{\fP})\right)\big|_{\overline{\cU_{i,s}}} 
    \geq \deg \left(\omega^{\log}_{\cC_i/S} \otimes f_{i}^*\cO(\ttwist\infty_{\fP})\right)\big|_{\cC_{i,s}}
   \end{equation}
\end{lemma}
\begin{proof}
We prove \eqref{eq:degree-comparison} by checking the following
  \begin{equation}\label{eq:degree-comparison-local}
    \deg \left(\omega^{\log}_{\cC_3/S} \otimes f_3^*\cO(\ttwist\infty_{\fP})\right)\big|_{\cZ} \\
    \geq \deg \left(\omega^{\log}_{\cC_i/S} \otimes  f_i^*\cO(\ttwist\infty_{\fP})\right)\big|_{\cZ}
  \end{equation}
 for each irreducible component $\cZ \subset \ol{\cU_{i,s}}$.

 Since $\cU_{i,s} \to \cC_{i,s}$ is a dominant morphism contracting
 rational components with precisely two special points in $\cU_{i,s}$,
 there is an effective divisor $D'$ of $\cC_{3,s}$ supported on
 $\overline{\cU_{i}} \setminus \cU_{i}$ such that
\begin{equation}\label{equ:comparing-omega}
\omega^{\log}_{\cC_3/S}|_{\overline{\cU_{i,s}}} = \omega^{\log}_{\cC_i/S}|_{\overline{\cU_{i,s}}}(D').
\end{equation}
Restricting to $\cZ$, we obtain
\begin{equation}\label{equ:component-comparing-omega}
\deg \omega^{\log}_{\cC_3/S}|_{\cZ} = \deg \omega^{\log}_{\cC_i/S}|_{\cZ} + \deg D'|_{\cZ}.
\end{equation}
Suppose $\cZ$ is contracted in $\cC_i$. By \eqref{diag:merging-maps}, we obtain
\[
\deg f_3^*\cO(\infty_{\fP})|_{\cZ} = \deg f_i^*\cO(\infty_{\fP})|_{\cZ} = 0.
\]
Then \eqref{eq:degree-comparison-local} immediately follows from \eqref{equ:component-comparing-omega}.

Now assume $\cZ$ is mapped to an irreducible component $\cZ' \subset \cC_i$. Consider the case that $f_{3}(\cZ) \subset \infty_{\fP}$, hence $f_i(\cZ') \subset \infty_{\fP}$ by \eqref{diag:merging-maps}. By \eqref{equ:curve-in-infinity}, we obtain the equality in \eqref{eq:degree-comparison-local}.

It remains to consider the case $f_{3}(\cZ) \not\subset \infty_{\fP}$.
Let $\cL_3$ and $\cL_i$ be the corresponding spin structures over $\cC_3$
and $\cC_i$ respectively, see Proposition \ref{prop:map-field-equiv}.
Note that $\cL_3|_{\cU_i} \cong \cL_i|_{\cU_i}$ by
\eqref{diag:merging-maps}.
By \eqref{equ:comparing-omega} and Definition \ref{def:spin}, there is
an effective divisor $D$ supported on $\ol{\cU_i} \setminus \cU_i$
such that $r\cdot D = D'$ and
$\cL_{3}|_{\cZ} \cong \cL_i|_{\cZ}(D|_{\cZ})$. By
\eqref{equ:proj-bundle}, we have
\[
\deg f_3^*\cO(\infty_{\fP})|_{\cZ} - \deg f_i^*\cO(\infty_{\fP})|_{\cZ'} \geq \frac{1}{\etwist}\deg D|_{\cZ}.
\]
Combining this with \eqref{equ:component-comparing-omega}, we obtain \eqref{eq:degree-comparison-local}.
\end{proof}

Suppose $C_1 \neq C_2$.
Then we have $U_i \neq C_i$ for some $i$, say $i=1$.
By construction each connected component of $C_3 \setminus U_1$ is a
tree of proper rational curves in $U_2$ with no marked point, hence
$\cT := (\cC_3 \setminus \cU_1) \subset \cU_2$.

By construction, the composition $\cT \to \cC_3 \to \cC_2$ is a closed
immersion and $f_{3}|_{\cT} = f_{2}|_{\cT}$.
Since $\deg \omega^{\log}_{\cC_3/S}|_{\cT} < 0$ (unless
$\cT = \emptyset$), and $\cT$ is contracted to $\cC_1$ and hence maps
to a point in $\cX$, the stability of $f_2$ implies
\[
  \deg\left(\omega^{\log}_{\cC_3/S} \otimes f_3^*\cO(\ttwist\infty_{\fP})\right)\big|_{\cT}
  = \deg\left(\omega^{\log}_{\cC_2/S} \otimes f_2^*\cO(\ttwist\infty_{\fP})\right)\big|_{\cT} > 0.
\]
Using Lemma~\ref{lem:degree-comparison}, We calculate
\begin{multline*}
  \deg\left(\omega^{\log}_{\cC_3/S} \otimes f_{3}^*\cO(\ttwist\infty_{\fP})\right)\big|_{\cC_{3,s}} \\
  = \deg\left(\omega^{\log}_{\cC_3/S} \otimes f_3^*\cO(\ttwist\infty_{\fP})\right)\big|_{\overline{\cU_{1, s}}} + \deg\left(\omega^{\log}_{\cC_3/S} \otimes f_3^*\cO(\ttwist\infty_{\fP})\right) \big|_{\cT} \\
  \geq  \deg\left(\omega^{\log}_{\cC_1/S} \otimes f_1^*\cO(\ttwist\infty_{\fP})\right)\big|_{\cC_{1,s}} + \deg\left(\omega^{\log}_{\cC_3/S} \otimes f_3^*\cO(\ttwist\infty_{\fP})\right)\big|_\cT.
\end{multline*}
Since
$\deg f_{3,s}^*\cO(\ttwist\infty_{\fP}) = \deg
f_{1,s}^*\cO(\ttwist\infty_{\fP})$ is given by the sum of contact orders, we conclude that
$\cT = \cC_3 \setminus \cU_1 = \emptyset$.

Observe that $C_3 = U_1 \to C_1$ contracts proper rational components
with precisely two special points.
Let $Z \subset C_3$ be such a component, and let
$\cZ = Z \times_{C_3}\cC_3$.
Since $f_3|_{\cC_3 = \cU_1}$ factors through $f_1$, we have
\begin{equation}\label{equ:bridge-trivial-degree}
  \deg f_3^*\cO(\infty_{\fP})|_{\cZ} = 0.
\end{equation}

On the other hand, since $Z$ has two special points in $C_3$ and is
contracted in $C_1$, it is not contracted in $C_2$.
Denote by $\cZ' \subset \cC_2$ the component dominating
$Z \subset C_2$.
Then $\cZ'$ has precisely two special points.
Furthermore $f_2|_{\cZ'}$ and $f_3|_{\cZ}$ coincide away from
the two special points.
Using \eqref{equ:bridge-trivial-degree}, we observe that
$\deg f_2^*\cO(\infty_{\fP})|_{\cZ'} = 0$, which contradicts the
stability of $f_2$.
Thus $C_3 \to C_1$ is an isomorphism.

This finishes the proof of separatedness.

\subsection{Rigidifying (pre-)stable reductions}\label{sss:representability}
We start constructing the stable limit as in Theorem \ref{thm:weak-valuative}. Recall from Section \ref{ss:valuative-set-up} that it suffices to construct an extension of the underlying structures where $\cC_{\eta}$ is smooth and irreducible. Suppose we have   an underlying R-map extending $f_{\eta}$:
\begin{equation}\label{equ:underlying-triangle}
  f \colon \cC \to \fP
\end{equation}
where $\cC \to S$ is a pre-stable curve over $S$.
We modify $f$ to obtain a representable morphism to $\fP$ as follows.

Forming the relative coarse moduli \cite[Theorem~3.1]{AOV11}, we
obtain a diagram
\begin{equation*}
  \xymatrix{
    \cC \ar[d]_\pi \ar[r]^f & \fP \ar[d] \\
    \cC^r \ar[ur]^{f^r} \ar[r] & \BC,
  }
\end{equation*}
in which the upper triangle is commutative, $f^r$ is representable,
and $\pi$ is proper and quasi-finite.
Note that since
$\omega^{\log}_{C/S}|_{\cC^r} = \omega^{\log}_{\cC^r/S}$, the lower
triangle is also commutative.

\begin{proposition}\label{prop:valuative-representability}
  Notations as above, we have:
  \begin{enumerate}
  \item $f^r$ is a representable underlying $R$-map over $S$ extending $f_{\eta}$.
  \item If $f$ satisfies the positivity condition
    \eqref{equ:hyb-stability}, then so does $f^r$.
  \end{enumerate}
\end{proposition}
\begin{proof}
  Both parts follow easily from the above observations.
\end{proof}

\subsection{Pre-stable reduction}

Next, we construct a (not necessarily stable) family
\eqref{equ:underlying-triangle} across the central fiber.
We will show that such a family can be constructed by taking the
stable map limit twice in a suitable way.
It is worth mentioning that the method here is very different than the
one in \cite{CJRS18P} in order to handle the general situation of this
paper.

In the following, $g_{\eta}$ and $\cL_{\eta}$ denote the pre-stable
map and the spin structure on $\cC_{\eta}$ associated to $f_\eta$.

\subsubsection{The first stable map limit}\label{sss:stable-map-limit-1}

Let $g_\eta$ be the prestable map underlying $f_\eta$.
Then, let $g_0\colon \cC_0 \to \cX$ be any pre-stable map extending
$g_{\eta}$ whose existence follows from \cite{AbVi02}.
Possibly after a further base change, we construct the following
commutative diagram:
\begin{equation}\label{diag:stable-map-limits-1}
\xymatrix{
\cC_0'' \ar[d] \ar[rr]^{f_0} && \fP \ar[d] \\
\cC_0' \ar[d] \ar[rr]^{\cL_0} && \fX \ar[d] \\
\cC_0 \ar[rr]^{(\omega^{\log}_{\cC_0/S},g_0) \ \ \ \ } && \BC\times\cX
}
\end{equation}

First, there is a unique stable map limit
$\cC_0' \to \cC_0\times_{\BC\times\cX}\fX$ extending the one given
by the spin $\cL_{\eta}$.
This yields the spin structure $\cL_0$ on $\cC_0'$.
Furthermore, the morphism $\cC'_0 \to \cC_0$ is quasi-finite.
We then take the unique stable map limit
$h\colon \cC_0'' \to \cP_{\cC_0'} := \fP\times_{\fX}\cC_0'$ extending the
one given by $f_{\eta}$.

To see the difference between the above stable map limits and a pre-stable reduction, we observe:
\begin{lemma}\label{lem:unfolding-bridges}
Suppose we are given a commutative diagram
\begin{equation}
\xymatrix{
\cC \ar[d] \ar[rr]^{f} && \fP \ar[d] \\
\cC' \ar[rr]^{\omega^{\log}_{\cC'/S}} && \BC
}
\end{equation}
where $\cC'$ and $\cC$ are two pre-stable curves over $S$ such that $\cC \to \cC'$ contracts only rational components with two special points. Then $f$ is a pre-stable reduction as in \eqref{equ:underlying-triangle}.
\end{lemma}
\begin{proof}
The lemma follows from that $\omega^{\log}_{\cC'/S}|_{\cC} = \omega^{\log}_{\cC/S}$.
\end{proof}

Observe that
$\omega^{\log}_{\cC_0/S}|_{\cC'_0} = \omega^{\log}_{\cC'_0/S}$.
If $\cC''_0 \to \cC'_0$ contracts no rational tails then it can only
contract rational bridges.
Thus we obtain a pre-stable reduction in this case by applying Lemma
\ref{lem:unfolding-bridges}.
Otherwise, we show that a pre-stable reduction can be achieved by
repeating stable map limits one more time as follows.

\subsubsection{The second stable map limit}
Set $\cC_1 = \cC''_0$. We will construct the following commutative diagram:
\begin{equation}\label{diag:stable-map-limits-2}
\xymatrix{
\cC_2 \ar[d] \ar[rrrr]^{f_1} &&&& \fP \ar[d] \\
\tilde{\cC}_1' \ar[rr] && \cC_1' \ar[d] \ar[rr]^{\cL_1} && \fX \ar[d] \\
&& \cC_1 \ar[rr]^{(\omega^{\log}_{\cC_1/S},g_1) \ \ \ \ } && \BC\times\cX
}
\end{equation}

First, $g_1$ is the composition of $\cC_1 \to \cC_0$ with $g_0$, and $\cL_1$ is the spin structure over $\cC'_1$ obtained by taking the stable map limit as in \eqref{diag:stable-map-limits-1}.

\smallskip

Second, we construct a quasi-finite morphism of pre-stable curves  $\tilde{\cC}_1' \to \cC_1'$ over $S$ such that over $\eta$ it is the identity $\cC_{\eta} \to \cC_{\eta}$, and the identity $\cL_{\eta} \to \cL_{\eta}$ extends to a morphism of line bundles
\begin{equation}\label{equ:spin-iterate}
  \cL_0|_{\tilde{\cC}'_1} \to \cL_1|_{\tilde{\cC}'_1}
\end{equation}
whose $r$-th power is the natural morphism
\begin{equation}\label{equ:log-cot-iterate}
  (\omega_{\cC_0/S}^{\log} \otimes g_0^* \bL^\vee)|_{\tilde{\cC}'_1} \to
  (\omega_{\cC_1/S}^{\log} \otimes g_1^* \bL^\vee)|_{\tilde{\cC}'_1}.
\end{equation}

Let $\sqrt[r]{\cC_1'}$ be the $r$th root stack of
$$(\omega_{\cC_0/S}^{\log} \otimes g_0^* \bL^\vee)^{\vee}|_{\tilde{\cC}'_1} \otimes
  (\omega_{\cC_1/S}^{\log} \otimes g_1^* \bL^\vee)|_{\tilde{\cC}'_1},$$
and $\sqrt[r]{(\cC_1',s)}$ be the $r$-th root
stack of the section $s$ of the above line bundle given by \eqref{equ:log-cot-iterate}. We form the fiber product
\[
  \hat{\cC}'_1 := \cC'_1\times_{\sqrt[r]{\cC_1'}}\sqrt[r]{(\cC_1',s)},
\]
where the morphism $\cC'_1 \to \sqrt[r]{\cC_1'}$ is defined via
$\cL_0^\vee|_{\tilde{\cC}'_1} \otimes \cL_1|_{\tilde{\cC}'_1}$.
The identities
$\cL_{\eta} = \cL|_{\cC_{\eta}} = \cL_{1}|_{\cL_{\eta}}$ induce a
stable map $\cC_{\eta} \to \hat{\cC}'_1$ which, possibly after a
finite base change of $S$, extends to a quasi-finite stable map
$\tilde{\cC}'_1 \to \hat{\cC}'_1$.
Since $\hat{\cC}'_{1} \to \cC'_1$ is quasi-finite, the composition
$\tilde{\cC}'_1 \to \hat{\cC}'_1 \to \cC'_1$ gives the desired
quasi-finite morphism.
Thus, $\cL_1$ pulls back to a spin structure on $\tilde{\cC}'_1$.
Furthermore, the universal $r$-th root of $\sqrt[r]{(\cC_1',s)}$ pulls
back to a section of $\cL_1 \otimes \cL^\vee|_{\tilde{\cC}'_1}$ as
needed.

\smallskip

Finally, we construct $f_1$ in the same way as the stable map limit in
\eqref{diag:stable-map-limits-1} but using the spin structure
$\cL_1|_{\tilde{\cC}'_1}$.
We will show:

\begin{proposition}\label{prop:pre-stable-reduction}
The morphism $\cC_2 \to \tilde{\cC}'_1$ contracts no rational tails.
\end{proposition}
Together with Lemma \ref{lem:unfolding-bridges}, we obtain a pre-stable reduction \eqref{equ:underlying-triangle}.

\subsubsection{The targets of the two limits}

Consider $ \cP_i := \tilde{\cC}'_1\times_{\fX} \fP$ for $i=0,1$,
where the arrow $\tilde{\cC}'_1 \to \fX$ is induced by $\cL_i$.
The morphism \eqref{equ:spin-iterate} induces a birational map
$ c\colon \cP_0 \dashrightarrow \cP_1 $ whose indeterminacy locus is
precisely the infinity divisor $\infty_{\cP_0} \subset \cP_0$ over the
degeneracy locus of \eqref{equ:spin-iterate}.
Its inverse $c^{-1}\colon \cP_1 \dashrightarrow \cP_0$ is given by the
composition
\begin{multline*}
 \cP_1 =  \PP^\bw\left(\bigoplus_{j > 0} (g_1^*(\bE_j^\vee) \otimes \cL_1^{\otimes j})|_{\tilde{\cC}'_1} \oplus \cO\right) \\
  \dashrightarrow \PP^\bw\left(\bigoplus_{j > 0} (g_1^*(\bE_j^\vee) \otimes \cL_1^{\otimes j})|_{\tilde{\cC}'_1} \oplus (\cL_1 \otimes \cL_0^\vee)^{\otimes \etwist}|_{\tilde{\cC}'_1}\right) \\
  \cong \PP^\bw\left(\bigoplus_{j > 0} (g_0^*(\bE_j^\vee) \otimes \cL_0^{\otimes j})|_{\tilde{\cC}'_1} \oplus \cO\right) = \cP_0,
\end{multline*}
where the first map is multiplication of the last coordinate by the
$\etwist$th power of the section of
$(\cL_1 \otimes \cL_0^\vee)|_{\tilde{\cC}'_1}$ given by
\eqref{equ:spin-iterate}.
Therefore, the indeterminacy locus of $c^{-1}$ is the zero section
$\zero_{\cP_1} \subset \cP_1$ over the degeneracy locus of
\eqref{equ:spin-iterate}.

We have arrived at the following commutative diagram
\begin{equation}\label{diag:compare-target}
\xymatrix{
 && \cP_0 \ar@/_1pc/@{-->}[dd]_{c} \ar@/^1pc/@{<--}[dd]^{c^{-1}} \ar[rrd] && \\
\cC_2 \ar[rru]^{f_0} \ar[rrd]_{f_1} &&&& \tilde{\cC}'_1 \\
  && \cP_1  \ar[rru] &&
}
\end{equation}
where by abuse of notations $f_0$ and $f_1$ are given by the
corresponding arrows in \eqref{diag:stable-map-limits-1} and
\eqref{diag:stable-map-limits-2}.
Indeed, $f_0\colon \cC_2 \to \cP_0$ is given by the composition
$ \cC_2 \to \tilde{\cC}'_1 \to \cC'_1 \to \cC_1 \to \fP.
$

\subsubsection{Comparing the two limits along vertical rational tails} A rational tail of $\cC_2$ over the closed fiber is called \emph{vertical} if it is contracted in $\tilde{\cC}'_1$.

\begin{lemma}\label{lem:VRT-in-infty}
If $\cZ \subset \cC_2$ is a vertical rational tail, then $f_0(\cZ) \subset \infty_{\cP_0}$.
\end{lemma}
\begin{proof}
Note that $f_0$ contracts any vertical rational tails. Suppose $f_0(\cZ) \not\subset \infty_{\cP_0}$. Then $c\circ f_0$ is well-defined along $\cZ$ hence $f_1|_\cZ = c\circ f_0|_{\cZ}$. This contradicts the stability of $f_1$ as a stable map limit.
\end{proof}

For $i=0,1$, denote by $p_i\colon \cP_i \dashrightarrow \infty_{\cP_i}$ the projection from the zero section $\zero_{\cP_i}$ to $\infty_{\cP_i}$. Thus $p_i$ is a rational map well-defined away from $\zero_{\cP_i}$. Furthermore, we observe that $\infty_{\cP_0} \cong \infty_{\cP_1}$. Using this isomorphism, we have $p_0 = p_1\circ c$ and $p_1 = p_0\circ c^{-1}$.

\begin{lemma}\label{lem:const-proj}
Let $\cZ \subset \cC_2$ be a vertical rational tail. Then $p_1\circ f_1$ contracts an open dense subset of $\cZ$.
\end{lemma}
\begin{proof}
Since $f_1$ is a stable map and $\cZ$ is a vertical rational tail, we have $f_1(\cZ) \not\subset \zero_{\cP_1}$. Thus $p_1\circ f_1$ is well-defined on an open dense $U \subset \cZ$ such that $f_1(U)$ avoids $\zero_{\cP_1}$. Observe that $c^{-1}$ is well-defined on $f_1(U)$. We then have
$
p_1 \circ f_1 |_U = p_0\circ c^{-1} \circ f_1 |_U = p_0 \circ f_0|_U.
$
Here $p_0$ is well-defined on $f_0|_U$ by Lemma \ref{lem:VRT-in-infty}. The statement follows from that $f_0$ contracts any vertical rational tail.
\end{proof}

\begin{corollary}\label{cor:no-VRT-in-infty}
  If $\cZ \subset \cC_2$ is a vertical rational tail, then the image
  $f_1(\cZ)$ dominates a line joining $\zero_{\cP_1}$ and a point on
  $\infty_{\cP_1}$.
\end{corollary}
\begin{proof}
  By Lemma \ref{lem:const-proj}, $f_1(\cZ)$ has support on a fiber of
  $p_1$.
  Since $\cZ$ intersects $\infty_{\cP_1}$ at its unique node, it
  suffices to show that $f_1(\cZ) \not\subset \infty_{\cP_1}$ hence
  $f_1|_{\cZ}$ dominates a fiber of $p_1$.
  Otherwise, since $p_1|_{\infty_{\cP_1}}$ is the identity, $f_1$
  contracts $\cZ$ by Lemma \ref{lem:const-proj}.
  This contradicts the stability of $f_1$ constructed as a stable map
  limit.
\end{proof}

\begin{proof}[Proof of Proposition \ref{prop:pre-stable-reduction}]
  We show that Corollary~\ref{cor:no-VRT-in-infty} and
  Lemma~\ref{lem:VRT-in-infty} contradict each other, hence rule out
  the existence of vertical rational tails.

  Let $\cZ\subset \cC_2$ be a vertical rational tail.
  The pre-stable map $\cC_2 \to \cX$ factors through $\cC_2 \to \cC_1$
  along which $\cZ$ is contracted to a smooth unmarked point on
  $\cC_1$.
  Thus there is a Zariski neighborhood $U \subset \cC_2$ containing
  $q$ such that $\bE_j|_{U}$ splits.
  Denote by $\{H_{ijk}\}_{k=1}^{\rk \bE_j}$ the collection of
  hyperplanes in $\cP_i|_{U}$ corresponding to each splitting factor
  of $\bE_j|_{U}$.

  By Corollary~\ref{cor:no-VRT-in-infty} there is a smooth unmarked
  point $q \in \cZ$ such that $f_1(q) \in \zero_{\cP_1}$, hence
  $f_1(q) \in H_{1jk}$ for all $j$ and $k$.
  We will show that $f_0(q) \in H_{0jk}$ for all $j$ and $k$ as well,
  hence $f_0(q) \in \zero_{\cP_0}$, which contradicts
  Lemma~\ref{lem:VRT-in-infty}.

Suppose $\cZ$ intersects $H_{1jk}$ properly at $q$ via $f_1$. Let $D_{1jk} \subset U$ be an irreducible component of $f_1^*(H_{1jk})$ containing $q$. Then $D_{1jk}$ is a multi-section over $S$ with the general fiber $D_{1jk,\eta} \subset f_{1,\eta}^*(H_{1jk,\eta}) = f_{0,\eta}^*(H_{0jk,\eta})$. Taking closure, we observe that  $f_0(q) \in D_{1jk} \subset f_{0}^*(H_{0jk})$.

Suppose $f_1(\cZ) \subset H_{1jk}$. Note that $p_1\circ f_1 = p_0 \circ c^{-1} \circ f_1 = p_0 \circ f_0$ are well-defined along an open dense subset of $\cZ$. Then Lemma \ref{lem:VRT-in-infty} together with $p_1 \circ f_1(\cZ) \subset \infty_{\cP_1} \cap H_{1jk} \cong \infty_{\cP_0} \cap H_{0jk}$ implies that $f_0$ contracts $\cZ$ to a point of $\infty_{\cP_0} \cap H_{0jk}$.
\end{proof}

\subsection{Stabilization}

Let $f\colon \cC \to \fP$ be a pre-stable reduction extending
$f_{\eta}$ over $S$ as in \eqref{equ:underlying-triangle}.
We next show that by repeatedly contracting unstable rational bridges
and rational tails as in Section~\ref{sss:stabilize-bridge} and
\ref{sss:stabilize-tail}, we obtain a pre-stable reduction satisfying
\eqref{equ:hyb-stability}.
Together with Proposition~\ref{prop:valuative-representability}, this
will complete the proof of Theorem~\ref{thm:weak-valuative}.

\subsubsection{Stabilizing unstable rational bridges}\label{sss:stabilize-bridge}
Let $\cZ \subset \cC$ be an unstable rational bridge. We contract $\cZ$ as follows. Consider $\cC \to C \to C'$ where the first arrow takes the coarse curve, and the second arrow contracts the component corresponding to $\cZ$. Since $\omega^{log}_{C'/S}|_{\cC} = \omega^{log}_{\cC/S}$, we have a commutative diagram
\[
\xymatrix{
\cC \ar@/^1pc/[drr]^{f} \ar@/_1pc/[rdd] \ar@{-->}[dr]^{f_{C'}}&& \\
& \fP_{C'} \ar[r] \ar[d] & \fP \ar[d] \\
& C' \ar[r]^{\omega^{log}_{C'/S}} & \BC
}
\]
where the square is cartesian and the dashed arrow $f_{C'}$ is induced by the fiber product. By Corollary \ref{cor:rat-bridge-stability}, $\cZ$ is contracted along $f_{C'}$.

Note that $f_{\eta}\colon \cC_{\eta} \to \fP$ yields a stable map $\cC_{\eta} \to \fP_{C'}$ which, possibly after a finite base change, extends to a stable map $f'_{C'}\colon \cC' \to \fP_{C'}$. Let $q \in C'$ be the node to which $\cZ$ contracts.

\begin{lemma}\label{lem:remove-unstable-bridge}
The composition $\cC' \to \fP_{C'} \to C'$ is the coarse moduli morphism. Furthermore, let $\tilde{q} \in \cC'$ be the node above $q \in C'$. Then we have $f|_{\cC\setminus \cZ} = f'|_{\cC'\setminus \{q\}}$.
\end{lemma}
\begin{proof}
Let $\bar{f}'\colon \bar{\cC}' \to P_{C'}$ be the coarse stable map of $f'_{C'}$, and $\bar{f}_{C'}\colon C \to P_{C'}$ be the coarse stable map of $f_{C'}$. Thus $\bar{f}'$ is the stabilization of $\bar{f}_{C'}$ as a stable map. By construction,  the image of $\cZ$ in $C$ is the only unstable component of $\bar{f}_{C'}$, hence is the only component contracted along $C \to \bar{\cC}'$. Therefore $\cC' \to C'$ is the coarse moduli. Since the modification is local along $\cZ$, the second statement follows from the first one.
\end{proof}

Let $f'$ be the composition $\cC' \to \fP_{C'} \to \fP$. The above lemma implies that $\omega^{\log}_{\cC'/S} = \omega^{\log}_{C'/S}|_{\cC'}$. Thus $f'\colon \cC' \to \fP$ is a new pre-stable reduction extending $f_{\eta}$ with $\cZ$ removed.

\subsubsection{Stabilizing rational tails}\label{sss:stabilize-tail}
Let $\cZ \subset \cC$ be an unstable rational tail, $\cC \to \cC'$ be the contraction of $\cZ$, and $p \in \cC'$ be the image of $\cZ$. Possibly after a finite extension, we take the stable map limit
\[
f'\colon \cC'' \to \fP_{\cC'} := \cC'\times_{\BC}\fP
\]
extending the one induced by $f_{\eta}$.
We will also use $f'\colon \cC'' \to \fP$ for the corresponding
morphism.
Let $\cT \subset \cC''$ be the tree of rational components contracted
to $p$.
Since $f'$ is a modification of $f$ around $\cZ$, we observe that
$f'_{\cC'' \setminus \cT} = f|_{\cC\setminus \cZ}$.

\begin{proposition}\label{prop:stabilize-tails}
The composition $\cC'' \to \fP_{\cC'} \to \cC'$ is the identity.   Therefore, $f'\colon \cC' \to \fP$ is a pre-stable reduction extending $f_{\eta}$ with $\cZ$ contracted but everywhere else is identical to $f$.
\end{proposition}

The proof of the above proposition occupies the rest of this section. Since $p$ is a smooth unmarked point of $\cC'$, it suffices to show that $\cT$ contains no component. We first consider the following case.

\begin{lemma}
Notations and assumptions as above, suppose that $f(\cZ) \subset \zero_{\fP}$. Then Proposition \ref{prop:stabilize-tails} holds.
\end{lemma}
\begin{proof}
  Since $f'_{\cC'' \setminus \cT} = f|_{\cC\setminus \cZ}$, the
  assumption implies
  $\deg f^*\cO(\infty_{\fP}) \leq \deg
  (f')^*\cO(\infty_{\fP})|_{\overline{\cC''_s\setminus\cT}}$.
  On the other hand, we have
  $\deg (f')^*\cO(\infty_{\fP})|_{\cT} \geq 0$, and ``$=$'' iff $\cT$
  is a single point.
  Thus, the lemma follows from
  $ \deg f^*\cO(\infty_{\fP}) = \deg
  (f')^*\cO(\infty_{\fP})|_{\overline{\cC''_s\setminus\cT}} + \deg
  (f')^*\cO(\infty_{\fP})|_{\cT} $.
\end{proof}

We now impose the condition $f(\cZ) \not\subset \zero_{\fP}$.
Observe that the pre-stable map $g\colon \cC \to \cX$ contracts $\cZ$,
hence factors through a pre-stable map $g'\colon \cC' \to \cX$.
Since $p$ is a smooth unmarked point, we may choose a Zariski
neighborhood $U' \subset \cC'$ of $p$ such that $(g')^*\bE_i|_{U'}$
splits for each $i$.
Denote by $U = U'\times_{\cC'} \cC$.
Then $g^*\bE_i|_{U}$ splits as well for each $i$.
The $j$-th splitting factors of $\oplus\bE_i|_{U}$ and
$\oplus\bE_i|_{U'}$ define families of hyperplanes
\begin{equation}\label{equ:local-hyperplane}
H_j \subset \fP_U, \ \ \ \mbox{and} \ \ \ H'_j \subset \fP_{U'}
\end{equation}
over $U$ and $U'$ respectively for $j = 1, 2, \cdots, n$.

\begin{lemma}\label{lem:unstable-tail}
Notations and assumptions as above, for each $j$ we have $\deg (f^*H_j)|_{\cZ} \leq 0$. In particular, $f(\cZ) \not\subset \zero_{\fP}$ implies that $f(\cZ) \cap \zero_{\fP} = \emptyset$.
\end{lemma}
\begin{proof}
Observe that $\bigcap_j H_j$ is the zero section $\zero_{\fP_U}$. Thus, it suffices to show that $\deg (f^*H_j)|_{\cZ} \leq 0$ for each $j$.

Since $\cZ$ is contracted by $f$, $\bE_i$ and $\bL$ are both trivial along $\cZ$. Thus, we have $\fP_{\cZ} = \PP^\bw(\oplus_j \cL_{\cZ}^{\otimes i_j} \oplus \cO)$ where the direct sum is given by the splitting of $\bE_i$ for all $i$. The corresponding section $\cZ \to \fP_{\cZ}$ is defined by a collection of sections $(s_1, \cdots, s_n, s_{\infty})$ with no base point, where $s_j \in H^0(\cL^{i_j}\otimes f^*\cO(w_{i_j} \infty_{\fP})|_{\cZ})$ and $s_{\infty} \in H^0(f^*\cO(\infty_{\fP})|_{\cZ})$. In particular, we have $f^*\cO(H_j)|_{\cZ} = \cL^{i_j}\otimes f^*\cO(w_{i_j} \infty_{\fP})|_{\cZ}$. Note that $w_{i_j} = \etwist \cdot i$ by the choice of weights \eqref{equ:universal-proj}. We calculate
\[
  (\cL^{i_j}\otimes f^*\cO(w_{i_j} \infty_{\fP})|_{\cZ})^{r}
  = (\cL \otimes f^*\cO(\infty_{\fP})|_{\cZ})^{\ttwist i}
  = \big(\omega^{\log}_{\cC/S}\otimes f^*\cO(\ttwist \infty_{\fP})|_{\cZ}\big)^{ i}.
\]
Since $\cZ$ is unstable, we have
$\deg \omega^{\log}_{\cC/S}\otimes f^*\cO(\ttwist \infty_{\fP})|_{\cZ}
\leq 0$, which implies $\deg (f^*H_j)|_{\cZ} \leq 0$.
\end{proof}

To further proceed, consider the spin structure $\cL'$ over $\cC'$ and
observe that $\cL'|_{\cC'\setminus \{p\}} = \cL|_{\cC\setminus \cZ}$.
Using the same construction as for \eqref{equ:spin-iterate}, we obtain
a quasi-finite morphism $\widetilde{\cC} \to \cC$ between two pre-stable
curves over $S$ which is isomorphic away from $\cZ$ and its pre-image
in $\widetilde{\cC}$, and a canonical morphism of line bundles
$\cL'|_{\widetilde{\cC}} \to \cL|_{\widetilde{\cC}}$ extending the identity
$\cL'|_{\cC'\setminus \{p\}} = \cL|_{\cC\setminus \cZ}$, whose $r$-th
power is the canonical morphism
$\omega^{\log}_{\cC'/S}|_{\widetilde{\cC}} \to
\omega^{\log}_{\cC/S}|_{\widetilde{\cC}}$.
Define:
\[
\fP_{\widetilde{\cC}} := \fP\times_{\BC}\widetilde{\cC} \ \ \ \mbox{and} \ \ \ \fP'_{\widetilde{\cC}} := \fP_{\cC'}\times_{\cC'}\widetilde{\cC}
\]
We have arrived at the following commutative diagram
\[
\xymatrix{
\widetilde{\cC} \ar[rr]^{\widetilde{f}}  && \fP_{\widetilde{\cC}} \ar[rrd] \ar@/^1pc/@{-->}[dd]^{c^{-1}}&&&& \\
&&&& \widetilde{\cC} \ar[rrd] && \\
\widetilde{\cC}''' \ar[rr]^{\widetilde{f}'} \ar[uu] \ar@/_1pc/[rrrrd]^{f''}&& \fP'_{\widetilde{\cC}} \ar@/^1pc/@{-->}[uu]^{c} \ar[rru] \ar[rrd] &&&& \cC' \\
&& &&\fP_{\cC'} \ar[rru] &&
}
\]
where $\widetilde{f}$ is the section obtained by pulling back $f$, $c$
and $c^{-1}$ are the two birational maps defined using
$\cL'|_{\widetilde{\cC}} \to \cL|_{\widetilde{\cC}}$ similarly to
\eqref{diag:compare-target}, and $\widetilde{f}'$ is the stable map
limit extending the one given by $f_{\eta}$.

Denote by $\cZ$ the corresponding rational tail of $\widetilde{\cC}$, and
by $\widetilde{\cZ} \subset \widetilde{\cC}'''$ the component dominating $\cZ$.
By Lemma \ref{lem:unstable-tail}, the image $\widetilde{f}(\cZ)$ avoids
$\zero_{\fP_{\widetilde{\cC}}}|_{\cZ}$ which is the indeterminacy locus of
$c^{-1}$.
This implies that
$\widetilde{f}'(\widetilde{\cZ}) \subset c^{-1}\circ \widetilde{f}(\cZ) \subset
\infty_{\fP'_{\widetilde{\cC}}}$.
Thus by the commutativity of the above diagram, any rational tail of
$\widetilde{\cC}'''$ contracted to a point on $\cZ$, is also contracted by
$\widetilde{f}'$.
Now the stability of $\widetilde{f}'$ as a stable map implies:

\begin{lemma}
$\widetilde{\cC}''' \to \widetilde{\cC}$ contracts no component.
\end{lemma}

Furthermore:
\begin{lemma}
  The rational tail $\widetilde{\cZ}$ is contracted by $f''$.
\end{lemma}
\begin{proof}
  Write $\widetilde{U} = \widetilde{\cC}\times_{\cC'}U$.
  By abuse of notations, denote by
  $H_{j} \subset \fP_{\widetilde{\cC}}$ and
  $H'_{j} \subset \fP'_{\widetilde{\cC}}$ the families of hyperplanes
  over $\widetilde{U}$ obtained by pulling back the corresponding
  hyperplanes in \eqref{equ:local-hyperplane}.
  From the construction of $c^{-1}$, we observe that
  $\widetilde{f}(\cZ) \subset H_j$ for some $H_j$ implies that
  $\widetilde{f}'(\widetilde{\cZ}) \subset H'_j$.

  Suppose $f''(\widetilde{\cZ})$ is one-dimensional.
  Then $\widetilde{\cZ}$ intersects some $H'_j$ properly and
  non-trivially.
  Since $H'_{j}$ is a family over $\widetilde U$,
  $(\widetilde{f}')^*(H'_j)$ contains a non-empty irreducible
  multi-section over $U$ which intersects $\widetilde{\cZ}$.
  Denote this multi-section by $D$.
  Consider the general fiber
  $D_{\eta} \subset f_{\eta}^*(H'_{j,\eta}) = f_{\eta}^*(H_{j,\eta})$.
  The closure $\overline{D_{\eta}} \subset \widetilde{f}^{*}H_j$
  intersects $\cZ$ non-trivially.
  By Lemma \ref{lem:unstable-tail}, we necessarily have
  $\widetilde{f}(\cZ) \subset H_j$ hence
  $\widetilde{f}'(\widetilde{\cZ}) \subset H'_{j}$ by the previous
  paragraph.
  This contradicts the assumption that $\widetilde{\cZ}$ and $H'_j$
  intersect properly.
\end{proof}

Finally, observe that the coarse pre-stable map of $f''$ factors
through the coarse stable map of $f'\colon \cC'' \to \fP_{\cC'}$.
The above two lemmas show that the unstable components of
$\widetilde{\cC}'''$ with respect to $f''$ are precisely those
contracted in $\cC'$.
Therefore, the arrow $\cC'' \to \cC'$ contracts no component.
This completes the proof of Proposition \ref{prop:stabilize-tails}.


\section{Reducing perfect obstruction theories along boundary}
\label{sec:POT-reduction}

For various applications in this and our subsequent papers
\cite{CJR20P1, CJR20P2}, we further develop a general machinery,
initiated in \cite{CJRS18P}, on reducing a perfect obstruction theory
along a Cartier divisor using cosections.
Furthermore, we prove a formula relating the two virtual cycles
defined using a perfect obstruction theory and its reduction under the
general setting in Section~\ref{ss:boundary-cycle}. Since log
structures are irrelevant in this section, we will assume all log
structures to be trivial for simplicity.

\subsection{Set-up of the reduction}\label{ss:reduction-set-up}
Throughout this section we will consider a sequence of morphisms of algebraic stacks
\begin{equation}\label{equ:stacks-reduction}
\scrM \to \fH \to \fM
\end{equation}
where $\scrM$ is a separated Deligne--Mumford stack, and the second morphism is smooth of Deligne--Mumford type.

Let $\Delta \subset \fM$ be an effective Cartier divisor, and let $\Delta_{\fH}$ and $\Delta_{\scrM}$ be its pull-backs in $\fH$ and $\scrM$ respectively. 
Let $\FF$ be the complex with amplitude $[0,1]$ over $\fM$
\[
\cO_{\fM}  \stackrel{\epsilon}{\longrightarrow} \cO_{\fM}(\Delta)
\]
where $\epsilon$ is the canonical section defining $\Delta$.

We further assume two relative perfect obstruction theories
\begin{equation}\label{equ:canonical-POT}
\varphi_{\scrM/\fM} \colon \TT_{\scrM/\fM} \to \EE_{\scrM/\fM} \ \ \ \mbox{and} \ \ \ \varphi_{\fH/\fM} \colon \TT_{\fH/\fM} \to \EE_{\fH/\fM}
\end{equation}
which fit in a commutative diagram
\begin{equation}\label{diag:compatible-POT}
\xymatrix{
\TT_{\scrM/\fM} \ar[rr] \ar[d]_{\varphi_{\scrM/\fM}} && \TT_{\fH/\fM} \ar[d]^{\varphi_{\fH/\fM}|_{\scrM}} \\
\EE_{\scrM/\fM} \ar[rr]^{\sigma^{\bullet}_{\fM}} && \EE_{\fH/\fM}
}
\end{equation}
such that $H^1(\EE_{\fH/\fM}) \cong \cO_{\fM}(\Delta)|_{\fH}$, and the following cosection
\begin{equation}\label{equ:general-cosection}
\sigma_{\fM} := H^1(\sigma^{\bullet}_{\fM}) \colon  H^1(\EE_{\scrM/\fM})  \to H^1(\EE_{\fH/\fM}|_{\scrM}) \cong \cO_{\fM}(\Delta)|_{\scrM}
\end{equation}
is surjective along $\Delta_{\scrM}$.

\subsection{The construction of the reduction}
Consider the composition
\[
\EE_{\fH/\fM} \to H^1(\EE_{\fH/\fM})[-1] \cong \cO_{\fM}(\Delta)|_{\scrM}  \twoheadrightarrow \cok(\epsilon)[-1].
\]
Since $\fH \to \fM$ is smooth, we have $\cok(\epsilon)[-1] \cong \FF|_{\fH}$. Hence the above composition defines a morphism
\begin{equation}\label{equ:reduction-map}
\EE_{\fH/\fM} \to \FF|_{\fH}
\end{equation}
over $\fH$. We form the distinguished triangles
\begin{equation}\label{equ:reduced-POT}
\EE^{\red}_{\fH/\fM} \to \EE_{\fH/\fM} \to \FF|_{\fH} \stackrel{[1]}{\to} \ \ \ \mbox{and} \ \ \  \EE^{\red}_{\scrM/\fM} \to \EE_{\scrM/\fM} \to \FF|_{\scrM} \stackrel{[1]}{\to},
\end{equation}
where the middle arrow in the second triangle is the composition of \eqref{equ:reduction-map} with $\sigma^{\bullet}_{\fM}$.

\begin{theorem}\label{thm:reduction}
Notations and assumptions as above, we have:
\begin{enumerate}
 \item There is a factorization of perfect obstruction theories
 \[
 \xymatrix{
 \TT_{*/\fM} \ar[rr]^{\varphi_{*/\fM}} \ar[rd]_{\varphi^{\red}_{*/\fM}} && \EE_{*/\fM} \\
 &\EE^{\red}_{*/\fM} \ar[ru]&
 }
 \]
such that $\varphi^{\red}_{*/\fM}|_{*\setminus\Delta_*} = \varphi_{*/\fM}|_{*\setminus\Delta_*}$  for $* = \scrM$ or $\fH$.

 \item There is a canonical commutative diagram
 \[
 \xymatrix{
 \EE^{\red}_{\scrM/\fM} \ar[rr] \ar[d]_{\sigma^{\bullet,\red}_{\fM}} && \EE_{\scrM/\fM} \ar[d]^{\sigma^{\bullet}_{\fM}} \\
 \EE^{\red}_{\fH/\fM}|_{\scrM} \ar[rr] && \EE_{\fH/\fM}|_{\scrM}
 }
 \]
 such that $H^1(\EE^{\red}_{\fH/\fM}) \cong \cO_{\fH}$. Furthermore, the {\em reduced cosection}
\[
\sigma^{\red}_{\fM}:= H^1(\sigma^{\red,\bullet}_{\fM}) \colon H^1(\EE^{\red}_{\scrM/\fM}) \to H^1(\EE^{\red}_{\fH/\fM}|_{\scrM}) \cong \cO_{\scrM}
\]
is surjective along $\Delta_{\scrM}$, and satisfies $\sigma^{\red}_{\fM}|_{\scrM\setminus\Delta_{\scrM}} = \sigma_{\fM}|_{\scrM\setminus\Delta_{\scrM}}$.
\end{enumerate}
\end{theorem}
This theorem will be proven below in Section~\ref{ss:proof-reduction}.

In case $\fM$ admits a fundamental class $[\fM]$, denote by $[\scrM]^{\vir}$ and $[\scrM]^{\red}$ the virtual cycles giving by the perfect obstruction theories $\varphi_{\scrM/\fM}$ and $\varphi^{\red}_{\scrM/\fM}$ respectively.

\begin{remark}
  In order to construct the cone $\EE^\red_{\scrM/\fM}$, instead of having the auxiliary stack $\fH$, it suffices to
  assume the existence of a cosection
  $\sigma_\fM\colon H^1(\EE_{\scrM/\fM}) \to \cO_{\fM}(\Delta)|_\scrM$.
  Furthermore, the proof of Theorem~\ref{thm:reduction} shows that if
  $\sigma_\fM$ is surjective along $\Delta_\scrM$, then
  $\EE^\red_{\scrM/\fM}$ is perfect of amplitude $[0, 1]$.
  On the other hand,  in practice the auxiliary stack $\fH$ provides a convenient criterion to
  ensure the factorization of Theorem~\ref{thm:reduction} (1).
\end{remark}

\subsection{Decending to the absolute reduced theory}\label{ss:general-absolut-theory}
We further assume $\fM$ is smooth. Consider the morphism of triangles:
\begin{equation}\label{diag:relative-to-absolute}
\xymatrix{
\TT_{{*}/\fM} \ar[r] \ar[d]_{\varphi_{*/\fM}^{\red}} & \TT_{{*}} \ar[r] \ar[d]_{\varphi^{\red}_{*}} & \TT_{\fM}|_{{*}} \ar[r]^{[1]} \ar[d]^{\cong} & \\
\EE^{\red}_{{*}/\fM} \ar[r]  & \EE^{\red}_{{*}} \ar[r]  & \TT_{\fM}|_{{*}} \ar[r]^{[1]}  &
}
\end{equation}
for $*=\fH$ or $\scrM$. By \cite[Proposition A.1.(1)]{BrLe00}, $\varphi^{\red}_\scrM$ is a perfect obstruction theory compatible with $\varphi^{\red}_{\scrM/\fM}$, hence induces the same virtual cycle $[\scrM]^{\red}$.

\begin{lemma}\label{lem:cosection-descent}
The induced morphism $H^1(\EE^{\red}_{{\fH}/{\fM}}) \to H^1(\EE^{\red}_{{\fH}})$ is an isomorphism of $\cO_{{\fH}}$.
\end{lemma}
\begin{proof}
Since $\fM$ is smooth, we have $H^1(\TT_{\fM}) = 0$. Consider the induced morphism between long exact sequences
\[
\xymatrix{
H^{0}(\TT_{{\fH}}) \ar[r] \ar[d]^{\cong} & H^{0}(\TT_{\fM|_{{\fH}}}) \ar[r] \ar[d]^{\cong} & H^{1}(\TT_{{\fH}/\fM}) \ar[r] \ar[d] & H^{1}(\TT_{{\fH}}) \ar[r] \ar[d] & 0 \\
H^{0}(\EE^{\red}_{{\fH}}) \ar[r] & H^{0}(\TT_{\fM}|_{{\fH}}) \ar[r] & H^{1}(\EE^{\red}_{{\fH}/\fM}) \ar[r] & H^{1}(\EE^{\red}_{{\fH}}) \ar[r] & 0
}
\]
Since ${\fH} \to \fM$ is smooth, the two horizontal arrows on the left are both surjective. Thus $H^1(\EE^{\red}_{{\fH}/\fM}) \to H^1(\EE^{\red}_{{\fH}})$ is an isomorphism. By Theorem \ref{thm:reduction} (2), we have $H^1(\EE^{\red}_{{\fH}}) \cong \cO_{{\fH}}$.
\end{proof}

By Theorem \ref{thm:reduction}, we obtain a morphism of triangles
\[
\xymatrix{
\EE^{\red}_{\scrM/\fM} \ar[r] \ar[d]_{{\sigma^{\bullet,\red}_{\fM}}}  & \EE^{\red}_{\scrM} \ar[r] \ar[d]_{\sigma^{\bullet,\red}} & \TT_{\scrM} \ar[r]^{[1]}  \ar[d] & \\
\EE^{\red}_{{\fH}/{\fM}}|_{\scrM} \ar[r]  & \EE^{\red}_{\fH}|_{\scrM} \ar[r]  & \TT_{\fM}|_{\scrM} \ar[r]^{[1]}  &
}
\]
Taking $H^1$ and applying Lemma \ref{lem:cosection-descent}, we have a commutative diagram
\begin{equation}\label{diag:cosections}
\xymatrix{
H^1(\EE^{\red}_{\scrM/\fM}) \ar@{->>}[r] \ar[d]_{\sigma^{\red}_{\fM}} & H^1(\EE^{\red}_{\scrM}) \ar[d]^{\sigma^{\red}} \\
\cO_{\scrM} \ar[r]^{=} &  \cO_{\scrM}.
}
\end{equation}

Denote by $\scrM(\sigma^{\red}) \subset \scrM$ the closed substack along which the cosection $\sigma^{\red}$ degenerates, and write $\iota\colon \scrM(\sigma^{\red}) \hookrightarrow \scrM$ for the closed embedding. Let $[\scrM]_{\sigma^{\red}}$ be the cosection localized virtual cycle as in \cite{KiLi13}. We conclude that

\begin{theorem}\label{thm:generali-localized-cycle}
With the assumptions in Section \ref{ss:reduction-set-up} and further assuming that $\fM$ is smooth, we have
\begin{enumerate}
 \item The cosection $\sigma^{\red}$ is surjective along $\Delta_{\scrM}$.
 \item $\iota_*[\scrM]_{\sigma^{\red}} = [\scrM]^{\red}$.
\end{enumerate}
\end{theorem}
\begin{proof}
(1) follows from the surjectivity of $\sigma^{\red}_{\fM}$ along $\Delta_{\scrM}$ and \eqref{diag:cosections}. (2) follows from \cite[Theorem 1.1]{KiLi13}.
\end{proof}

\subsection{Proof of Theorem \ref{thm:reduction}}
\label{ss:proof-reduction}
By \eqref{diag:compatible-POT}, we obtain a commutative diagram of solid arrows
\begin{equation}\label{diag:compare-POTs}
\xymatrix{
\TT_{\scrM/\fM} \ar@/^1pc/[rrd] \ar@{-->}[rd]_{\varphi^{\red}_{\scrM/\fM}}  \ar[dd] &&&& \\
 & \EE^{\red}_{\scrM/\fM} \ar[dd] \ar[r] & \EE_{\scrM/\fM} \ar[dd] \ar[r] & \FF|_{\scrM} \ar[r]^{[1]} \ar@{=}[dd] & \\
\TT_{\fH/\fM}|_{\scrM} \ar@/^1pc/[rrd]|{\ \ \ \hole \ \   } \ar@{-->}[rd]_{\varphi^{\red}_{\fH/\fM}|_{\scrM}} &&&& \\
 & \EE^{\red}_{\fH/\fM}|_{\scrM} \ar[r] & \EE_{\fH/\fM}|_{\scrM} \ar[r] & \FF|_{\scrM} \ar[r]^{[1]} &
}
\end{equation}
where the two horizontal lines are given by \eqref{equ:reduced-POT}, and the two solid curved arrows are given by \eqref{equ:canonical-POT}.

Since $\fH \to \fM$ is smooth of Deligne--Mumford type, $\TT_{\fH/\fM}$ is the relative tangent bundle $T_{\fH/\fM}$. Thus the composition $\TT_{\fH/\fM} \to \EE_{\fH/\fM} \to \FF|_{\fH}$ is trivial, which leads to the desired arrow $\varphi^{\red}_{\fH/\fM}$.

Similarly, the composition $\TT_{\scrM/\fM} \to \EE_{\scrM/\fM} \to \FF|_{\scrM}$ factors through $\TT_{\fH/\fM}|_{\scrM} \to \FF|_{\scrM}$ hence is also trivial, which leads to $\varphi^{\red}_{\scrM/\fM}$. This proves the factorization part in (1), and the commutative diagram in (2).

For the perfect obstruction theories part, observe that $\EE^{\red}_{\fH/\fM}$ and $\EE^{\red}_{\scrM/\fM}$ are at least perfect in $[0,2]$ as $\FF$ is perfect in $[0,1]$. It remains to show that $H^2(\EE^{\red}_{\fH/\fM}) = 0$ and $H^{2}(\EE^{\red}_{\scrM/\fM})=0$. Taking the long exact sequence of the first triangle in \eqref{equ:reduced-POT}, we have
\[
H^1(\EE_{\fH/\fM}) \to H^1(\FF|_{\fH}) \to H^2(\EE^{\red}_{\fH/\fM}) \to 0.
\]
Since the left arrow is precisely $\cO_{\fM}(\Delta) \twoheadrightarrow \cok \epsilon$, we obtain $H^2(\EE^{\red}_{\fH/\fM}) = 0$.

Similarly, we have the long exact sequence
\[
H^1(\EE_{\scrM/\fM}) \to H^1(\FF|_{\scrM}) \to H^2(\EE^{\red}_{\scrM/\fM}) \to 0,
\]
where the left arrow is given by the composition
\[
H^1(\EE_{\scrM/\fM}) \stackrel{\sigma_{\fM}}{\to} H^1(\EE_{\fH/\fM}|_{\scrM}) \twoheadrightarrow H^1(\FF|_{\scrM}).
\]
Since $\FF|_{\scrM\setminus \Delta_{\scrM}} = 0$ and $\sigma_{\fM}$ is surjective along $\Delta_{\scrM}$, the above composition is surjective, hence $H^2(\EE^{\red}_{\scrM/\fM}) = 0$.

We next verify that $\varphi^{\red}_{\scrM/\fM}$ and $\varphi^{\red}_{\fH/\fM}$ are obstruction theories. Indeed, the factorization of  (1) implies a surjection $H^0(\TT_{\scrM/\fM}) \twoheadrightarrow H^0(\EE^{\red}_{\scrM/\fM})$ and an injection $H^1(\TT_{\scrM/\fM}) \hookrightarrow H^1(\EE^{\red}_{\scrM/\fM})$. Since $\FF|_{\scrM}$ is perfect in $[0,1]$, $H^0(\TT_{\scrM/\fM}) \twoheadrightarrow H^0(\EE^{\red}_{\scrM/\fM})$ is an injection, hence an isomorphism.  The case that $\varphi^{\red}_{\fH/\fM}$ is an obstruction theory can be proved similarly. This completes the proof of (1).

Observe that $H^0(\FF|_{\fH}) = 0$ since $\fH \to \fM$ is smooth. The first triangle in \eqref{equ:reduced-POT} implies an exact sequence
\[
0 \to H^1(\EE^{\red}_{\fH/\fM}) \to H^1(\EE_{\fH/\fM}) \to H^1(\FF|_{\fH}) \to 0.
\]
Using \eqref{equ:general-cosection} and the construction of \eqref{equ:reduction-map}, we obtain $H^1(\EE^{\red}_{\fH/\fM}) \cong \cO_{\fH}$.

Now \eqref{diag:compare-POTs} induces a morphism of long exact sequences
 \[
 \xymatrix{
 0 \ar[r] & H^0(\FF|_{\scrM}) \ar[r] \ar[d]^{\cong} & H^1(\EE^{\red}_{\scrM/\fM}) \ar[r] \ar[d]^{\sigma^{\red}_{\fM}} & H^1(\EE_{\scrM/\fM}) \ar[r] \ar[d]^{\sigma_{\fM}} & H^1(\FF|_{\scrM}) \ar[r]  \ar[d]^{\cong}& 0 \\
 0 \ar[r] & H^0(\FF|_{\scrM}) \ar[r]  & H^1(\EE^{\red}_{\fH/\fM}) \ar[r]  & H^1(\EE_{\fH/\fM}) \ar[r]  & H^1(\FF|_{\scrM}) \ar[r]  & 0
 }
 \]
The surjectivity of $\sigma^{\red}_{\fM}$ along $\Delta_{\scrM}$ follows from the surjectivity of $\sigma_{\fM}$  along $\Delta_{\scrM}$. This finishes the proof of (2).

\subsection{The reduced boundary cycle}\label{ss:boundary-cycle}
The pull-backs
\[
\EE_{\Delta_{\scrM}/\Delta} := \EE_{\scrM/\fM}|_{\Delta_{\scrM}} \ \ \ \mbox{and} \ \ \ \EE_{\Delta_{\fH}/\Delta} := \EE_{\fH/\fM}|_{\Delta_{\fH}}.
\]
define perfect obstruction theories of $\Delta_{\scrM} \to \Delta$ and $\Delta_{\fH} \to \Delta$ respectively. Consider the sequence of morphisms
\[
\EE_{\Delta_{\scrM}/\Delta} \to \EE_{\Delta_{\fH}/\Delta}|_{\Delta_{\scrM}} \to H^{1}(\EE_{\fH/\fM}|_{\Delta_{\scrM}})[-1] \to H^1(\FF)|_{\Delta_{\scrM}}[-1]
\]
where the last arrow is given by \eqref{equ:reduction-map}. Since
\begin{equation}\label{equ:F1}
H^1(\FF|_{\Delta}) = \cO_{\Delta}(\Delta),
\end{equation}
we obtain a triangle
\begin{equation}\label{equ:boundary-red-POT}
\EE^{\red}_{\Delta_{\scrM}/\Delta} \to \EE_{\Delta_{\scrM}/\Delta} \to \cO_{\Delta}(\Delta)|_{\Delta_{\scrM}}[-1] \stackrel{[1]}{\to}
\end{equation}

The two virtual cycles $[\scrM]^{\vir}$ and $[\scrM]^{\red}$ are related as follows.

\begin{theorem}\label{thm:boundary-cycle}
Notations and assumptions as above, we have
\begin{enumerate}
 \item There is a canonical factorization of perfect obstruction theories
  \[
 \xymatrix{
 \TT_{\Delta_{\scrM}/\Delta} \ar[rr]^{\varphi_{\Delta_{\scrM}/\Delta}} \ar[rd]_{\varphi^{\red}_{\Delta_{\scrM}/\Delta}} && \EE_{\Delta_{\scrM}/\Delta} \\
 &\EE^{\red}_{\Delta_{\scrM}/\Delta} \ar[ru]&
 }
 \]
 Denote by $[\Delta_{\scrM}]^{\red}$ the virtual cycle associated to $\varphi^{\red}_{\Delta_{\scrM}/\Delta}$, called the \emph{reduced boundary cycle}.

 \item Suppose $\fM$ is smooth. Then we have a relation of virtual cycles
 \[
 [\scrM]^{\vir} = [\scrM]^{\red} + i_*[\Delta_{\scrM}]^{\red}
 \]
 where $i \colon \Delta_{\scrM} \to \scrM$ is the natural embedding.
\end{enumerate}
\end{theorem}
\begin{proof}
The proof of Theorem \ref{thm:boundary-cycle} (1) is similar to Theorem \ref{thm:reduction} (1), and will be omitted. We next consider (2).

Recall that $\scrM(\sigma^{\red}) \subset \scrM$ is the locus where
$\sigma^{\red}$ hence $\sigma_{\fM}$ degenerates.
Replacing $\scrM$ by $\scrM\setminus \scrM(\sigma^{\red})$ we may
assume that $\sigma_{\fM}$ is everywhere surjective.
Since the cosection localized virtual cycle $[\scrM]_{\sigma^{\red}}$ is
represented by a Chow cycle supported on $\scrM(\sigma^{\red})$, which
is empty by our assumption, we see that
$[\scrM]_{\sigma^{\red}} = 0$.
By Theorem~\ref{thm:generali-localized-cycle}, it remains to show that
\begin{equation}\label{equ:canonical=boundary}
[\scrM] = i_*[\Delta_{\scrM}]^{\red}.
\end{equation}

To proceed, we consider the triangle
\begin{equation}\label{equ:t-red-POT}
\EE^{\tred}_{\scrM/\fM} \to \EE_{\scrM/\fM} \to \cO_{\fM}(\Delta)|_{\scrM}[-1] \stackrel{[1]}{\to}
\end{equation}
where the middle arrow is given by \eqref{equ:reduction-map} and \eqref{equ:F1}. Similar to the case of (1), we obtain a factorization of perfect obstruction theories
\[
 \xymatrix{
 \TT_{\scrM/\fM} \ar[rr]^{\varphi_{\scrM/\fM}} \ar[rd]_{\varphi^{\tred}_{\scrM/\fM}} && \EE_{\scrM/\fM} \\
 &\EE^{\tred}_{\scrM/\fM} \ar[ru]&
 }
\]

Let $[\scrM]^{\tred}$ be the virtual cycle corresponding to the perfect obstruction theory $\varphi^{\tred}_{\scrM/\fM}$. We call $[\scrM]^{\tred}$ the \emph{totally reduced virtual cycle} to be distinguished from $[\scrM]^{\red}$. Comparing \eqref{equ:boundary-red-POT} and \eqref{equ:t-red-POT}, we have
$
\varphi^{tred}_{\scrM/\fM}|_{\Delta_{\scrM}} = \varphi^{\red}_{\Delta_{\scrM}/\Delta},
$
hence
\[
i^{!} [\scrM]^{tred}= [\Delta_{\scrM}]^{\red}.
\]

Since $\fM$ is smooth, as in \eqref{diag:relative-to-absolute}, we may
construct absolute perfect obstruction theories associated to
$\varphi_{\scrM/\fM}$ and $\varphi^{\tred}_{\scrM/\fM}$ respectively:
\[
\varphi_{\scrM}\colon \TT_{\scrM/\fM} \to \EE_{\scrM} \ \ \ \mbox{and} \ \ \ \varphi^{\tred}_{\scrM}\colon \TT_{\scrM/\fM} \to \EE^{\tred}_{\scrM}.
\]

By the same construction as in
Section~\ref{ss:general-absolut-theory}, the cosection $\sigma_{\fM}$
descends to an absolute cosection
$\sigma\colon H^1(\EE_{\scrM}) \to \cO_{\fM}(\Delta)|_{\scrM}$ which
is everywhere surjective.
Let $\fE_{\scrM}$ and $\fE^{tred}_{\scrM}$ be the vector bundle stacks
of $\EE_{\scrM}$ and $\EE^{\tred}_{\scrM}$ respectively.
Then $\fE^{tred}_{\scrM}$ is the kernel cone stack of
$\fE_{\scrM} \to \cO_{\fM}(\Delta)|_{\scrM}$ induced by $\sigma$.
Let $\fC_{\scrM}$ be the intrinsic normal cone of $\scrM$.
Unwinding the definition of cosection localized virtual cycle in
\cite[Definition 3.2]{KiLi13}, we have
\begin{equation}\label{equ:tred=bred}
[\scrM]_\sigma = i^{!} 0^!_{\fE^{\tred}_{\scrM}}[\fC_{\scrM}] = i^![\scrM]^{\tred} = [\Delta_{\scrM}]^{\red}.
\end{equation}
where $[\scrM]_\sigma$ is the cosection localized virtual cycle
corresponding to $\sigma$.
Finally, \eqref{equ:canonical=boundary} follows from
$i_*[\scrM]_\sigma = [\scrM]^{\vir}$, see \cite[Theorem~1.1]{KiLi13}.
\end{proof}


\end{document}